\documentclass[english,12pt]{article}
\usepackage{xcolor}
\definecolor{note_fontcolor}{rgb}{0.800781, 0.800781, 0.800781}
\usepackage{babel}
\usepackage{mathtools}
\usepackage{amsmath}
\usepackage{amssymb}
\usepackage{amsthm}
\usepackage%[pdfusetitle, bookmarks=true,bookmarksnumbered=false,bookmarksopen=false, breaklinks=false,pdfborder={0 0 0},pdfborderstyle={},backref=false,colorlinks=false]
        {hyperref}

\usepackage{tikz-cd}
\usepackage{tikz}
                      \usepackage{tikzit}
\tikzstyle{nc}=[fill=white, draw=black, shape=circle, minimum size=0.40cm]
\tikzstyle{nr}=[fill=white, draw=black, shape=rectangle, minimum width=0.60cm, minimum height=0.60cm]
\tikzstyle{nrg}=[fill={rgb,255: red,199; green,199; blue,199}, draw=black, shape=rectangle]
\tikzstyle{sr}=[fill=white, draw=black, shape=rectangle, minimum width=0.60cm, minimum height=0.40cm, inner sep=0pt]
\tikzstyle{lsr}=[fill=white, draw=black, shape=rectangle, minimum width=0.60cm, minimum height=0.40cm, inner sep=0pt, rotate=30]
\tikzstyle{rsr}=[fill=white, draw=black, shape=rectangle, minimum width=0.60cm, minimum height=0.40cm, inner sep=0pt, rotate=-30]
\tikzstyle{rrrsr}=[fill=white, draw=black, shape=rectangle, minimum width=0.60cm, minimum height=0.40cm, inner sep=0pt, rotate=-120]
\tikzstyle{ivsr}=[fill=none, draw=none, shape=rectangle, minimum width=0.60cm, minimum height=0.40cm, inner sep=0pt]
\tikzstyle{ivlsr}=[fill=none, draw=none, shape=rectangle, minimum width=0.60cm, minimum height=0.40cm, inner sep=0pt, rotate=30]
\tikzstyle{ivrsr}=[fill=none, draw=none, shape=rectangle, minimum width=0.60cm, minimum height=0.40cm, inner sep=0pt, rotate=-30]
\tikzstyle{ivrrrsr}=[fill=none, draw=none, shape=rectangle, minimum width=0.60cm, minimum height=0.40cm, inner sep=0pt, rotate=-120]
\tikzstyle{mr}=[fill=white, draw=black, shape=rectangle, minimum width=1.00cm, minimum height=0.50cm, inner sep=0pt]
\tikzstyle{product}=[fill={rgb,255: red,255; green,5; blue,80}, draw=black, shape=circle, inner sep=0pt, minimum size=0.15cm]
\tikzstyle{unit}=[fill={rgb,255: red,255; green,166; blue,217}, draw=black, shape=circle, inner sep=0pt, minimum size=0.15cm]
\tikzstyle{coproduct}=[fill={rgb,255: red,2; green,145; blue,255}, draw=black, shape=circle, inner sep=0pt, minimum size=0.15cm]
\tikzstyle{counit}=[fill={rgb,255: red,165; green,219; blue,255}, draw=black, shape=circle, inner sep=0pt, minimum size=0.15cm]
\tikzstyle{antipode}=[fill=white, draw=black, shape=circle, inner sep=0pt, minimum size=0.15cm]
\tikzstyle{region}=[fill={rgb,255: red,191; green,191; blue,191}, draw={rgb,255: red,191; green,191; blue,191}, shape=circle, minimum size=0.80cm]
\tikzstyle{boundarydisc}=[fill={rgb,255: red,128; green,128; blue,128}, draw=black, shape=circle, minimum size=0.80cm]
\tikzstyle{identity}=[fill=black, draw=black, shape=circle, inner sep=0pt, minimum size=0.15cm]
\tikzstyle{S}=[fill=white, draw=black, shape=rectangle, minimum width=0.15cm, minimum height=0.15cm]
\tikzstyle{S-}=[fill={rgb,255: red,199; green,199; blue,199}, draw=black, shape=rectangle, minimum width=0.15cm, minimum height=0.15cm]
\tikzstyle{dot}=[fill=white, draw=black, shape=circle, inner sep=0pt, minimum size=0.07cm]
\tikzstyle{bdot}=[fill=black, draw=black, shape=circle, inner sep=0pt, minimum size=0.07cm]
\tikzstyle{m}=[fill={rgb,255: red,255; green,5; blue,80}, draw={rgb,255: red,255; green,5; blue,80}, shape=circle, inner sep=0pt, minimum size=0.14cm, tikzit shape=circle]
\tikzstyle{n}=[fill={rgb,255: red,30; green,127; blue,255}, draw={rgb,255: red,30; green,127; blue,255}, shape=circle, inner sep=0pt, minimum size=0.14cm, tikzit shape=circle]
\tikzstyle{mod}=[fill={rgb,255: red,172; green,229; blue,232}, draw={rgb,255: red,14; green,115; blue,177}, shape=circle]
\tikzstyle{smod}=[fill={rgb,255: red,172; green,229; blue,232}, draw={rgb,255: red,14; green,115; blue,177}, shape=circle, inner sep=0pt, minimum size=0.15cm]
\tikzstyle{smod1}=[fill={rgb,255: red,150; green,255; blue,138}, draw={rgb,255: red,5; green,130; blue,3}, shape=circle, inner sep=0pt, minimum size=0.15cm]
\tikzstyle{smod2}=[fill={rgb,255: red,255; green,202; blue,239}, draw={rgb,255: red,126; green,0; blue,124}, shape=circle, inner sep=0pt, minimum size=0.15cm]
\tikzstyle{smod3}=[fill={rgb,255: red,255; green,213; blue,164}, draw={rgb,255: red,121; green,57; blue,2}, shape=circle, inner sep=0pt, minimum size=0.15cm]
\tikzstyle{mid-nc}=[fill=white, draw=black, shape=circle, inner sep=0pt, minimum size=0.25cm]
\tikzstyle{l70-none}=[fill=none, draw=none, shape=circle, rotate=70]
\tikzstyle{r45}=[fill=none, draw=none, shape=circle, rotate=-45]

\tikzstyle{di}=[draw=black, ->]
\tikzstyle{da}=[-, dashed]
\tikzstyle{da-di}=[dashed, ->]
\tikzstyle{th}=[-, thick]
\tikzstyle{th-di}=[->, thick]
\tikzstyle{th-da}=[-, dashed, thick]
\tikzstyle{gr}=[-, draw={rgb,255: red,0; green,186; blue,143}, thick]
\tikzstyle{gr-di}=[draw={rgb,255: red,0; green,186; blue,143}, ->, thick]
\tikzstyle{gr-da}=[-, dashed, draw={rgb,255: red,0; green,186; blue,143}, thick]
\tikzstyle{pu}=[-, draw={rgb,255: red,124; green,95; blue,239}, thick]
\tikzstyle{pu-da}=[-, dashed, draw={rgb,255: red,124; green,95; blue,239}, thick]
\tikzstyle{re}=[-, draw={rgb,255: red,255; green,5; blue,80}, thick]
\tikzstyle{re-di}=[->, draw={rgb,255: red,255; green,5; blue,80}, thick]
\tikzstyle{re-da}=[-, draw={rgb,255: red,255; green,5; blue,80}, dashed, thick]
\tikzstyle{bl}=[-, draw={rgb,255: red,14; green,115; blue,177}, thick]
\tikzstyle{bl-di}=[draw={rgb,255: red,14; green,115; blue,177}, ->, thick]
\tikzstyle{bl-da}=[-, draw={rgb,255: red,14; green,115; blue,177}, dashed, thick]
\tikzstyle{lgr}=[-, draw={rgb,255: red,188; green,220; blue,156}, thick]
\tikzstyle{lgr-da}=[-, draw={rgb,255: red,188; green,220; blue,156}, thick, dashed]
\tikzstyle{lbl}=[-, draw={rgb,255: red,168; green,207; blue,245}, thick]
\tikzstyle{lbl-da}=[-, draw={rgb,255: red,168; green,207; blue,245}, thick, dashed]
\tikzstyle{pi}=[-, draw={rgb,255: red,208; green,160; blue,160}, thick]
\tikzstyle{pi-da}=[-, draw={rgb,255: red,208; green,160; blue,160}, thick, dashed]
\tikzstyle{sh}=[-, draw={rgb,255: red,171; green,171; blue,171}]
\tikzstyle{sh-da}=[-, draw={rgb,255: red,171; green,171; blue,171}, dashed]
\tikzstyle{fi-gr}=[-, fill={rgb,255: red,194; green,228; blue,162}, draw={rgb,255: red,194; green,228; blue,162}]
\tikzstyle{fi-bl}=[-, fill={rgb,255: red,175; green,215; blue,255}, draw={rgb,255: red,175; green,215; blue,255}]
\tikzstyle{fi-pi}=[-, fill={rgb,255: red,246; green,208; blue,208}, draw={rgb,255: red,246; green,208; blue,208}]
\tikzstyle{fi-pu}=[-, fill={rgb,255: red,230; green,203; blue,246}, draw={rgb,255: red,230; green,203; blue,246}]
\tikzstyle{fi-ye}=[-, fill={rgb,255: red,255; green,255; blue,155}, draw={rgb,255: red,255; green,255; blue,155}]
\tikzstyle{fi-or}=[-, fill={rgb,255: red,255; green,197; blue,51}, draw={rgb,255: red,255; green,197; blue,51}]
\tikzstyle{fi-dg}=[-, fill={rgb,255: red,152; green,170; blue,139}, draw={rgb,255: red,152; green,170; blue,139}]
\tikzstyle{fi-db}=[-, fill={rgb,255: red,111; green,147; blue,183}, draw={rgb,255: red,111; green,147; blue,183}]
\tikzstyle{fi-sh}=[-, fill={rgb,255: red,232; green,232; blue,232}, draw={rgb,255: red,232; green,232; blue,232}]
\tikzstyle{fi-dsh}=[-, fill={rgb,255: red,204; green,204; blue,204}, draw={rgb,255: red,204; green,204; blue,204}]
\tikzstyle{fi-lbl}=[-, fill={rgb,255: red,215; green,235; blue,255}, draw={rgb,255: red,215; green,235; blue,255}]
\tikzstyle{fi-lgr}=[-, fill={rgb,255: red,232; green,247; blue,216}, draw={rgb,255: red,232; green,247; blue,216}]
\tikzstyle{fi-lpu}=[-, fill={rgb,255: red,245; green,230; blue,255}, draw={rgb,255: red,245; green,230; blue,255}]
\tikzstyle{fi-lpi}=[-, fill={rgb,255: red,249; green,226; blue,226}, draw={rgb,255: red,249; green,226; blue,226}]
\tikzstyle{fi-vlbl}=[-, fill={rgb,255: red,229; green,251; blue,255}, draw={rgb,255: red,229; green,251; blue,255}]
\tikzstyle{vt}=[-, very thick]
\tikzstyle{vt-di}=[->, very thick]
\tikzstyle{vt-da}=[-, dashed, very thick]
\tikzstyle{vt-gr}=[-, draw={rgb,255: red,0; green,186; blue,143}, very thick]
\tikzstyle{vt-gr-di}=[draw={rgb,255: red,0; green,186; blue,143}, ->, very thick]
\tikzstyle{vt-gr-da}=[-, dashed, draw={rgb,255: red,0; green,186; blue,143}, very thick]
\tikzstyle{vt-pu}=[-, draw={rgb,255: red,124; green,95; blue,239}, very thick]
\tikzstyle{vt-pu-da}=[-, dashed, draw={rgb,255: red,124; green,95; blue,239}, very thick]
\tikzstyle{vt-re}=[-, draw={rgb,255: red,255; green,5; blue,80}, very thick]
\tikzstyle{vt-re-di}=[->, draw={rgb,255: red,255; green,5; blue,80}, very thick]
\tikzstyle{vt-re-da}=[-, draw={rgb,255: red,255; green,5; blue,80}, dashed, very thick]
\tikzstyle{vt-bl}=[-, draw={rgb,255: red,14; green,115; blue,177}, very thick]
\tikzstyle{vt-bl-di}=[draw={rgb,255: red,14; green,115; blue,177}, ->, very thick]
\tikzstyle{vt-bl-da}=[-, draw={rgb,255: red,14; green,115; blue,177}, dashed, very thick]
\tikzstyle{vt-bl-da-di}=[draw={rgb,255: red,14; green,115; blue,177}, ->, very thick, dashed]
\tikzstyle{vt-lgr}=[-, draw={rgb,255: red,188; green,220; blue,156}, very thick]
\tikzstyle{vt-lgr-di}=[draw={rgb,255: red,188; green,220; blue,156}, ->, very thick]
\tikzstyle{vt-lgr-da}=[-, draw={rgb,255: red,188; green,220; blue,156}, very thick, dashed]
\tikzstyle{vt-lbl}=[-, draw={rgb,255: red,168; green,207; blue,245}, very thick]
\tikzstyle{vt-lbl-di}=[->, draw={rgb,255: red,168; green,207; blue,245}, very thick]
\tikzstyle{vt-lbl-da}=[-, draw={rgb,255: red,168; green,207; blue,245}, very thick, dashed]
\tikzstyle{vt-lpp}=[-, draw={rgb,255: red,223; green,199; blue,240}, very thick]
\tikzstyle{vt-vlbl}=[-, draw={rgb,255: red,226; green,236; blue,245}, very thick]
\tikzstyle{vt-vlpp}=[-, draw={rgb,255: red,237; green,231; blue,240}, very thick]
\tikzstyle{vt-vlgr}=[-, draw={rgb,255: red,240; green,255; blue,225}, very thick]
\tikzstyle{vt-pi}=[-, draw={rgb,255: red,236; green,151; blue,151}, very thick]
\tikzstyle{vt-pi-da}=[-, draw={rgb,255: red,236; green,151; blue,151}, very thick, dashed]
\tikzstyle{vt-ye}=[-, draw={rgb,255: red,241; green,241; blue,115}, very thick]
\tikzstyle{vt-or}=[-, draw={rgb,255: red,245; green,150; blue,32}, very thick]
\tikzstyle{vt-sh}=[-, fill=none, draw={rgb,255: red,171; green,171; blue,171}, very thick]
\tikzstyle{vt-sh-di}=[draw={rgb,255: red,171; green,171; blue,171}, very thick, ->]
\tikzstyle{vt-sh-da}=[-, fill=none, draw={rgb,255: red,171; green,171; blue,171}, very thick, dashed]
\tikzstyle{fi-bx}=[-, fill={rgb,255: red,172; green,229; blue,232}, draw={rgb,255: red,14; green,115; blue,177}, thick]
\tikzstyle{ut-gr}=[-, draw={rgb,255: red,0; green,186; blue,143}, line width=3pt]
\tikzstyle{ut-gr-di}=[draw={rgb,255: red,0; green,186; blue,143}, ->, line width=3pt]
\tikzstyle{ut-ig}=[-, draw={rgb,255: red,255; green,5; blue,80}, line width=3pt]
\tikzstyle{ut-ig-di}=[draw={rgb,255: red,255; green,5; blue,80}, line width=3pt, ->]
\tikzstyle{op-sh}=[-, draw=none, fill={rgb,255: red,113; green,113; blue,113}, opacity=0.5]
\tikzstyle{op-lsh}=[-, fill={rgb,255: red,173; green,173; blue,173}, draw=none, opacity=0.5]
\tikzstyle{op-pi}=[-, draw=none, fill={rgb,255: red,246; green,188; blue,188}, opacity=0.7]
\tikzstyle{yt-fill}=[-, fill=white]

\usepackage{quiver}
\everymath{\displaystyle}
\usepackage{lmodern}
\usepackage{fancyhdr}
\usepackage{bbm}
\usepackage{graphicx}

	     \newcommand\Cite[2] {\cite[#1]{#2}}

\makeatletter

\numberwithin{equation}{section}
\numberwithin{figure}{section}

\theoremstyle{definition}
\newtheorem{thm}{\protect\theoremname}[section]
\newtheorem*{thmI}{\protect\theoremname}
\newtheorem{rem}[thm]{\protect\remarkname}
\newtheorem{cor}[thm]{\protect\corollaryname}
\newtheorem{example}[thm]{\protect\examplename}
\newtheorem{defn}[thm]{\protect\definitionname}
\newtheorem{lem}[thm]{\protect\lemmaname}
\newtheorem{prop}[thm]{\protect\propositionname}

\@ifundefined{date}{}{\date{}}
\makeatletter
\newcommand{\pgets}{}
\DeclareRobustCommand{\pgets}{\mathrel{\mathpalette\p@to@gets\gets}}
\newcommand{\p@to@gets}[2]{%
  \ooalign{\hidewidth$\m@th#1\mapstochar\mkern5mu$\hidewidth\cr$\m@th#1\to$\cr}%
}
\newcommand{\Pgets}{}
\DeclareRobustCommand{\Pgets}{\mathrel{\mathpalette\P@to@gets\gets}}
\newcommand{\P@to@gets}[2]{%
\ooalign{\hidewidth$\m@th#1\mapstochar\mkern3mu$\hidewidth\cr$\m@th#1\longrightarrow$\cr}%
}
\makeatother

\usetikzlibrary{decorations,arrows}
\usetikzlibrary{decorations.markings}
\usetikzlibrary{decorations.pathmorphing}
\usetikzlibrary{matrix,arrows}

\usepackage{microtype}
\DeclareOldFontCommand{\it}{\normalfont\itshape}{\mathit}

\AtBeginDocument{%
   \def\MR#1{}
}

\makeatother

\providecommand{\corollaryname}{Corollary}
\providecommand{\definitionname}{Definition}
\providecommand{\examplename}{Example}
\providecommand{\lemmaname}{Lemma}
\providecommand{\propositionname}{Proposition}
\providecommand{\remarkname}{Remark}
\providecommand{\theoremname}{Theorem}

\newcommand\void[1]  {}

\def\bbord         {\mathcal B\mathrm{ord}_{2}^{\mathrm{or,o/c}}}

\def\be            {\begin{equation}}
\def\bimod         {\text{-}\mathrm{mod}\text{-}}
\def\bl            {\mathrm{Bl}}
\def\bprof         {\mathcal{P}\mathrm{rof}_{\mathbbm{k}}}

\def\cc            {{\mathcal C}}
\def\cir           {\,{\circ}\,}
\def\ca            {\mathcal{A}}

\def\cat           {\mathrm{Cat}}
\def\cb            {\mathcal{B}}
\def\cbc           {\mathcal{BC}}
\def\ccat          {\mathcal{C}\mathrm{at}}
\def\ch            {\mathcal{H}}
\def\cfrc          {\mathcal{F}r(\mathcal{C})}

\def\colax         {colax} % rather than oplax
\def\Colon         {:\quad}
\def\Coloneqq      {\,{\coloneqq}\,}

\def\cv            {\mathcal{V}}
\def\czc           {\mathcal{Z}(\mathcal{C})}
\def\da            {\mathbb{A}}

\def\Dashv         {\,{\dashv}\,}
\def\db            {\mathbb{B}}
\def\dbl           {\mathcal{D}\mathrm{bl}}
\def\dbord         {\mathbb{OC}\mathrm{Bord}_{2}^{\mathrm{or}}}
\def\dd            {\mathbb{D}}
\def\df            {\mathrm{F}}
\def\di            {\mathbb{I}}   %% interval as object \ell in bord
\def\Di            {I}       %% interval in \ell \times I

\def\disto         {\ast}    %% distinguished object 
\def\dprof         {\mathbb{P}\mathrm{rof}_{\mathbbm{k}}}
\def\ds            {\mathbb{S}}
\def\dSN           {\mathbb{S}\mathrm{N}}
\def\dSNb          {\mathbb{S}\mathrm{N}_{\cb}}

\def\dSNc          {\mathbb{S}\mathrm{N}_{\cc}}

\def\dSNfrc        {\ds\mathrm{N}_{\cfrc}}
\def\ducorc        {\mathbb{U}\mathrm{Cor}_{\cc}}
\def\ee            {\end{equation}}
\def\eq            {\,{=}\,}
\def\eqv           {\Rarr{\,\simeq\,}}

\def\grph          {\varGamma}
\def\hfrc          {1\text-\Hom_{\cfrc}}
\def\Hom           {\mathrm{Hom}}
\def\id            {\mathrm{id}}
\def\iN            {\,{\in}\,}
\def\iso           {\rarr{\,\cong\,}}

\def\Itemize       {\def\leftmargini{1.71em}~\\[-1.55em]\begin{itemize}\setlength{\itemsep}{-0.25ex}}
\def\kar           {\mathrm{Kar}}
\def\ko            {{\ensuremath{\mathbbm k}}}
\def\mcg           {\mathrm{Map}}
\def\nat           {\mathrm{Nat}}

\def\op            {\mathbin{\mathrm{op}}}

\def\otA           {\,{\otimes_{\!A}}\,}
\def\otB           {\,{\otimes_{\!B}}\,}
\def\oti           {\,{\otimes}\,}

\def\pto           {\pgets}
\def\Pto           {\Pgets}
\def\ptO           {\,{\pto}\,}
\def\PtO           {\,{\Pto}\,}
\newcommand\rarr[1]{\xrightarrow{~#1~}}
\newcommand\Rarr[1]{\,{\xrightarrow{\,#1\,}}\,}

\newcommand\RarR[1]{\,{\xRightarrow{\,#1\,}}\,}

\def\SN            {\mathrm{SN}}

\def\SNb           {\mathrm{SN}_{\cb}}
\def\snb           {\mathrm{sn}_{\cb}^{}}
\def\SNc           {\SN_{\cc}}
\def\snc           {\mathrm{sn}_{\cc}}
\def\SNfrc         {\mathrm{SN}_{\cfrc}}
\def\snsphmod      {\mathrm{sn}_{\mathcal M\hspace*{-0.6pt}od\text-\cc}}
\def\snfrc         {\mathrm{sn}_{\cfrc}^{}}
\def\sphmod        {\mathcal M\hspace*{-1.1pt}od\text-\cc}
\def\surf          {\varSigma}
\def\surfhat       {\widehat\varSigma}
\def\Times         {\,{\times}\,}

\def\twodim        {two-di\-men\-si\-onal}
\def\tu            {\mathbbm{1}}
\def\ucorc         {\mathrm{UCor}_{\cc}}
\def\vct           {\mathrm{Vect}_\ko}

\def\worldsheet    {world sheet}

\makeatletter
\newcommand{\Rrightarrowfill@}{\arrowfill@\equiv\equiv\Rrightarrow}
\newcommand\xRrightarrow[1][]{\ext@arrow 0359\Rrightarrowfill@{#1}{}}
\makeatother

\definecolor{DarkViolet} {rgb}{0.580392,0.000000,0.827450}
\definecolor{ForestGreen}{rgb}{0.133333,0.545098,0.133333}
\definecolor{red3}       {rgb}{0.803921,0.000000,0.000000}

\begin{document}

\begin{center}

\vskip 11mm

{\Large \textbf{Modular functors and CFT correlators \\[8pt] via double categories}}

\vskip 18mm

{\large 
J\"urgen Fuchs~\textsuperscript{$a$},$\quad$
Christoph Schweigert~\textsuperscript{$b$}$\quad$and$\quad$
Yang Yang~\textsuperscript{$c$}
}

\vskip 10mm

\it$^a$  Teoretisk fysik, \ Karlstads Universitet
\\  
Universitetsgatan 21, \ S\,--\,651\,88\, Karlstad  
\\[9pt]
\it$^b$  Fachbereich Mathematik, \ Universit\"at Hamburg
\\  
Bereich Algebra und Zahlentheorie
\\  
Bundesstra\ss e 55, \ D\,--\,20\,146\, Hamburg  
\\[9pt] 
\it$^c$ Zentrum Mathematik, \ Technische Universit\"at M\"unchen
\\  
Boltzmannstra\ss e 3, \ D\,--\,85\,748\, Garching 

\end{center}

   \vspace*{2.9cm}

\noindent{\sc Abstract}\\[3pt] 
We point out that double categories provide a natural setting for modular functors 
obtained by a bicategorical string-net construction: The source of the modular
functor -- which is now a double functor -- is a symmetric monoidal double category
of bordisms, with bordisms as horizontal morphisms and smooth embeddings of manifolds
as vertical morphisms. The target of the modular functor is a double category with
profunctors and functors as horizontal and vertical morphisms.
 \\
The correlators and field functors for a conformal field theory based on a pivotal 
monoidal category $\cc$ can then be understood in the unified setting of a vertical
transformation between the modular functors for two pointed pivotal bicategories,
the delooping of $\cc$ and the bicategory of $\Delta$-separable 
symmetric Frobenius algebras in $\cc$.
Using skein theoretic methods, we show that this vertical transformation
is an equivalence, which implies that field functors are equivalences of categories 
and that universal correlators are isomorphisms of vector spaces.

\newpage
\tableofcontents
\newpage

%%%%%%%%%%%%%%%%%%%%%%%%%%%%%%%%%%%%%%%%%%%%%%%%%%%%%%%%%%%%%%%%%%%%%%%%

\section{Introduction}

String-net models have recently been found to provide an efficient handle on the
study of correlators in two-dimensional conformal field theories 
\cite{fusY,yangYa3,fusY2}. Concretely, the following picture has emerged. Let $\cc$
be a pivotal linear tensor category. (Readers directly interested in conformal field
theory may wish to imagine $\cc$ as a category of representations of a suitably nice 
vertex operator algebra, ignoring the braiding.)
The monoidal category $\cc$ gives rise to two bicategories: On
the one hand, $\cbc$, a bicategory with a single object whose 1-morphisms are the
objects in $\cc$ and whose 2-morphisms are the morphisms in $\cc$. And on the other
hand, the bicategory $\cfrc$ of $\Delta$-se\-pa\-rable symmetric Frobenius algebras in
$\cc$, with bimodules as 1-morphisms. There is a forgetful functor
$u\colon \cfrc\Rarr~ \cbc$, which has the structure of a separable Frobenius functor.
Both $\cbc$ and $\cfrc$ are pivotal, and both
are pointed, i.e.\ endowed with a distinguished object. For $\cbc$ the latter is,
obviously, the single object, while for $\cfrc$ the distinguished object is the 
Frobenius algebra defined on the monoidal unit. 

In \cite{fusY2} we developed a bicategorical string-net construction, which
in particular yields a modular functor for each of the two bicategories $\cbc$ and
$\cfrc$, i.e.\ two symmetric monoidal functors
  \be
  \SNfrc, \SNc\Colon \bbord \rarr~ \bprof 
  \label{ref:twomod}
  \ee
from a symmetric monoidal bicategory of open-closed bordisms with values in a 
symmetric monoidal bicategory of profunctors. The string nets based on the bicategory
$\cbc$ describe 
conformal blocks and their monodromies. More specifically, they provide conformal
blocks for the Drinfeld center $\czc$ of $\cc$ which, for $\cc$ a modular tensor 
category, is equivalent, as a braided monoidal category, to the Deligne product of 
$\cc$ with $\overline{\cc}$, i.e.\ $\cc$ endowed with the opposite braiding.
The string nets for $\cfrc$ describe classes of decorations for \emph{world sheets},
that is, for surfaces that allow for the definition of conformal field theory 
correlators. The data of such a decoration include a stratification of the surface
and the specification of a full conformal field theory on 
each 2-cell of the surface, which amounts to assigning a $\Delta$-separable symmetric
Frobenius algebra, an object of $\cfrc$, to each 2-cell. There is more to the 
decoration of the surface: boundary conditions, in the form of modules over Frobenius
algebras, are assigned to free boundary segments, as well as defect conditions,
realized as bimodules over Frobenius algebras, on embedded line defects. The data 
for the world sheet also determine which fields -- boundary fields, bulk fields and,
more generally, defect fields -- are inserted on the surface. The corresponding
`insertion points' come with local coordinates and are thus really intervals (for
boundary fields) or circles (for bulk and defect fields).

 \medskip

Given the two modular functors \eqref{ref:twomod}, two pieces of data were 
constructed in \cite{fusY2}:
 \Itemize

 \item
Field objects: A modular functor associates a linear category
$\mathrm{SN}(\ell)$, called
cylinder category, to an oriented 1-manifold $\ell$. In our context we get two such
categories, $\SNfrc(\ell)$ and $\SNc(\ell) \,{\equiv}\, \mathrm{SN}_{\cbc}(\ell)$. 
The forgetful functor $u\colon \cfrc\Rarr~ \cbc$ provides us with a family of functors
  \be
  \df_{\!\ell}^{} \equiv \df_{\!\cc}(\ell) \Colon \SNfrc(\ell) \rarr~ \SNc(\ell)
  \label{eqI:fifu}
  \ee
parametrized by the objects $\ell$ in the bordism category $\bbord$, called the
\emph{field functors}.

A decorated world sheet assigns to each boundary interval or circle $\ell$
an object in the $\cfrc$-colored cylinder category $\SNfrc(\ell)$. The image of this
object under the field functor $\df_{\!\ell}$ is an object in $\cc$ if $\ell$ is
an interval (boundary fields), respectively in the Drinfeld center $Z(\cc)$ if 
$\ell$ is a circle (bulk and defect fields). This \emph{field object} describes the
corresponding field insertion as a representation of the chiral symmetry algebra
of the conformal field theory.
	
 \item
Universal correlators: The string-net constructions assign a vector space 
$\SNfrc(\surf;-,{\backsim})$ and $ \SNc(\surf;\df_{\!\ell}-,\df_{\!\ell'}{\backsim})$,
respectively, to a surface $\surf\colon\ell\PtO\ell'$ with boundary conditions for 
the incoming boundary $\ell$ and outgoing boundary $\ell'$. The forgetful functor 
$u\colon \cfrc\Rarr~ \cbc$ provides us with a family of natural transformations
  \be
  \ucorc(\surf) \Colon \SNfrc(\surf;-,{\backsim}) \xRightarrow{\phantom{ww}}
  \SNc(\surf;\df_{\!\ell}-,\df_{\!\ell'}{\backsim})
  \label{eqI:ucorc(surf)}
  \ee
parametrized by the 1-morphisms $\surf\colon\ell\PtO\ell'$ in $\bbord$. The natural
transformations \eqref{eqI:ucorc(surf)} have been called \emph{universal correlators}
in \cite{fusY2}, because they specify conformal field theory correlators, in a way
that will be recalled in some detail in Section \ref{sec:cft1}.
	
\end{itemize}

In the present paper we address two questions that arise in this context.

{\def\leftmargini{9.9em}
\begin{itemize}
\item[{\bf Question 1:}] Is there a natural way to package the field functors and the 
 \\
universal correlators into a single mathematical quantity?
\end{itemize}

\noindent
The answer to this question is affirmative. It combines in a novel way field content
and correlators of a quantum field theory in a single mathematical quantity.
To arrive at this conclusion requires a 
further refinement of the setup. Our starting point are the modular functors given 
by string-net constructions. In our context, a modular functor is a functor from a
geometric two-dimensional category to an algebraic two-dimensional category. We 
observe in this article 
that it is very natural to work with two-dimensional categorical structures 
that are double categories rather than, as in \cite{fusY2}, bicategories. A double 
category has two types of 1-morphisms, known as horizontal and vertical 1-morphisms.
For each type there is a composition, being strictly associative for vertical
1-morphisms, but associative only up to associators for horizontal 1-morphisms.
 
On the geometric side we deal with the double category $\dbord$ of open-closed
two-dimensi\-o\-nal bordisms,
which combines two obvious types of morphisms between manifolds with boundary:
A horizontal 1-morphism in $\dbord$ is, as usual for a bordism category, 
a cospan $\overline\ell \,{\to}\, \surf \,{\leftarrow}\, \ell'$ of manifolds,
with $\ell$ and $\ell'$ closed oriented manifolds and $\surf$ a compact oriented
manifold with boundaries. The vertical 1-morphisms in $\dbord$ are smooth embeddings
of manifolds. Hereby the double categorical setting combines two different notions
of locality in field theory, which is highly satisfactory.
On the algebraic side, another double category $\dprof$ imposes itself. $\dprof$
has profunctors (also known as bimodules) as horizontal 1-morphisms and
ordinary functors as vertical 1-morphisms. Indeed, a string-net construction
naturally associates a profunctor to a bordism and a functor to
a smooth embedding. Our first main result (Theorem \ref{thm:dfunsn}) is:

\begin{thmI} 
Let $\cb$ be a pointed pivotal bicategory.
There is a symmetric monoidal double functor 
  \be
  \dSNb\Colon \dbord \rarr~ \dprof \,.
  \ee
\end{thmI}

By applying this theorem to the two pointed pivotal bicategories $\cbc$ and $\cfrc$,
we identify the natural setting that makes it possible to
describe field functors and universal correlators in a single structure.
We find (Theorem \ref{thm:vtransUCor}):

\begin{thmI} 
Let $\cc$ be a \ko-linear pivotal tensor category. There is a canonical monoidal 
vertical transformation 
  \be
  \ducorc \Colon \dSNfrc \xRightarrow{\phantom{xx}} \dSNc \,.
  \label{eq:intro.vert}
  \ee
The component of $\ducorc$ at an object $\ell$ in $\dbord$ is given by the field 
functor $\df_{\!\ell}$ in \eqref{eqI:fifu}, while its component at a horizontal 
1-morphism $\surf$ in $\dbord$ is the universal cor\-re\-lator $\ucorc(\surf)$
\eqref{eqI:ucorc(surf)}.
\end{thmI}

At this point it is reasonable to ask how closely the two pivotal pointed
bicategories $\cfrc$ and $\cbc$ are related. As it turns out, $\cfrc$ can be seen as 
being obtained from $\cbc$ by a suitable completion -- 
an orbifold completion in the sense of \cite{caRun3}, which is a 
variant of higher Karoubification, or condensation \cite{gaJo2}. This leads us to 

{\def\leftmargini{9.9em} 
\begin{itemize}
\item[{\bf Question 2:}] 
How close are the bicategorical string-net \\ modular functors for $\cfrc$ and $\cbc$?
\end{itemize}

\noindent
In Theorem \ref{thm:eqvmf} we give a concise answer in the language of double
categories: The monoidal vertical transformation $\ducorc$ \eqref{eq:intro.vert} is a
monoidal equivalence. With the help of constructions \cite{wesSh} based on the double 
categorical notions of conjoints and companions, $\ducorc$ lifts to an equivalence 
  \be
  \widehat{\ducorc}\Colon\SNfrc \xRightarrow{~\simeq\,} \SNc
  \ee
in the world of monoidal bicategories, Theorem \ref{thm:eqvmf} condenses the
following two results about field functors and universal correlators, which
are both obtained by geometric constructions using string nets:
 \Itemize

 \item 
Theorem \ref{thm:Feq} states that each of the field functors $\df_{\!\ell}$
is part of an adjoint equivalence of categories: 	

\begin{thmI} 
Let $\cc$ be a \ko-linear pivotal tensor category and $\ell$ a compact oriented
one-ma\-ni\-fold with possibly non-empty boundary. There is an adjoint equivalence 
  \be
  \df{\!_\ell}^{} ~\colon~ \SNfrc(\ell)
  \begin{tikzcd}[column sep=2.2em]
  \hspace*{-0.4em} \ar[yshift=4.5pt]{r} & \hspace*{-0.4em}
  \ar[yshift=-3.4pt]{l}[swap,yshift=-0.9pt]{\simeq}
  \end{tikzcd}
  \SNc(\ell) ~\colon\, \imath_\ell^{} \,.
  \ee
\end{thmI}

\item
By theorem \ref{thm:UcorIso},
universal correlators give {\em isomorphisms} of string-net spaces:

\begin{thmI} 
Let $\cc$ be a \ko-linear pivotal tensor category
and $\surf\colon\ell\ptO\ell'$ an oriented open-clo\-sed bordism. Then
for every boundary datum $(a,a')\iN\SNfrc(\overline\ell\,{\sqcup}\,\ell')$,
the universal correlator provides an isomorphism 
  \be 
  \ucorc(\surf)_{a,a'} \Colon
  \SNfrc(\surf;a,a')\iso\SNc(\surf;\df_{\!\ell}^{}\,a,\df_{\!\ell'}^{}\,a')
  \ee of string-net spaces.
\end{thmI}
\end{itemize}

As we explain in Subsection \ref{sec:cft1}, this shows that for the class of conformal
field theories to which our results apply, CFT correlators exhaust the complete space
of conformal blocks, in the sense that every conformal block can be expressed as a 
linear combination of correlators.

 \medskip

The rest of this paper is organized as follows. Rather than presenting the final
picture in terms of double categories right away, we start in Section \Ref{sec:fresh}
with presenting a few of the results of \cite{fusY,fusY2} in a way adapted to our
present perspective. We then employ, in Section \ref{sec:fifu+ucor:equiv}, string-net 
techniques to prove that field functors are part of an adjoint equivalence and that
universal correlators are isomorphisms. In Section \ref{sec:dbl-1} we review 
pertinent double-ca\-te\-gorical structures. In the final Section \ref{sec:dbl-2}
we then use the double-categorical notions to state the results of Section
\ref{sec:fifu+ucor:equiv} concisely in a conceptual manner.

\bigskip

\emph{Acknowledgements}:
J.F.\ is supported by VR under project no.\ 2022-02931.
C.S.\ acknowledges support by the Deutsche Forschungsgemeinschaft (DFG, German
Research Foundation) under Germany's Excellence Strategy - EXC 2121 `
`Quantum Universe'' - 390833306 and the Collaborative Research Center - SFB 1624 `
`Higher structures, moduli spaces and integrability'' - 506632645.
Y.Y. acknowledges support by the Deutsche Forschungsgemeinschaft (DFG, 
German Research Foundation) under the Walter Benjamin Programme - 694132.

%%%%%%%%%%%%%%%%%%%%%%%%%%%%%%%%%%%%%%%%%%%%%%%%%%%%%%%%%%%%%%%%%%%%%%%%

\section{A reminder about string-net constructions} \label{sec:fresh}

In this paper, we use a \emph{string-net construction} to obtain a modular functor.
In the present section and in Section \ref{sec:fifu+ucor:equiv}
we work with the following variant of the notion of a modular functor: An 
\emph{open-closed modular functor} -- or \emph{modular functor}, for short -- 
is a symmetric monoidal pseudofunctor
  \be
  \bl \Colon \bbord \rarr~ \bprof
  \label{eq:def:bl}
  \ee 
between the following symmetric monoidal bicategories: $\bbord$ is the bicategory
having compact oriented one-manifolds $\ell$ with possibly non-empty boundary as objects, \twodim\ oriented open-closed bordisms $\ell\PtO\ell'$ as 1-morphims,
and isotopy classes of diffeomorphisms as 2-morphisms; $\bprof$ -- for \ko\ a field,
which throughout the paper we take to be algebraically closed -- is the bicategory
of \ko-linear essentially small categories, profunctors and natural transformations. 
As we will see, profunctors are indispensable in our context. 
A 1-morphism between two open-closed modular functors $\bl$ and $\bl'$
is a pseudonatural transformation $\bl \Rarr~ \bl'$ of modular functors. 
This provides a natural notion of equivalence of open-closed modular functors.

Any string-net construction requires a categorical input datum. The simplest and 
and most common such datum is a pivotal, or even spherical, fusion category $\cc$ 
\cite{leWe,kirI24}. An object of $\bbord$ is a disjoint union of oriented intervals
$\di$ and oriented circles $\ds^1$. Accordingly, the categories that the modular
functor $\bl$ assigns to the one-manifolds $\di$ and $\ds^1$ are important
characteristics of the modular functor. In the case of a pivotal fusion category
as input, we expect equivalences
  \be
  \bl_\cc(\di) \simeq\cc \qquad\text{and}\qquad \bl_\cc(\ds^1) \simeq \czc
  \label{eq:blc-I-S1}
  \ee
of \ko-linear categories, where $\czc$ is the Drinfeld center of $\cc$. 

For the construction in the present paper, we work with a more general input datum
\Cite{Thm.\,3.36}{fusY2}: a bicategory $\cb$ that is \emph{pointed}, i.e.\ is endowed 
with a distinguished object $\disto\iN\cb$, and 
strictly \emph{pivotal}, i.e.\ rigid with a strict equality between
the identity and double dual pseudofunctors. This -- mathematically natural --
setting is also the appropriate one for the applying the string-net construction to 
the construction of correlators of two-dimensional conformal field theories.

Let us describe the basic idea of this bicategorical string-net construction. A 
$\cb$-\emph{decoration} $a$ on a compact oriented one-manifold $\ell$, possibly with
boundary, consists of a pair of data: first, a collection of finitely many oriented
points in the interior of $\ell$, whose complement partitions $\ell$ into finitely 
many intervals; second, an assignment of an object in $\cb$ to each of the intervals
in this partition, and of a compatible 1-morphism in $\cb$ to each of the oriented
points. If $\ell$ has a non-empty boundary, then the object that is assigned to any
intervals that contains a boundary point is required to be the distinguished object
$\disto\iN\cb$.
We allow for $\cb$-decorations for which the collection of points on a connected 
component of $\ell$ is empty; if such a connected component is a circle, then we can 
assign to it any object of $\cb$. 
The pivotal structure on $\cb$ is, by definition, the identity on objects
and induces an involution on 1-morphisms. As a consequence, any
$\cb$-decoration $a$ on a one-manifold $\ell$ induces a $\cb$-decoration
$\overline a$ on the one-manifold $\overline\ell$ with opposite orientation.

Now let $\surf\colon \ell\PtO\ell'$ be a two-dimensional bordism, and let $a$ and 
$a'$ be $\cb$-decorations on the one-manifolds $\ell$ and $\ell'$, respectively. 
(For our purposes, a bordim is a two-manifold with boundary; we do not consider
them up to diffeomorphism.) The basic building block of the string-net construction
is the assignment of a vector space $\snb(\surf;a,a')$ to these data. An element
of $\snb(\surf;a,a')$ is a linear combination of finite $\cb$-colored string diagrams
on $\surf$ that are compatible with the boundary datum $a'$ on $\ell$ and
$\overline{a}$ on $\overline\ell$, modulo the local relations that are captured by 
the graphical calculus for the pivotal bicategory $\cb$. The mapping class group
$\mcg(\surf)$ of $\surf$ acts naturally on the set of $\cb$-colored string diagrams; 
this induces a (genuine, non-projective) action of $\mcg(\surf)$ on
the vector space $\snb(\surf;a,a')$.

The first step in the construction of a modular functor consists in assigning to any
one-ma\-nifold $\ell$ a \ko-linear category $\snb(\ell)$, called the \emph{cylinder 
category} of $\ell$. By definition, $\snb(\ell)$ is the category that 
has $\cb$-de\-co\-ra\-tions on the one-manifold $\ell$ as objects and the vector 
spaces $\snb(\ell \Times \Di;a,a')$, with $\Di$ is the standard interval, as morphism
spaces. Morphisms of $\snb(\ell)$ can thus be represented by $\cb$-colored 
string diagrams on the cylinder $\ell\Times\Di$; composition in the cylinder category
is then induced by concatenating (or, in other words, stacking) string diagrams. 

We note a few properties of cylinder categories: For the one-manifold $\overline\ell$
with reversed orientation, the cylinder category $\snb(\overline\ell)$ is the
opposite category of $\snb(\ell)$. For a disjoint union $\ell\,{\sqcup}\, \ell'$, 
the cyclinder category comes with a canonical equivalence to the tensor product
of the cylinder categories of $\ell$ and $\ell'$,
  \be
  \snb(\ell\sqcup\ell') \simeq \snb(\ell) \otimes \snb(\ell') \,.
  \ee
Here the tensor product is the Cartesian product of $\vct$-enriched categories, 
i.e.\ objects are ordered pairs and morphism spaces are tensor products over $\ko$.

Invoking the concatenation of string nets, one sees that for every open-clo\-sed 
bordism $\surf\colon\ell\PtO\ell'$, the vector spaces $\snb(\surf;a,a')$ combine 
into a \ko-linear functor
  \be
  \snb(\surf) \Colon \snb(\ell)^{\op} \oti \snb(\ell') \rarr~ \vct \,.
  \label{eq:SNB(surfneu)}
  \ee
Again, vector spaces come with actions of the mapping class group $\mcg(\surf)$.
Put differently, we deal with a profunctor from $\snb(\ell)$ to $\snb(\ell')$;
we denote this profunctor by the same symbol:
  \be
  \snb(\surf) \Colon \snb(\ell)\Pto \snb(\ell') \,.
  \label{eq:SNB(surfneu)neu}
  \ee
It is worth stressing that here profunctors, rather then functors, arise naturally.
Since a profunctor is equivalent to a colimit-preserving functor between presheaf
categories, one could avoid working with profunctors.
These presheaf categories are, however, rather large. 

There is no reason to expect a cylinder category to be idempotent-complete.
The cylinder categories $\snb(\ell)$ indeed do not provide all the objects we need
for our constructions. In contrast, their Cauchy completions do.
Recall that a category is Cauchy complete iff it has 
all small absolute colimits. In the $\vct$-enriched setting, this translates to 
having all finite direct sums and all idempotent-splittings.
Hence we set up an open-closed modular functor 
  \be
  \SNb \Colon \bbord \rarr~ \bprof
  \label{eq:def:SNb}
  \ee
as follows: $\SNb$ sends an oriented one-manifold $\ell$ to the Cauchy completion
  \be
  \SNb(\ell) \coloneqq \mathrm{Cau(\snb(\ell))}
  \label{eq:SNb=(snb)}
  \ee
of the cylinder category $\snb(\ell)$. When $\snb(\ell)$ is complete
with respect to finite direct sums (which is e.g.\ the case if all spaces of 
2-morphisms of the bicategory $\cb$ are fi\-ni\-te-di\-rect-sum complete), then
$\SNb(\ell)$ coincides with the Karoubi envelope $\kar(\snb(\ell))$ of the cylinder 
category. In the sequel we will assume that this is indeed the case. Accordingly,
the objects of $\SNb(\ell)$, to which we also refer as the \emph{boundary data} on 
$\ell$, are idempotents in $\snb(\ell)$. 

The categories $\snb(\ell)$ and $\SNb(\ell)$ are not equivalent in the bicategory
$\mathcal{C}at_\ko$ of $\ko$-linear categories, linear functors and natural
transformations. They are, however, canonically equivalent in the bicategory
$\bprof$ (which is one of the advantages of working with profunctors). This follows 
from the fact that the composition of profunctors is via a coend (which exists
because the category $\vct$ of \emph{all} vector spaces is cocomplete), and that the 
coend over a cylinder category coincides with the coend over its Cauchy completion.
As a consequence, for any open-clo\-sed bordism $\surf\colon\ell\PtO\ell'$ we have
a canonical profunctor 
  \be
  \SNb(\surf) \Colon \SNb(\ell) \Pto \SNb(\ell') \,.
  \ee
More concretely, to any boundary datum 
  \be
  (a,a') \in \SNb(\ell)^{\op}\oti\SNb(\ell')
  \simeq \SNb(\overline\ell\,{\sqcup}\,\ell')
  \ee
there is associated a vector space $\SNb(\surf;a,a')$ whose elements are represented
by linear combinations of $\cb$-colored string diagrams on $\surf$ that are 
compatible with the boundary datum (in which appropriate idempotents have to be 
inserted). And again these vector spaces come with an action of the mapping class
group $\mcg(\surf)$. Note that a priori the vector space $\SNb(\surf;a,a')$
is not necessarily finite-dimensional; to attain finite-dimensionality,
further properties of $\cb$ are required.

\medskip

There are two particularly appealing classes of pointed pivotal bicategories $\cb$.
The first is the delooping $\cbc$ of a pivotal fusion  category $\cc$; in this case
$\SNc \,{\equiv}\, \SN_{\cbc}$ coincides (see \Cite{Rem.\,3.6.2}{fusY2}) with the 
modular functor $\bl_\cc$ that we alluded to in \eqref{eq:blc-I-S1}.
The other one is the pivotal bicategory $\cfrc$ of \emph{$\Delta$-separable
symmetric Frobenius algebras} in a pivotal fusion category $\cc$,\,%
 \footnote{~The bicategory $\cfrc$ we are considering here differs from the one
 considered in \cite{fusY,fusY2}, which it contains as a full sub-bicategory:
 in \cite{fusY,fusY2} the Frobenius algebras are in addition required to be simple
 and to satisfy a more restrictive separability condition.}
with the monoidal unit $\tu$, with its trivial structure of a Frobenius algebra, as 
distinguished object.

Both of these types of bicategories are based on an underlying pivotal fusion
category $\cc$. Moreover, there is an obvious functor between them, the forgetful 
functor
  \be
  u \Colon \cfrc \rarr~ \cbc
  \label{eq:ufr}
  \ee
that sends every Frobenius algebra in $\cc$ to the single 
object in the delooping $\cbc$ and takes every bimodule and bimodule morphism to
their underlying object and morphism, respectively, in $\cc$. It is therefore 
natural to wonder whether the string-net constructions $\SNfrc$ and
$\SNc$ for the same pivotal fusion category $\cc$ are related by a 1-morphism, and
if so, whether that 1-morphism is an equivalence. As we will now explain, such a
1-morphism can indeed be obtained is a natural manner. In subsequent sections we
will show that it is an equivalence.
 
A pseudonatural transformation from $\SNfrc$ to $\SNc$ has two layers of structure: 
first, a family of linear functors 
  \be
  \df_{\!\ell}^{} \equiv \df_{\!\cc}(\ell) 
  \Colon \SNfrc(\ell) \rarr~ \SNc(\ell)
  \label{eq:def(ell)}
  \ee
parametrized by the objects $\ell\iN\bbord$, and second, a family of
natural transformations
  \be
  \ucorc(\surf) \Colon \SNfrc(\surf;-,{\backsim}) \xRightarrow{\phantom{ww}}
  \SNc(\surf;\df_{\!\ell}-,\df_{\!\ell'}{\backsim})
  \label{eq:ucorc(surf)}
  \ee
parametrized by the 1-morphisms $\surf\colon\ell\PtO\ell'$ in $\bbord$.
 
We refer to the functors \eqref{eq:def(ell)} as the family of \emph{field functors},
and to the natural transformations \eqref{eq:ucorc(surf)} as the family of 
\emph{universal correlators}.
This terminology is borrowed from an application to conformal field theory, which we
will recall in Section \ref{sec:cft1} below, and which was the main theme of
\cite{fusY} and of Section 4 of \cite{fusY2}. The reader unfamiliar with conformal
field theory can safely regard the terms `field functor' and `universal correlator'
as mere vocabulary for the two structural layers of a pseudonatural transformation.

We define the field functors \eqref{eq:def(ell)} and universal correlators 
\eqref{eq:ucorc(surf)} as follows.
The functor $\df_{\!\ell}$ is the extension via Cauchy completion (which by the
universal property of Cauchy completion is unique) of a functor
  \be
  \mathrm{f}_{\ell}^{} \Colon \snfrc(\ell) \rarr~ \SNc(\ell) 
  \ee
whose domain is the cylinder category. The functor $\mathrm{f}_{\ell}^{}$ is,
in turn, given by the following prescription:
$\mathrm{f}_{\ell}^{}$ maps objects of $\snfrc(\ell)$ according to
  \be
{\tikzstyle{every picture}=[tikzfig]
  {\begin{tikzpicture}
	\begin{pgfonlayer}{nodelayer}
		\node [style=none] (113) at (0, 1) {};
		\node [style=none] (114) at (-1, 0) {};
		\node [style=none] (115) at (1, 0) {};
		\node [style=none] (116) at (0, -1) {};
		\node [style=bdot] (121) at (0, 1) {};
		\node [style=bdot] (122) at (0, -1) {};
		\node [style=none] (125) at (0, 1.5) {\scriptsize $X$};
		\node [style=none] (126) at (1.5, 0) {\scriptsize $Y$};
		\node [style=bdot] (127) at (1, 0) {};
		\node [style=none] (128) at (0, -1.5) {\scriptsize $Z$};
	\end{pgfonlayer}
	\begin{pgfonlayer}{edgelayer}
		\draw [style=vt-lgr, bend right=45] (114.center) to (116.center);
		\draw [style=vt-pi, bend right=45] (116.center) to (115.center);
		\draw [style=vt-lbl, bend right=45] (115.center) to (113.center);
		\draw [style=vt-lgr, bend right=45] (113.center) to (114.center);
	\end{pgfonlayer}
\end{tikzpicture}}
  {%
  }}%
 \quad\xmapsto{\phantom{xx}}\quad %
{\tikzstyle{every picture}=[tikzfig]
  {\begin{tikzpicture}
	\begin{pgfonlayer}{nodelayer}
		\node [style=none] (137) at (0, 1.625) {};
		\node [style=none] (138) at (-1.625, 0) {};
		\node [style=none] (139) at (1.625, 0) {};
		\node [style=none] (140) at (0, -1.625) {};
		\node [style=none] (125) at (0, 2.75) {\scriptsize $X$};
		\node [style=none] (126) at (2.75, 0) {\scriptsize $Y$};
		\node [style=none] (113) at (0, 1) {};
		\node [style=none] (114) at (-1, 0) {};
		\node [style=none] (115) at (1, 0) {};
		\node [style=none] (116) at (0, -1) {};
		\node [style=none] (129) at (0, 2.25) {};
		\node [style=none] (130) at (-2.25, 0) {};
		\node [style=none] (131) at (2.25, 0) {};
		\node [style=none] (132) at (0, -2.25) {};
		\node [style=none] (141) at (0, -2.75) {\scriptsize $Z$};
	\end{pgfonlayer}
	\begin{pgfonlayer}{edgelayer}
		\draw [style=vt-pi, bend right=45] (140.center) to (139.center);
		\draw [style=vt-lbl, bend right=45] (139.center) to (137.center);
		\draw [style=vt-lgr] (137.center)
			 to [bend right=45] (138.center)
			 to [bend right=45] (140.center);
		\draw [style=vt-bl-di] (116.center) to (132.center);
		\draw [style=vt-bl-di] (113.center) to (129.center);
		\draw [style=vt] (115.center)
			 to [bend right=45] (113.center)
			 to [bend right=45] (114.center)
			 to [bend right=45] (116.center)
			 to [bend right=45] cycle;
		\draw [style=vt] (131.center)
			 to [bend right=45] (129.center)
			 to [bend right=45] (130.center)
			 to [bend right=45] (132.center)
			 to [bend right=45] cycle;
		\draw [style=vt-bl-di] (115.center) to (131.center);
	\end{pgfonlayer}
\end{tikzpicture}}
  {%
  }}%

  \label{eq:pic:fl}
  \ee
to objects in the Cauchy-completed category $\SNc(\ell)$. Here the right hand side
shows the idempotent that underlies the object in the Cauchy completion $\SNc(\ell)$,
and is therefore depicted by a $\cc$-colored string net on the cylinder 
$\ell\Times\Di$. This string net is obtained as follows. First the boundary datum in
$\snfrc(\ell)$ is thickened to an annulus. Then the Frobenius algebra that colors a
2-cell of the so obtained annulus is replaced by a network of \emph{Frobenius lines},
i.e.\ lines that are colored with the object in $\cc$ that underlies the Frobenius
algebra. Technically, the network is required to furnish a \emph{full Frobenius graph}
in the sense of Definition 3.20 of \cite{fusY}; the vertices of the network are 
colored by structure morphisms of the Frobenius algebra. Further, bimodules over the
Frobenius algebras are replaced by the underlying objects in $\cc$ as well; left and 
right actions are used to connect Frobenius lines to bimodule lines. In order for
the so obtained string net on $\ell\Times\Di$ to actually constitute an idempotent,
the requirement of $\Delta$-separability on the Frobenius algebras must be imposed.
 
To see what happens on morphisms, we use the forgetful functor 
$u$ from \eqref{eq:ufr}. The functor $u$
is a Frobenius monoidal functor and therefore gives rise to the operation of
$u$-\emph{conjugation}, in the sense of Definition 2.26 of \cite{fusY2}; this 
allows us to define the map on the 0-cells of the string diagrams.

We have already seen that the functor $\mathrm{f}_{\ell}^{}$ does
not simply map a boundary datum in $\snfrc$ to a boundary datum in $\SNc^{}$ by
using the forgetful functor $u$, but rather it also adds a network of full 
Frobenius graphs. Here we consider a map between string nets that is obtained by 
employing the same logic: we apply $u$ to objects, $u$-conjugation to vertices,
and then add full Frobenius graphs. The following picture illustrates the resulting 
prescription:
  \be
  \hspace*{-1.5em}
  \scalebox{0.75}{%
{\tikzstyle{every picture}=[tikzfig]
  {\input{WS2.tikz}}
  {%
  }}%
} \hspace*{0.7em} \xmapsto{~~~~~~} \hspace*{-0.2em}
  \scalebox{0.75}{%
{\tikzstyle{every picture}=[tikzfig]
  {\input{WS3.tikz}}
  {%
  }}%
}
  \ee
It remains to check that this prescription indeed defines a linear map 
  \be
  \ucorc(\surf)_{a,a'} \Colon 
  \SNfrc(\surf;a,a') \rarr~ \SNc(\surf;\df_{\!\ell}\,a,\df_{\!\ell'}\,a') \,.
  \label{eq:ucorc-surf-aa'}
  \ee
This is a consequence of the functoriality of the string-net construction under rigid
separable Frobenius functors, i.e.\ \Cite{Def.\,2.28}{fusY2} 2-functors endowed
with a lax and a \colax\ structure\,%
 \footnote{~In \Cite{Def.\,2.28}{fusY2} the term `oplax' instead of `colax' was used.}
that fulfill special-Frobenius-like compatibility conditions which strictly preserve
adjoints. The pertinent rigid separable Frobenius functor is the forgetful functor
\eqref{eq:ufr}. Lastly, $\ucorc(-)$ is compatible with sewing in the obvious sense. 
As it turns out, a concise analysis of sewing is facilitated by using the language 
of double categories; we defer this issue to Section \ref{sec:dbl-2}.

\begin{rem} \label{rem:rigidF}
Every Frobenius functor preserves adjoints. For the case of monoidal (not necessarily
rigid) Frobenius functors this statement can be found in \cite{daPas}. The arguments
presented there can easily be generalized to the bicategorical setting. In the
present paper we continue to use the unframed version of graphical calculus developed
in Section 2 of \cite{fusY2} to define string nets; accordingly we work with strictly
pivotal bicategories and with rigid separable Frobenius functors between them, rather
than with the more general class of pivotal separable Frobenius functors.
\end{rem}

\begin{rem}
The closed sector of the modular functor $\bl_\cc$ is expected to be isomorphic to
the Reshetikhin-Turaev modular functor for the Drinfeld center $\czc$. When combined 
with the classification result for modular functors that was obtained in \cite{brWo},
this isomorphism implies in particular that, for $\cc$ pivotal fusion, the
equivalence $\bl_\cc(\ds^1)\,{\simeq}\,\czc$ in \eqref{eq:blc-I-S1} is even an 
equivalence of ribbon categories. 
Also, as shown in \cite{mswy}, for $\cc$ any pivotal finite tensor category a
string-net modular functor can be obtained by applying a two-dimensional admissible
skein module construction \cite{cogp4} to the tensor ideal of projective objects in
$\cc$ and performing a finite cocompletion, and this modular functor is equivalent
to the one built in \cite{lyub6} from the Drinfeld center of $\cc$.
\end{rem}

%%%%%%%%%%%%%%%%%%%%%%%%%%%%%%%%%%%%%%%%%%%%%%%%%%%%%%%%%%%%%%%%%%%%%%%%

\section{Field functors and universal correlators as equivalences}
\label{sec:fifu+ucor:equiv}

\subsection{Field functors are equivalences of cylinder categories}

We are now going to show that for every fixed compact oriented one-manifold $\ell$,
with possibly non-empty boundary, the field functor $\df_{\!\ell}$
is an equivalence of categories. To this end we construct a functor
  \be
  \imath_\ell^{} \Colon \SNc(\ell) \rarr~ \SNfrc(\ell)
  \label{eq:def:imath}
  \ee
which will be shown to be an essential inverse of the field functor 
$\df_{\!\ell} \colon \SNfrc(\ell) \Rarr~ \SNc(\ell)$.
We specify the functor $\imath_\ell^{}$ as follows. $\imath_\ell^{}$ interprets 
$\cc$-colored boundary data and string nets as $\cfrc$-colored ones, by using the 
inclusion $\cbc \,{\hookrightarrow}\, \cfrc$ of
pointed bicategories that sends the single object $*$ of $\cbc$ to the trivial 
Frobenius algebra $\tu \iN \cc$. The prescription on objects determines $\imath$
completely, because $\tu$-(bi)modules do not carry any information apart from their
underlying objects.
It is straightforward to see that $\imath_\ell$ provides a section for the field 
functor, i.e.
  \be
  \df_{\!\ell}^{}\circ\imath_\ell^{} = \id_{\SNc(\ell)} \,. 
  \label{eq:F.imath=1}
  \ee
We claim that the functors $\df_{\!\ell}$ and $\imath_{\ell}$
in fact furnish an adjoint equivalence $\SNfrc(\ell) \,{\simeq}\, \SNc(\ell)$ of
linear categories. To show this, we construct the counit 
  \be
  \varepsilon_\ell^{} \Colon
  \imath_\ell^{}\circ\df_{\!\ell}^{} \rarr~ \id_{\SNfrc(\ell)}
  \label{eq:counit}
  \ee
of the adjunction and show that it is an isomorphism.

For keeping the discussion uncluttered, in the pictures below we illustrate our
arguments by taking the one-manifold $\ell$ to be the circle $\ds^1$; the 
generalization to any compact oriented one-manifold is easy. Thus consider the
$\cfrc$-colored string net $[\grph]$ that is represented by the graph 
  \be
  \grph \equiv \grph_{\!X,Y,Z;\varphi} ~= \quad%
{\tikzstyle{every picture}=[tikzfig]
  {\begin{tikzpicture}
	\begin{pgfonlayer}{nodelayer}
		\node [style=none] (25) at (0, -0.5) {};
		\node [style=none] (26) at (-1.75, 0.25) {};
		\node [style=none] (27) at (1.75, 0.25) {};
		\node [style=none] (28) at (0, 3.75) {};
		\node [style=none] (29) at (-1.75, 3) {};
		\node [style=none] (30) at (1.75, 3) {};
		\node [style=none] (45) at (0, 3.75) {};
		\node [style=none] (46) at (-1.75, 3) {};
		\node [style=none] (47) at (1.75, 3) {};
		\node [style=none] (48) at (0, 2.25) {};
		\node [style=none] (31) at (0, -0.5) {};
		\node [style=none] (32) at (-1.75, 0.25) {};
		\node [style=none] (35) at (1.75, 0.25) {};
		\node [style=none] (39) at (-1.75, -3) {};
		\node [style=none] (40) at (1.75, -3) {};
		\node [style=none] (41) at (0, -3.75) {};
		\node [style=none] (10) at (-1.75, 0.25) {};
		\node [style=none] (11) at (1.75, 0.25) {};
		\node [style=mid-nc] (12) at (0, -0.5) {$\varphi$};
		\node [style=none] (14) at (-0.25, 1) {};
		\node [style=none] (15) at (-1.75, 0.25) {};
		\node [style=none] (16) at (0.25, 1) {};
		\node [style=none] (17) at (1.75, 0.25) {};
		\node [style=none] (0) at (0, 3.75) {};
		\node [style=none] (1) at (-1.75, 3) {};
		\node [style=none] (2) at (1.75, 3) {};
		\node [style=none] (3) at (0, 2.25) {};
		\node [style=none] (5) at (-1.75, -3) {};
		\node [style=none] (6) at (1.75, -3) {};
		\node [style=none] (7) at (0, -3.75) {};
		\node [style=none] (18) at (1.3, -0.25) {};
		\node [style=none] (19) at (1.425, -0.225) {};
		\node [style=none] (20) at (-0.51, -2.92) {$X$};
		\node [style=none] (21) at (0.43, 1.53) {$Y$};
		\node [style=none] (22) at (-0.95,-0.88) {$Z$};
	\end{pgfonlayer}
	\begin{pgfonlayer}{edgelayer}
		\draw [style=fi-pu] (29.center)
			 to (26.center)
			 to [in=-180, out=-90, looseness=0.75] (25.center)
			 to [in=-90, out=0, looseness=0.75] (27.center)
			 to (30.center)
			 to [in=0, out=90, looseness=0.75] (28.center)
			 to [in=90, out=180, looseness=0.75] cycle;
		\draw [style=op-lsh] (46.center)
			 to [in=-180, out=-90, looseness=0.75] (48.center)
			 to [in=-90, out=0, looseness=0.75] (47.center)
			 to [in=0, out=90, looseness=0.75] (45.center)
			 to [in=90, out=180, looseness=0.75] cycle;
		\draw [style=fi-bl] (35.center)
			 to (40.center)
			 to [in=0, out=-90, looseness=0.75] (41.center)
			 to [in=-90, out=-180, looseness=0.75] (39.center)
			 to (32.center)
			 to [in=-180, out=-90, looseness=0.75] (31.center)
			 to [in=-90, out=0, looseness=0.75] cycle;
		\draw [style=vt-bl-di] (7.center) to (12);
		\draw [style=vt-bl] (10.center)
			 to [in=-180, out=-90, looseness=0.75] (12.center)
			 to [in=-90, out=0, looseness=0.75] (11.center);
		\draw [style=vt-bl-di] (12) to (3.center);
		\draw [style=vt-bl-da, in=90, out=180, looseness=0.75] (14.center) to (15.center);
		\draw [style=vt-bl-da, in=90, out=0, looseness=0.75] (16.center) to (17.center);
		\draw [style=vt] (1.center)
			 to [in=-180, out=-90, looseness=0.75] (3.center)
			 to [in=-90, out=0, looseness=0.75] (2.center)
			 to [in=0, out=90, looseness=0.75] (0.center)
			 to [in=90, out=180, looseness=0.75] cycle;
		\draw [style=vt] (5.center)
			 to [in=-180, out=-90, looseness=0.75] (7.center)
			 to [in=-90, out=0, looseness=0.75] (6.center);
		\draw [style=vt] (1.center) to (5.center);
		\draw [style=vt] (2.center) to (6.center);
		\draw [style=vt-bl-di] (18.center) to (19.center);
	\end{pgfonlayer}
\end{tikzpicture}}
  {%
  }}%

  \ee
on the cylinder over $\ell$.
Here the lower and upper half of the cylinder are labeled by Frobenius algebras
$A$ and $B$ (which we indicate by using two different colors),
$X\iN\hfrc(A,A)$, $Y\iN\hfrc(B,B)$ and $Z\iN\hfrc(B,A)$ are bimodules, and 
$\varphi\iN \Hom_{B\bimod A}(Z\otA X,
         $\linebreak[0]$
Y\otB Z)$
is a bimodule morphism. The graph $\grph$ realizes the general form of
a morphism in the cylinder category $\SNfrc(\ell)$, in the sense that any object in
the arrow category for $\SNfrc(\ell)$ is isomorphic to $\grph_{\!X,Y,Z;\varphi}$
for a suitable choice of 1-morphisms $X,Y,Z$ and bimodule morphism $\varphi$. 

In case $A \eq B$, $\grph$ represents an object in the Cauchy-completed cylinder 
category $\SNfrc(\ell)$ iff the graph $\grph$ and the concatenation 
$\grph\cir\grph$ of $\grph$ with itself represent the same string net. Let
  \be
  a ~=~~ \scalebox{0.5}{%
{\tikzstyle{every picture}=[tikzfig]
  {\begin{tikzpicture}
	\begin{pgfonlayer}{nodelayer}
		\node [style=none] (25) at (0, -0.5) {};
		\node [style=none] (26) at (-1.75, 0.25) {};
		\node [style=none] (27) at (1.75, 0.25) {};
		\node [style=none] (28) at (0, 3.75) {};
		\node [style=none] (29) at (-1.75, 3) {};
		\node [style=none] (30) at (1.75, 3) {};
		\node [style=none] (45) at (0, 3.75) {};
		\node [style=none] (46) at (-1.75, 3) {};
		\node [style=none] (47) at (1.75, 3) {};
		\node [style=none] (48) at (0, 2.25) {};
		\node [style=none] (31) at (0, -0.5) {};
		\node [style=none] (32) at (-1.75, 0.25) {};
		\node [style=none] (35) at (1.75, 0.25) {};
		\node [style=none] (39) at (-1.75, -3) {};
		\node [style=none] (40) at (1.75, -3) {};
		\node [style=none] (41) at (0, -3.75) {};
		\node [style=none] (10) at (-1.75, 0.25) {};
		\node [style=none] (11) at (1.75, 0.25) {};
		\node [style=mid-nc] (12) at (0, -0.5) {};
		\node [style=none] (14) at (-0.25, 1) {};
		\node [style=none] (15) at (-1.75, 0.25) {};
		\node [style=none] (16) at (0.25, 1) {};
		\node [style=none] (17) at (1.75, 0.25) {};
		\node [style=none] (0) at (0, 3.75) {};
		\node [style=none] (1) at (-1.75, 3) {};
		\node [style=none] (2) at (1.75, 3) {};
		\node [style=none] (3) at (0, 2.25) {};
		\node [style=none] (5) at (-1.75, -3) {};
		\node [style=none] (6) at (1.75, -3) {};
		\node [style=none] (7) at (0, -3.75) {};
		\node [style=none] (18) at (1.3, -0.25) {};
		\node [style=none] (19) at (1.425, -0.225) {};
	\end{pgfonlayer}
	\begin{pgfonlayer}{edgelayer}
		\draw [style=fi-bl] (29.center)
			 to (26.center)
			 to [in=-180, out=-90, looseness=0.75] (25.center)
			 to [in=-90, out=0, looseness=0.75] (27.center)
			 to (30.center)
			 to [in=0, out=90, looseness=0.75] (28.center)
			 to [in=90, out=180, looseness=0.75] cycle;
		\draw [style=op-lsh] (46.center)
			 to [in=-180, out=-90, looseness=0.75] (48.center)
			 to [in=-90, out=0, looseness=0.75] (47.center)
			 to [in=0, out=90, looseness=0.75] (45.center)
			 to [in=90, out=180, looseness=0.75] cycle;
		\draw [style=fi-bl] (35.center)
			 to (40.center)
			 to [in=0, out=-90, looseness=0.75] (41.center)
			 to [in=-90, out=-180, looseness=0.75] (39.center)
			 to (32.center)
			 to [in=-180, out=-90, looseness=0.75] (31.center)
			 to [in=-90, out=0, looseness=0.75] cycle;
		\draw [style=vt-bl-di] (7.center) to (12);
		\draw [style=vt-bl] (10.center)
			 to [in=-180, out=-90, looseness=0.75] (12.center)
			 to [in=-90, out=0, looseness=0.75] (11.center);
		\draw [style=vt-bl-di] (12) to (3.center);
		\draw [style=vt-bl-da, in=90, out=180, looseness=0.75] (14.center) to (15.center);
		\draw [style=vt-bl-da, in=90, out=0, looseness=0.75] (16.center) to (17.center);
		\draw [style=vt] (1.center)
			 to [in=-180, out=-90, looseness=0.75] (3.center)
			 to [in=-90, out=0, looseness=0.75] (2.center)
			 to [in=0, out=90, looseness=0.75] (0.center)
			 to [in=90, out=180, looseness=0.75] cycle;
		\draw [style=vt] (5.center)
			 to [in=-180, out=-90, looseness=0.75] (7.center)
			 to [in=-90, out=0, looseness=0.75] (6.center);
		\draw [style=vt] (1.center) to (5.center);
		\draw [style=vt] (2.center) to (6.center);
		\draw [style=vt-bl-di] (18.center) to (19.center);
	\end{pgfonlayer}
\end{tikzpicture}}
  {%
  }}%
} ~~ \in\SNfrc(\ell)
  \label{eq:pic:a}
  \ee
represent an object in $\SNfrc(\ell)$, where 
the picture shows the underlying idempotent of $a$, and where we suppress all labels.
The composite functor $\imath_\ell\cir\df_{\!\ell}$ acts on objects of $\SNfrc(\ell)$
by mapping a $\cfrc$-colored string net to another $\cfrc$-co\-lo\-red string net
according to
  \be
  \imath_\ell^{}\cir\df_{\!\ell}^{} \Colon a ~=~ \scalebox{0.6}{%
{\tikzstyle{every picture}=[tikzfig]
  {\begin{tikzpicture}
	\begin{pgfonlayer}{nodelayer}
		\node [style=none] (25) at (0, -0.5) {};
		\node [style=none] (26) at (-1.75, 0.25) {};
		\node [style=none] (27) at (1.75, 0.25) {};
		\node [style=none] (28) at (0, 3.75) {};
		\node [style=none] (29) at (-1.75, 3) {};
		\node [style=none] (30) at (1.75, 3) {};
		\node [style=none] (45) at (0, 3.75) {};
		\node [style=none] (46) at (-1.75, 3) {};
		\node [style=none] (47) at (1.75, 3) {};
		\node [style=none] (48) at (0, 2.25) {};
		\node [style=none] (31) at (0, -0.5) {};
		\node [style=none] (32) at (-1.75, 0.25) {};
		\node [style=none] (35) at (1.75, 0.25) {};
		\node [style=none] (39) at (-1.75, -3) {};
		\node [style=none] (40) at (1.75, -3) {};
		\node [style=none] (41) at (0, -3.75) {};
		\node [style=none] (10) at (-1.75, 0.25) {};
		\node [style=none] (11) at (1.75, 0.25) {};
		\node [style=mid-nc] (12) at (0, -0.5) {};
		\node [style=none] (14) at (-0.25, 1) {};
		\node [style=none] (15) at (-1.75, 0.25) {};
		\node [style=none] (16) at (0.25, 1) {};
		\node [style=none] (17) at (1.75, 0.25) {};
		\node [style=none] (0) at (0, 3.75) {};
		\node [style=none] (1) at (-1.75, 3) {};
		\node [style=none] (2) at (1.75, 3) {};
		\node [style=none] (3) at (0, 2.25) {};
		\node [style=none] (5) at (-1.75, -3) {};
		\node [style=none] (6) at (1.75, -3) {};
		\node [style=none] (7) at (0, -3.75) {};
		\node [style=none] (18) at (1.3, -0.25) {};
		\node [style=none] (19) at (1.425, -0.225) {};
	\end{pgfonlayer}
	\begin{pgfonlayer}{edgelayer}
		\draw [style=fi-bl] (29.center)
			 to (26.center)
			 to [in=-180, out=-90, looseness=0.75] (25.center)
			 to [in=-90, out=0, looseness=0.75] (27.center)
			 to (30.center)
			 to [in=0, out=90, looseness=0.75] (28.center)
			 to [in=90, out=180, looseness=0.75] cycle;
		\draw [style=op-lsh] (46.center)
			 to [in=-180, out=-90, looseness=0.75] (48.center)
			 to [in=-90, out=0, looseness=0.75] (47.center)
			 to [in=0, out=90, looseness=0.75] (45.center)
			 to [in=90, out=180, looseness=0.75] cycle;
		\draw [style=fi-bl] (35.center)
			 to (40.center)
			 to [in=0, out=-90, looseness=0.75] (41.center)
			 to [in=-90, out=-180, looseness=0.75] (39.center)
			 to (32.center)
			 to [in=-180, out=-90, looseness=0.75] (31.center)
			 to [in=-90, out=0, looseness=0.75] cycle;
		\draw [style=vt-bl-di] (7.center) to (12);
		\draw [style=vt-bl] (10.center)
			 to [in=-180, out=-90, looseness=0.75] (12.center)
			 to [in=-90, out=0, looseness=0.75] (11.center);
		\draw [style=vt-bl-di] (12) to (3.center);
		\draw [style=vt-bl-da, in=90, out=180, looseness=0.75] (14.center) to (15.center);
		\draw [style=vt-bl-da, in=90, out=0, looseness=0.75] (16.center) to (17.center);
		\draw [style=vt] (1.center)
			 to [in=-180, out=-90, looseness=0.75] (3.center)
			 to [in=-90, out=0, looseness=0.75] (2.center)
			 to [in=0, out=90, looseness=0.75] (0.center)
			 to [in=90, out=180, looseness=0.75] cycle;
		\draw [style=vt] (5.center)
			 to [in=-180, out=-90, looseness=0.75] (7.center)
			 to [in=-90, out=0, looseness=0.75] (6.center);
		\draw [style=vt] (1.center) to (5.center);
		\draw [style=vt] (2.center) to (6.center);
		\draw [style=vt-bl-di] (18.center) to (19.center);
	\end{pgfonlayer}
\end{tikzpicture}}
  {%
  }}%
} 
  ~~\xmapsto{~~\df{\!_\ell}^{}~~}~~ \scalebox{0.6}{%
{\tikzstyle{every picture}=[tikzfig]
  {\begin{tikzpicture}
	\begin{pgfonlayer}{nodelayer}
		\node [style=none] (25) at (0, -0.5) {};
		\node [style=none] (26) at (-1.75, 0.25) {};
		\node [style=none] (29) at (-1.75, 3) {};
		\node [style=none] (30) at (1.75, 3) {};
		\node [style=none] (43) at (-1.75, 3) {};
		\node [style=none] (44) at (1.75, 3) {};
		\node [style=none] (14) at (-0.25, 1) {};
		\node [style=none] (15) at (-1.75, 0.25) {};
		\node [style=none] (16) at (0.25, 1) {};
		\node [style=none] (17) at (1.75, 0.25) {};
		\node [style=none] (18) at (1.3, -0.25) {};
		\node [style=none] (19) at (1.425, -0.225) {};
		\node [style=none] (46) at (-1.5, 2.575) {};
		\node [style=none] (47) at (-1, 2.45) {};
		\node [style=none] (48) at (-0.5, 2.275) {};
		\node [style=none] (49) at (0.5, 2.25) {};
		\node [style=none] (50) at (1.25, 2.45) {};
		\node [style=none] (52) at (-1.25, 1.75) {};
		\node [style=none] (53) at (-1, 1.25) {};
		\node [style=none] (54) at (0, 1.75) {};
		\node [style=none] (55) at (0, 1.5) {};
		\node [style=none] (56) at (-0.75, 0.5) {};
		\node [style=none] (57) at (-1.75, 1) {};
		\node [style=none] (58) at (-1.2, -0.3) {};
		\node [style=none] (59) at (0, 0.5) {};
		\node [style=none] (60) at (-0.5, 1.25) {};
		\node [style=none] (61) at (1.75, 2.25) {};
		\node [style=none] (62) at (1.025, 1.85) {};
		\node [style=none] (63) at (0.75, 1) {};
		\node [style=none] (64) at (0.5, 0.5) {};
		\node [style=none] (65) at (-1.25, 0.25) {};
		\node [style=none] (66) at (-0.5, -0.5) {};
		\node [style=none] (67) at (0, 0) {};
		\node [style=none] (68) at (-0.625, 0) {};
		\node [style=none] (69) at (0.5, 1.5) {};
		\node [style=none] (70) at (0, 0.25) {};
		\node [style=none] (71) at (0.65, -0.425) {};
		\node [style=none] (72) at (1.75, 0.75) {};
		\node [style=none] (73) at (1.75, 1.25) {};
		\node [style=none] (74) at (0.9, 1.45) {};
		\node [style=none] (75) at (-0.25, 1.5) {};
		\node [style=none] (76) at (0.8, 0.55) {};
		\node [style=none] (77) at (-0.975, -0.375) {};
		\node [style=none] (78) at (-1.15, -2.95) {};
		\node [style=none] (79) at (-0.5, -1.25) {};
		\node [style=none] (80) at (-1, -1.5) {};
		\node [style=none] (81) at (-1.75, -1.75) {};
		\node [style=none] (82) at (-0.575, -2.9) {};
		\node [style=none] (83) at (0, -2.5) {};
		\node [style=none] (84) at (-1.025, -3.6) {};
		\node [style=none] (85) at (0.75, -2) {};
		\node [style=none] (86) at (0.75, -1) {};
		\node [style=none] (87) at (1.75, -2.475) {};
		\node [style=none] (88) at (1, -2.5) {};
		\node [style=none] (89) at (0, -1.75) {};
		\node [style=none] (90) at (0, -3.25) {};
		\node [style=none] (91) at (1, -3.6) {};
		\node [style=none] (92) at (1.75, -1.25) {};
		\node [style=none] (93) at (-1.75, -0.5) {};
		\node [style=none] (94) at (-1.25, -1) {};
		\node [style=none] (95) at (-1.75, -2.75) {};
		\node [style=none] (96) at (-1.125, -2) {};
		\node [style=none] (97) at (-0.825, -2.475) {};
		\node [style=none] (98) at (0, -1.25) {};
		\node [style=none] (99) at (1.25, -1.15) {};
		\node [style=none] (100) at (0, -2) {};
		\node [style=none] (101) at (0, -2.75) {};
		\node [style=none] (102) at (0.75, -3) {};
		\node [style=none] (0) at (0, 3.75) {};
		\node [style=none] (1) at (-1.75, 3) {};
		\node [style=none] (2) at (1.75, 3) {};
		\node [style=none] (3) at (0, 2.25) {};
		\node [style=none] (5) at (-1.75, -3) {};
		\node [style=none] (6) at (1.75, -3) {};
		\node [style=none] (7) at (0, -3.75) {};
		\node [style=none] (103) at (1, -0.35) {};
		\node [style=none] (10) at (-1.75, 0.25) {};
		\node [style=none] (11) at (1.75, 0.25) {};
		\node [style=mid-nc] (12) at (0, -0.5) {};
	\end{pgfonlayer}
	\begin{pgfonlayer}{edgelayer}
		\draw [style=vt-bl-da, in=90, out=180, looseness=0.75] (14.center) to (15.center);
		\draw [style=vt-bl-da, in=90, out=0, looseness=0.75] (16.center) to (17.center);
		\draw [style=vt-bl-di] (18.center) to (19.center);
		\draw [style=vt-lbl] (46.center) to (52.center);
		\draw [style=vt-lbl] (52.center) to (47.center);
		\draw [style=vt-lbl] (52.center) to (57.center);
		\draw [style=vt-lbl] (52.center) to (53.center);
		\draw [style=vt-lbl] (53.center) to (60.center);
		\draw [style=vt-lbl] (60.center) to (54.center);
		\draw [style=vt-lbl] (60.center) to (59.center);
		\draw [style=vt-lbl] (53.center) to (56.center);
		\draw [style=vt-lbl] (57.center) to (65.center);
		\draw [style=vt-lbl] (65.center) to (56.center);
		\draw [style=vt-lbl] (65.center) to (58.center);
		\draw [style=vt-lbl] (56.center) to (66.center);
		\draw [style=vt-lbl] (68.center) to (67.center);
		\draw [style=vt-lbl] (55.center) to (69.center);
		\draw [style=vt-lbl] (69.center) to (49.center);
		\draw [style=vt-lbl] (69.center) to (63.center);
		\draw [style=vt-lbl] (63.center) to (64.center);
		\draw [style=vt-lbl] (63.center) to (62.center);
		\draw [style=vt-lbl] (62.center) to (50.center);
		\draw [style=vt-lbl] (62.center) to (61.center);
		\draw [style=vt-lbl] (70.center) to (64.center);
		\draw [style=vt-lbl] (64.center) to (72.center);
		\draw [style=vt-lbl] (74.center) to (73.center);
		\draw [style=vt-bl-di] (12) to (3.center);
		\draw [style=vt-lbl] (48.center) to (75.center);
		\draw [style=vt-lbl] (76.center) to (71.center);
		\draw [style=vt-lbl] (93.center) to (94.center);
		\draw [style=vt-lbl] (94.center) to (77.center);
		\draw [style=vt-lbl] (94.center) to (80.center);
		\draw [style=vt-lbl] (80.center) to (79.center);
		\draw [style=vt-lbl] (79.center) to (89.center);
		\draw [style=vt-lbl] (79.center) to (66.center);
		\draw [style=vt-lbl] (95.center) to (78.center);
		\draw [style=vt-lbl] (78.center) to (84.center);
		\draw [style=vt-lbl] (78.center) to (82.center);
		\draw [style=vt-lbl] (96.center) to (82.center);
		\draw [style=vt-lbl] (82.center) to (90.center);
		\draw [style=vt-lbl] (97.center) to (83.center);
		\draw [style=vt-lbl] (98.center) to (86.center);
		\draw [style=vt-lbl] (86.center) to (92.center);
		\draw [style=vt-lbl] (99.center) to (85.center);
		\draw [style=vt-lbl] (100.center) to (85.center);
		\draw [style=vt-lbl] (85.center) to (88.center);
		\draw [style=vt-lbl] (88.center) to (87.center);
		\draw [style=vt-lbl] (101.center) to (102.center);
		\draw [style=vt-lbl] (102.center) to (88.center);
		\draw [style=vt-lbl] (102.center) to (91.center);
		\draw [style=vt-bl-di] (7.center) to (12);
		\draw [style=vt] (1.center)
			 to [in=-180, out=-90, looseness=0.75] (3.center)
			 to [in=-90, out=0, looseness=0.75] (2.center)
			 to [in=0, out=90, looseness=0.75] (0.center)
			 to [in=90, out=180, looseness=0.75] cycle;
		\draw [style=vt] (5.center)
			 to [in=-180, out=-90, looseness=0.75] (7.center)
			 to [in=-90, out=0, looseness=0.75] (6.center);
		\draw [style=vt-lbl] (103.center) to (86.center);
		\draw [style=vt-bl] (10.center)
			 to [in=-180, out=-90, looseness=0.75] (12.center)
			 to [in=-90, out=0, looseness=0.75] (11.center);
		\draw [style=vt-lbl] (81.center) to (96.center);
		\draw [style=vt-lbl] (96.center) to (80.center);
		\draw [style=vt] (1.center) to (5.center);
		\draw [style=vt] (2.center) to (6.center);
	\end{pgfonlayer}
\end{tikzpicture}}
  {%
  }}%
}
  ~~\xmapsto{~~\imath_\ell^{}~~}~~ \scalebox{0.6}{%
{\tikzstyle{every picture}=[tikzfig]
  {\begin{tikzpicture}
	\begin{pgfonlayer}{nodelayer}
		\node [style=none] (25) at (0, -0.5) {};
		\node [style=none] (26) at (-1.75, 0.25) {};
		\node [style=none] (27) at (1.75, 0.25) {};
		\node [style=none] (28) at (0, 3.75) {};
		\node [style=none] (29) at (-1.75, 3) {};
		\node [style=none] (30) at (1.75, 3) {};
		\node [style=none] (42) at (0, 3.75) {};
		\node [style=none] (43) at (-1.75, 3) {};
		\node [style=none] (44) at (1.75, 3) {};
		\node [style=none] (45) at (0, 2.25) {};
		\node [style=none] (31) at (0, -0.5) {};
		\node [style=none] (32) at (-1.75, 0.25) {};
		\node [style=none] (35) at (1.75, 0.25) {};
		\node [style=none] (39) at (-1.75, -3) {};
		\node [style=none] (40) at (1.75, -3) {};
		\node [style=none] (41) at (0, -3.75) {};
		\node [style=none] (14) at (-0.25, 1) {};
		\node [style=none] (15) at (-1.75, 0.25) {};
		\node [style=none] (16) at (0.25, 1) {};
		\node [style=none] (17) at (1.75, 0.25) {};
		\node [style=none] (18) at (1.3, -0.25) {};
		\node [style=none] (19) at (1.425, -0.225) {};
		\node [style=none] (46) at (-1.5, 2.575) {};
		\node [style=none] (47) at (-1, 2.45) {};
		\node [style=none] (48) at (-0.5, 2.275) {};
		\node [style=none] (49) at (0.5, 2.25) {};
		\node [style=none] (50) at (1.25, 2.45) {};
		\node [style=none] (52) at (-1.25, 1.75) {};
		\node [style=none] (53) at (-1, 1.25) {};
		\node [style=none] (54) at (0, 1.75) {};
		\node [style=none] (55) at (0, 1.5) {};
		\node [style=none] (56) at (-0.75, 0.5) {};
		\node [style=none] (57) at (-1.75, 1) {};
		\node [style=none] (58) at (-1.2, -0.3) {};
		\node [style=none] (59) at (0, 0.5) {};
		\node [style=none] (60) at (-0.5, 1.25) {};
		\node [style=none] (61) at (1.75, 2.25) {};
		\node [style=none] (62) at (1.025, 1.85) {};
		\node [style=none] (63) at (0.75, 1) {};
		\node [style=none] (64) at (0.5, 0.5) {};
		\node [style=none] (65) at (-1.25, 0.25) {};
		\node [style=none] (66) at (-0.5, -0.5) {};
		\node [style=none] (67) at (0, 0) {};
		\node [style=none] (68) at (-0.625, 0) {};
		\node [style=none] (69) at (0.5, 1.5) {};
		\node [style=none] (70) at (0, 0.25) {};
		\node [style=none] (71) at (0.65, -0.425) {};
		\node [style=none] (72) at (1.75, 0.75) {};
		\node [style=none] (73) at (1.75, 1.25) {};
		\node [style=none] (74) at (0.9, 1.45) {};
		\node [style=none] (75) at (-0.25, 1.5) {};
		\node [style=none] (76) at (0.8, 0.55) {};
		\node [style=none] (77) at (-0.975, -0.375) {};
		\node [style=none] (78) at (-1.15, -2.95) {};
		\node [style=none] (79) at (-0.5, -1.25) {};
		\node [style=none] (80) at (-1, -1.5) {};
		\node [style=none] (81) at (-1.75, -1.75) {};
		\node [style=none] (82) at (-0.575, -2.9) {};
		\node [style=none] (83) at (0, -2.5) {};
		\node [style=none] (84) at (-1.025, -3.6) {};
		\node [style=none] (85) at (0.75, -2) {};
		\node [style=none] (86) at (0.75, -1) {};
		\node [style=none] (87) at (1.75, -2.475) {};
		\node [style=none] (88) at (1, -2.5) {};
		\node [style=none] (89) at (0, -1.75) {};
		\node [style=none] (90) at (0, -3.25) {};
		\node [style=none] (91) at (1, -3.6) {};
		\node [style=none] (92) at (1.75, -1.25) {};
		\node [style=none] (93) at (-1.75, -0.5) {};
		\node [style=none] (94) at (-1.25, -1) {};
		\node [style=none] (95) at (-1.75, -2.75) {};
		\node [style=none] (96) at (-1.125, -2) {};
		\node [style=none] (97) at (-0.825, -2.475) {};
		\node [style=none] (98) at (0, -1.25) {};
		\node [style=none] (99) at (1.25, -1.15) {};
		\node [style=none] (100) at (0, -2) {};
		\node [style=none] (101) at (0, -2.75) {};
		\node [style=none] (102) at (0.75, -3) {};
		\node [style=none] (0) at (0, 3.75) {};
		\node [style=none] (1) at (-1.75, 3) {};
		\node [style=none] (2) at (1.75, 3) {};
		\node [style=none] (3) at (0, 2.25) {};
		\node [style=none] (5) at (-1.75, -3) {};
		\node [style=none] (6) at (1.75, -3) {};
		\node [style=none] (7) at (0, -3.75) {};
		\node [style=none] (103) at (1, -0.35) {};
		\node [style=none] (10) at (-1.75, 0.25) {};
		\node [style=none] (11) at (1.75, 0.25) {};
		\node [style=mid-nc] (12) at (0, -0.5) {};
	\end{pgfonlayer}
	\begin{pgfonlayer}{edgelayer}
		\draw [style=fi-sh] (29.center)
			 to (26.center)
			 to [in=-180, out=-90, looseness=0.75] (25.center)
			 to [in=-90, out=0, looseness=0.75] (27.center)
			 to (30.center)
			 to [in=0, out=90, looseness=0.75] (28.center)
			 to [in=90, out=180, looseness=0.75] cycle;
		\draw [style=op-lsh] (43.center)
			 to [in=-180, out=-90, looseness=0.75] (45.center)
			 to [in=-90, out=0, looseness=0.75] (44.center)
			 to [in=0, out=90, looseness=0.75] (42.center)
			 to [in=90, out=180, looseness=0.75] cycle;
		\draw [style=fi-sh] (35.center)
			 to (40.center)
			 to [in=0, out=-90, looseness=0.75] (41.center)
			 to [in=-90, out=-180, looseness=0.75] (39.center)
			 to (32.center)
			 to [in=-180, out=-90, looseness=0.75] (31.center)
			 to [in=-90, out=0, looseness=0.75] cycle;
		\draw [style=vt-bl-da, in=90, out=180, looseness=0.75] (14.center) to (15.center);
		\draw [style=vt-bl-da, in=90, out=0, looseness=0.75] (16.center) to (17.center);
		\draw [style=vt-bl-di] (18.center) to (19.center);
		\draw [style=vt-lbl] (46.center) to (52.center);
		\draw [style=vt-lbl] (52.center) to (47.center);
		\draw [style=vt-lbl] (52.center) to (57.center);
		\draw [style=vt-lbl] (52.center) to (53.center);
		\draw [style=vt-lbl] (53.center) to (60.center);
		\draw [style=vt-lbl] (60.center) to (54.center);
		\draw [style=vt-lbl] (60.center) to (59.center);
		\draw [style=vt-lbl] (53.center) to (56.center);
		\draw [style=vt-lbl] (57.center) to (65.center);
		\draw [style=vt-lbl] (65.center) to (56.center);
		\draw [style=vt-lbl] (65.center) to (58.center);
		\draw [style=vt-lbl] (56.center) to (66.center);
		\draw [style=vt-lbl] (68.center) to (67.center);
		\draw [style=vt-lbl] (55.center) to (69.center);
		\draw [style=vt-lbl] (69.center) to (49.center);
		\draw [style=vt-lbl] (69.center) to (63.center);
		\draw [style=vt-lbl] (63.center) to (64.center);
		\draw [style=vt-lbl] (63.center) to (62.center);
		\draw [style=vt-lbl] (62.center) to (50.center);
		\draw [style=vt-lbl] (62.center) to (61.center);
		\draw [style=vt-lbl] (70.center) to (64.center);
		\draw [style=vt-lbl] (64.center) to (72.center);
		\draw [style=vt-lbl] (74.center) to (73.center);
		\draw [style=vt-bl-di] (12) to (3.center);
		\draw [style=vt-lbl] (48.center) to (75.center);
		\draw [style=vt-lbl] (76.center) to (71.center);
		\draw [style=vt-lbl] (93.center) to (94.center);
		\draw [style=vt-lbl] (94.center) to (77.center);
		\draw [style=vt-lbl] (94.center) to (80.center);
		\draw [style=vt-lbl] (80.center) to (79.center);
		\draw [style=vt-lbl] (79.center) to (89.center);
		\draw [style=vt-lbl] (79.center) to (66.center);
		\draw [style=vt-lbl] (95.center) to (78.center);
		\draw [style=vt-lbl] (78.center) to (84.center);
		\draw [style=vt-lbl] (78.center) to (82.center);
		\draw [style=vt-lbl] (96.center) to (82.center);
		\draw [style=vt-lbl] (82.center) to (90.center);
		\draw [style=vt-lbl] (97.center) to (83.center);
		\draw [style=vt-lbl] (98.center) to (86.center);
		\draw [style=vt-lbl] (86.center) to (92.center);
		\draw [style=vt-lbl] (99.center) to (85.center);
		\draw [style=vt-lbl] (100.center) to (85.center);
		\draw [style=vt-lbl] (85.center) to (88.center);
		\draw [style=vt-lbl] (88.center) to (87.center);
		\draw [style=vt-lbl] (101.center) to (102.center);
		\draw [style=vt-lbl] (102.center) to (88.center);
		\draw [style=vt-lbl] (102.center) to (91.center);
		\draw [style=vt-bl-di] (7.center) to (12);
		\draw [style=vt] (1.center)
			 to [in=-180, out=-90, looseness=0.75] (3.center)
			 to [in=-90, out=0, looseness=0.75] (2.center)
			 to [in=0, out=90, looseness=0.75] (0.center)
			 to [in=90, out=180, looseness=0.75] cycle;
		\draw [style=vt] (5.center)
			 to [in=-180, out=-90, looseness=0.75] (7.center)
			 to [in=-90, out=0, looseness=0.75] (6.center);
		\draw [style=vt-lbl] (103.center) to (86.center);
		\draw [style=vt-bl] (10.center)
			 to [in=-180, out=-90, looseness=0.75] (12.center)
			 to [in=-90, out=0, looseness=0.75] (11.center);
		\draw [style=vt-lbl] (81.center) to (96.center);
		\draw [style=vt-lbl] (96.center) to (80.center);
		\draw [style=vt] (1.center) to (5.center);
		\draw [style=vt] (2.center) to (6.center);
	\end{pgfonlayer}
\end{tikzpicture}}
  {%
  }}%
} 
  \label{eq:imath_.F}
  \ee
To emphasize that the morphism $\imath\circ\df_{\!\ell}(a)$ is an element in
$\SNfrc$ rather than in $\SNc$, the
2-cells in the picture on the right hand side of \eqref{eq:imath_.F} are colored
gray -- this is the color we use to represent the trivial Frobenius algebra
$\tu\iN\cfrc$. It directly follows from the properties of full Frobenius graphs
that the string net on the right hand side is an idempotent.

The counit in \eqref{eq:counit} is a morphism in the cylinder category $\SNfrc$ 
and can thus be defined in terms of a string net in $\SNfrc$ with appropriate
boundary values:
  \be
  (\varepsilon_\ell^{})_{a}^{} \Colon \imath_\ell^{}\circ\df{\!_\ell}^{}(a)
  ~=~~ \scalebox{0.6}{%
{\tikzstyle{every picture}=[tikzfig]
  {\begin{tikzpicture}
	\begin{pgfonlayer}{nodelayer}
		\node [style=none] (25) at (0, -0.5) {};
		\node [style=none] (26) at (-1.75, 0.25) {};
		\node [style=none] (27) at (1.75, 0.25) {};
		\node [style=none] (28) at (0, 3.75) {};
		\node [style=none] (29) at (-1.75, 3) {};
		\node [style=none] (30) at (1.75, 3) {};
		\node [style=none] (42) at (0, 3.75) {};
		\node [style=none] (43) at (-1.75, 3) {};
		\node [style=none] (44) at (1.75, 3) {};
		\node [style=none] (45) at (0, 2.25) {};
		\node [style=none] (31) at (0, -0.5) {};
		\node [style=none] (32) at (-1.75, 0.25) {};
		\node [style=none] (35) at (1.75, 0.25) {};
		\node [style=none] (39) at (-1.75, -3) {};
		\node [style=none] (40) at (1.75, -3) {};
		\node [style=none] (41) at (0, -3.75) {};
		\node [style=none] (14) at (-0.25, 1) {};
		\node [style=none] (15) at (-1.75, 0.25) {};
		\node [style=none] (16) at (0.25, 1) {};
		\node [style=none] (17) at (1.75, 0.25) {};
		\node [style=none] (18) at (1.3, -0.25) {};
		\node [style=none] (19) at (1.425, -0.225) {};
		\node [style=none] (46) at (-1.5, 2.575) {};
		\node [style=none] (47) at (-1, 2.45) {};
		\node [style=none] (48) at (-0.5, 2.275) {};
		\node [style=none] (49) at (0.5, 2.25) {};
		\node [style=none] (50) at (1.25, 2.45) {};
		\node [style=none] (52) at (-1.25, 1.75) {};
		\node [style=none] (53) at (-1, 1.25) {};
		\node [style=none] (54) at (0, 1.75) {};
		\node [style=none] (55) at (0, 1.5) {};
		\node [style=none] (56) at (-0.75, 0.5) {};
		\node [style=none] (57) at (-1.75, 1) {};
		\node [style=none] (58) at (-1.2, -0.3) {};
		\node [style=none] (59) at (0, 0.5) {};
		\node [style=none] (60) at (-0.5, 1.25) {};
		\node [style=none] (61) at (1.75, 2.25) {};
		\node [style=none] (62) at (1.025, 1.85) {};
		\node [style=none] (63) at (0.75, 1) {};
		\node [style=none] (64) at (0.5, 0.5) {};
		\node [style=none] (65) at (-1.25, 0.25) {};
		\node [style=none] (66) at (-0.5, -0.5) {};
		\node [style=none] (67) at (0, 0) {};
		\node [style=none] (68) at (-0.625, 0) {};
		\node [style=none] (69) at (0.5, 1.5) {};
		\node [style=none] (70) at (0, 0.25) {};
		\node [style=none] (71) at (0.65, -0.425) {};
		\node [style=none] (72) at (1.75, 0.75) {};
		\node [style=none] (73) at (1.75, 1.25) {};
		\node [style=none] (74) at (0.9, 1.45) {};
		\node [style=none] (75) at (-0.25, 1.5) {};
		\node [style=none] (76) at (0.8, 0.55) {};
		\node [style=none] (77) at (-0.975, -0.375) {};
		\node [style=none] (78) at (-1.15, -2.95) {};
		\node [style=none] (79) at (-0.5, -1.25) {};
		\node [style=none] (80) at (-1, -1.5) {};
		\node [style=none] (81) at (-1.75, -1.75) {};
		\node [style=none] (82) at (-0.575, -2.9) {};
		\node [style=none] (83) at (0, -2.5) {};
		\node [style=none] (84) at (-1.025, -3.6) {};
		\node [style=none] (85) at (0.75, -2) {};
		\node [style=none] (86) at (0.75, -1) {};
		\node [style=none] (87) at (1.75, -2.475) {};
		\node [style=none] (88) at (1, -2.5) {};
		\node [style=none] (89) at (0, -1.75) {};
		\node [style=none] (90) at (0, -3.25) {};
		\node [style=none] (91) at (1, -3.6) {};
		\node [style=none] (92) at (1.75, -1.25) {};
		\node [style=none] (93) at (-1.75, -0.5) {};
		\node [style=none] (94) at (-1.25, -1) {};
		\node [style=none] (95) at (-1.75, -2.75) {};
		\node [style=none] (96) at (-1.125, -2) {};
		\node [style=none] (97) at (-0.825, -2.475) {};
		\node [style=none] (98) at (0, -1.25) {};
		\node [style=none] (99) at (1.25, -1.15) {};
		\node [style=none] (100) at (0, -2) {};
		\node [style=none] (101) at (0, -2.75) {};
		\node [style=none] (102) at (0.75, -3) {};
		\node [style=none] (0) at (0, 3.75) {};
		\node [style=none] (1) at (-1.75, 3) {};
		\node [style=none] (2) at (1.75, 3) {};
		\node [style=none] (3) at (0, 2.25) {};
		\node [style=none] (5) at (-1.75, -3) {};
		\node [style=none] (6) at (1.75, -3) {};
		\node [style=none] (7) at (0, -3.75) {};
		\node [style=none] (103) at (1, -0.35) {};
		\node [style=none] (10) at (-1.75, 0.25) {};
		\node [style=none] (11) at (1.75, 0.25) {};
		\node [style=mid-nc] (12) at (0, -0.5) {};
	\end{pgfonlayer}
	\begin{pgfonlayer}{edgelayer}
		\draw [style=fi-sh] (29.center)
			 to (26.center)
			 to [in=-180, out=-90, looseness=0.75] (25.center)
			 to [in=-90, out=0, looseness=0.75] (27.center)
			 to (30.center)
			 to [in=0, out=90, looseness=0.75] (28.center)
			 to [in=90, out=180, looseness=0.75] cycle;
		\draw [style=op-lsh] (43.center)
			 to [in=-180, out=-90, looseness=0.75] (45.center)
			 to [in=-90, out=0, looseness=0.75] (44.center)
			 to [in=0, out=90, looseness=0.75] (42.center)
			 to [in=90, out=180, looseness=0.75] cycle;
		\draw [style=fi-sh] (35.center)
			 to (40.center)
			 to [in=0, out=-90, looseness=0.75] (41.center)
			 to [in=-90, out=-180, looseness=0.75] (39.center)
			 to (32.center)
			 to [in=-180, out=-90, looseness=0.75] (31.center)
			 to [in=-90, out=0, looseness=0.75] cycle;
		\draw [style=vt-bl-da, in=90, out=180, looseness=0.75] (14.center) to (15.center);
		\draw [style=vt-bl-da, in=90, out=0, looseness=0.75] (16.center) to (17.center);
		\draw [style=vt-bl-di] (18.center) to (19.center);
		\draw [style=vt-lbl] (46.center) to (52.center);
		\draw [style=vt-lbl] (52.center) to (47.center);
		\draw [style=vt-lbl] (52.center) to (57.center);
		\draw [style=vt-lbl] (52.center) to (53.center);
		\draw [style=vt-lbl] (53.center) to (60.center);
		\draw [style=vt-lbl] (60.center) to (54.center);
		\draw [style=vt-lbl] (60.center) to (59.center);
		\draw [style=vt-lbl] (53.center) to (56.center);
		\draw [style=vt-lbl] (57.center) to (65.center);
		\draw [style=vt-lbl] (65.center) to (56.center);
		\draw [style=vt-lbl] (65.center) to (58.center);
		\draw [style=vt-lbl] (56.center) to (66.center);
		\draw [style=vt-lbl] (68.center) to (67.center);
		\draw [style=vt-lbl] (55.center) to (69.center);
		\draw [style=vt-lbl] (69.center) to (49.center);
		\draw [style=vt-lbl] (69.center) to (63.center);
		\draw [style=vt-lbl] (63.center) to (64.center);
		\draw [style=vt-lbl] (63.center) to (62.center);
		\draw [style=vt-lbl] (62.center) to (50.center);
		\draw [style=vt-lbl] (62.center) to (61.center);
		\draw [style=vt-lbl] (70.center) to (64.center);
		\draw [style=vt-lbl] (64.center) to (72.center);
		\draw [style=vt-lbl] (74.center) to (73.center);
		\draw [style=vt-bl-di] (12) to (3.center);
		\draw [style=vt-lbl] (48.center) to (75.center);
		\draw [style=vt-lbl] (76.center) to (71.center);
		\draw [style=vt-lbl] (93.center) to (94.center);
		\draw [style=vt-lbl] (94.center) to (77.center);
		\draw [style=vt-lbl] (94.center) to (80.center);
		\draw [style=vt-lbl] (80.center) to (79.center);
		\draw [style=vt-lbl] (79.center) to (89.center);
		\draw [style=vt-lbl] (79.center) to (66.center);
		\draw [style=vt-lbl] (95.center) to (78.center);
		\draw [style=vt-lbl] (78.center) to (84.center);
		\draw [style=vt-lbl] (78.center) to (82.center);
		\draw [style=vt-lbl] (96.center) to (82.center);
		\draw [style=vt-lbl] (82.center) to (90.center);
		\draw [style=vt-lbl] (97.center) to (83.center);
		\draw [style=vt-lbl] (98.center) to (86.center);
		\draw [style=vt-lbl] (86.center) to (92.center);
		\draw [style=vt-lbl] (99.center) to (85.center);
		\draw [style=vt-lbl] (100.center) to (85.center);
		\draw [style=vt-lbl] (85.center) to (88.center);
		\draw [style=vt-lbl] (88.center) to (87.center);
		\draw [style=vt-lbl] (101.center) to (102.center);
		\draw [style=vt-lbl] (102.center) to (88.center);
		\draw [style=vt-lbl] (102.center) to (91.center);
		\draw [style=vt-bl-di] (7.center) to (12);
		\draw [style=vt] (1.center)
			 to [in=-180, out=-90, looseness=0.75] (3.center)
			 to [in=-90, out=0, looseness=0.75] (2.center)
			 to [in=0, out=90, looseness=0.75] (0.center)
			 to [in=90, out=180, looseness=0.75] cycle;
		\draw [style=vt] (5.center)
			 to [in=-180, out=-90, looseness=0.75] (7.center)
			 to [in=-90, out=0, looseness=0.75] (6.center);
		\draw [style=vt-lbl] (103.center) to (86.center);
		\draw [style=vt-bl] (10.center)
			 to [in=-180, out=-90, looseness=0.75] (12.center)
			 to [in=-90, out=0, looseness=0.75] (11.center);
		\draw [style=vt-lbl] (81.center) to (96.center);
		\draw [style=vt-lbl] (96.center) to (80.center);
		\draw [style=vt] (1.center) to (5.center);
		\draw [style=vt] (2.center) to (6.center);
	\end{pgfonlayer}
\end{tikzpicture}}
  {%
  }}%
}~
  \xrightarrow[\cong]{\qquad{\scalebox{0.7}{%
{\tikzstyle{every picture}=[tikzfig]
  {\begin{tikzpicture}
	\begin{pgfonlayer}{nodelayer}
		\node [style=none] (25) at (0.5, 0) {};
		\node [style=none] (26) at (-0.25, 1.75) {};
		\node [style=none] (27) at (-0.25, -1.75) {};
		\node [style=none] (28) at (-3.75, 0) {};
		\node [style=none] (29) at (-3, 1.75) {};
		\node [style=none] (30) at (-3, -1.75) {};
		\node [style=none] (42) at (-3.75, 0) {};
		\node [style=none] (43) at (-3, 1.75) {};
		\node [style=none] (44) at (-3, -1.75) {};
		\node [style=none] (45) at (-2.25, 0) {};
		\node [style=none] (31) at (0.5, 0) {};
		\node [style=none] (32) at (-0.25, 1.75) {};
		\node [style=none] (35) at (-0.25, -1.75) {};
		\node [style=none] (39) at (3, 1.75) {};
		\node [style=none] (40) at (3, -1.75) {};
		\node [style=none] (41) at (3.75, 0) {};
		\node [style=none] (14) at (-1, 0.25) {};
		\node [style=none] (15) at (-0.25, 1.75) {};
		\node [style=none] (16) at (-1, -0.25) {};
		\node [style=none] (17) at (-0.25, -1.75) {};
		\node [style=none] (18) at (0.25, -1.3) {};
		\node [style=none] (19) at (0.225, -1.425) {};
		\node [style=none] (46) at (-2.575, 1.5) {};
		\node [style=none] (47) at (-2.45, 1) {};
		\node [style=none] (48) at (-2.275, 0.5) {};
		\node [style=none] (49) at (-2.25, -0.5) {};
		\node [style=none] (50) at (-2.45, -1.25) {};
		\node [style=none] (52) at (-1.75, 1.25) {};
		\node [style=none] (53) at (-1.25, 1) {};
		\node [style=none] (54) at (-1.75, 0) {};
		\node [style=none] (55) at (-1.5, 0) {};
		\node [style=none] (56) at (-0.5, 0.75) {};
		\node [style=none] (57) at (-1, 1.75) {};
		\node [style=none] (58) at (0.3, 1.2) {};
		\node [style=none] (59) at (-0.5, 0) {};
		\node [style=none] (60) at (-1.25, 0.5) {};
		\node [style=none] (61) at (-2.25, -1.75) {};
		\node [style=none] (62) at (-1.85, -1.025) {};
		\node [style=none] (63) at (-1, -0.75) {};
		\node [style=none] (64) at (-0.5, -0.5) {};
		\node [style=none] (65) at (-0.25, 1.25) {};
		\node [style=none] (66) at (0.5, 0.5) {};
		\node [style=none] (67) at (0, 0) {};
		\node [style=none] (68) at (0, 0.625) {};
		\node [style=none] (69) at (-1.5, -0.5) {};
		\node [style=none] (70) at (-0.25, 0) {};
		\node [style=none] (71) at (0.425, -0.65) {};
		\node [style=none] (72) at (-0.75, -1.75) {};
		\node [style=none] (73) at (-1.25, -1.75) {};
		\node [style=none] (74) at (-1.45, -0.9) {};
		\node [style=none] (75) at (-1.5, 0.25) {};
		\node [style=none] (76) at (-0.55, -0.8) {};
		\node [style=none] (0) at (-3.75, 0) {};
		\node [style=none] (1) at (-3, 1.75) {};
		\node [style=none] (2) at (-3, -1.75) {};
		\node [style=none] (3) at (-2.25, 0) {};
		\node [style=none] (5) at (3, 1.75) {};
		\node [style=none] (6) at (3, -1.75) {};
		\node [style=none] (7) at (3.75, 0) {};
		\node [style=none] (10) at (-0.25, 1.75) {};
		\node [style=none] (11) at (-0.25, -1.75) {};
		\node [style=mid-nc] (12) at (0.5, 0) {};
	\end{pgfonlayer}
	\begin{pgfonlayer}{edgelayer}
		\draw [style=fi-sh] (29.center)
			 to (26.center)
			 to [in=90, out=0, looseness=0.75] (25.center)
			 to [in=0, out=-90, looseness=0.75] (27.center)
			 to (30.center)
			 to [in=-90, out=-180, looseness=0.75] (28.center)
			 to [in=-180, out=90, looseness=0.75] cycle;
		\draw [style=op-lsh] (43.center)
			 to [in=90, out=0, looseness=0.75] (45.center)
			 to [in=0, out=-90, looseness=0.75] (44.center)
			 to [in=-90, out=-180, looseness=0.75] (42.center)
			 to [in=-180, out=90, looseness=0.75] cycle;
		\draw [style=fi-bl] (35.center)
			 to (40.center)
			 to [in=-90, out=0, looseness=0.75] (41.center)
			 to [in=0, out=90, looseness=0.75] (39.center)
			 to (32.center)
			 to [in=90, out=0, looseness=0.75] (31.center)
			 to [in=0, out=-90, looseness=0.75] cycle;
		\draw [style=vt-bl-da, in=-180, out=90, looseness=0.75] (14.center) to (15.center);
		\draw [style=vt-bl-da, in=-180, out=-90, looseness=0.75] (16.center) to (17.center);
		\draw [style=vt-bl-di] (18.center) to (19.center);
		\draw [style=vt-lbl] (46.center) to (52.center);
		\draw [style=vt-lbl] (52.center) to (47.center);
		\draw [style=vt-lbl] (52.center) to (57.center);
		\draw [style=vt-lbl] (52.center) to (53.center);
		\draw [style=vt-lbl] (53.center) to (60.center);
		\draw [style=vt-lbl] (60.center) to (54.center);
		\draw [style=vt-lbl] (60.center) to (59.center);
		\draw [style=vt-lbl] (53.center) to (56.center);
		\draw [style=vt-lbl] (57.center) to (65.center);
		\draw [style=vt-lbl] (65.center) to (56.center);
		\draw [style=vt-lbl] (65.center) to (58.center);
		\draw [style=vt-lbl] (56.center) to (66.center);
		\draw [style=vt-lbl] (68.center) to (67.center);
		\draw [style=vt-lbl] (55.center) to (69.center);
		\draw [style=vt-lbl] (69.center) to (49.center);
		\draw [style=vt-lbl] (69.center) to (63.center);
		\draw [style=vt-lbl] (63.center) to (64.center);
		\draw [style=vt-lbl] (63.center) to (62.center);
		\draw [style=vt-lbl] (62.center) to (50.center);
		\draw [style=vt-lbl] (62.center) to (61.center);
		\draw [style=vt-lbl] (70.center) to (64.center);
		\draw [style=vt-lbl] (64.center) to (72.center);
		\draw [style=vt-lbl] (74.center) to (73.center);
		\draw [style=vt-bl-di] (3.center) to (12);
		\draw [style=vt-lbl] (48.center) to (75.center);
		\draw [style=vt-lbl] (76.center) to (71.center);
		\draw [style=vt-bl-di] (12) to (7.center);
		\draw [style=vt] (1.center)
			 to [in=90, out=0, looseness=0.75] (3.center)
			 to [in=0, out=-90, looseness=0.75] (2.center)
			 to [in=-90, out=-180, looseness=0.75] (0.center)
			 to [in=-180, out=90, looseness=0.75] cycle;
		\draw [style=vt] (5.center)
			 to [in=90, out=0, looseness=0.75] (7.center)
			 to [in=0, out=-90, looseness=0.75] (6.center);
		\draw [style=vt-bl] (10.center)
			 to [in=90, out=0, looseness=0.75] (12.center)
			 to [in=0, out=-90, looseness=0.75] (11.center);
		\draw [style=vt] (1.center) to (5.center);
		\draw [style=vt] (2.center) to (6.center);
	\end{pgfonlayer}
\end{tikzpicture}}
  {%
  }}%
}}_{\phantom|}\qquad}~
  \scalebox{0.6}{%
{\tikzstyle{every picture}=[tikzfig]
  {\begin{tikzpicture}
	\begin{pgfonlayer}{nodelayer}
		\node [style=none] (25) at (0, -0.5) {};
		\node [style=none] (26) at (-1.75, 0.25) {};
		\node [style=none] (27) at (1.75, 0.25) {};
		\node [style=none] (28) at (0, 3.75) {};
		\node [style=none] (29) at (-1.75, 3) {};
		\node [style=none] (30) at (1.75, 3) {};
		\node [style=none] (45) at (0, 3.75) {};
		\node [style=none] (46) at (-1.75, 3) {};
		\node [style=none] (47) at (1.75, 3) {};
		\node [style=none] (48) at (0, 2.25) {};
		\node [style=none] (31) at (0, -0.5) {};
		\node [style=none] (32) at (-1.75, 0.25) {};
		\node [style=none] (35) at (1.75, 0.25) {};
		\node [style=none] (39) at (-1.75, -3) {};
		\node [style=none] (40) at (1.75, -3) {};
		\node [style=none] (41) at (0, -3.75) {};
		\node [style=none] (10) at (-1.75, 0.25) {};
		\node [style=none] (11) at (1.75, 0.25) {};
		\node [style=mid-nc] (12) at (0, -0.5) {};
		\node [style=none] (14) at (-0.25, 1) {};
		\node [style=none] (15) at (-1.75, 0.25) {};
		\node [style=none] (16) at (0.25, 1) {};
		\node [style=none] (17) at (1.75, 0.25) {};
		\node [style=none] (0) at (0, 3.75) {};
		\node [style=none] (1) at (-1.75, 3) {};
		\node [style=none] (2) at (1.75, 3) {};
		\node [style=none] (3) at (0, 2.25) {};
		\node [style=none] (5) at (-1.75, -3) {};
		\node [style=none] (6) at (1.75, -3) {};
		\node [style=none] (7) at (0, -3.75) {};
		\node [style=none] (18) at (1.3, -0.25) {};
		\node [style=none] (19) at (1.425, -0.225) {};
	\end{pgfonlayer}
	\begin{pgfonlayer}{edgelayer}
		\draw [style=fi-bl] (29.center)
			 to (26.center)
			 to [in=-180, out=-90, looseness=0.75] (25.center)
			 to [in=-90, out=0, looseness=0.75] (27.center)
			 to (30.center)
			 to [in=0, out=90, looseness=0.75] (28.center)
			 to [in=90, out=180, looseness=0.75] cycle;
		\draw [style=op-lsh] (46.center)
			 to [in=-180, out=-90, looseness=0.75] (48.center)
			 to [in=-90, out=0, looseness=0.75] (47.center)
			 to [in=0, out=90, looseness=0.75] (45.center)
			 to [in=90, out=180, looseness=0.75] cycle;
		\draw [style=fi-bl] (35.center)
			 to (40.center)
			 to [in=0, out=-90, looseness=0.75] (41.center)
			 to [in=-90, out=-180, looseness=0.75] (39.center)
			 to (32.center)
			 to [in=-180, out=-90, looseness=0.75] (31.center)
			 to [in=-90, out=0, looseness=0.75] cycle;
		\draw [style=vt-bl-di] (7.center) to (12);
		\draw [style=vt-bl] (10.center)
			 to [in=-180, out=-90, looseness=0.75] (12.center)
			 to [in=-90, out=0, looseness=0.75] (11.center);
		\draw [style=vt-bl-di] (12) to (3.center);
		\draw [style=vt-bl-da, in=90, out=180, looseness=0.75] (14.center) to (15.center);
		\draw [style=vt-bl-da, in=90, out=0, looseness=0.75] (16.center) to (17.center);
		\draw [style=vt] (1.center)
			 to [in=-180, out=-90, looseness=0.75] (3.center)
			 to [in=-90, out=0, looseness=0.75] (2.center)
			 to [in=0, out=90, looseness=0.75] (0.center)
			 to [in=90, out=180, looseness=0.75] cycle;
		\draw [style=vt] (5.center)
			 to [in=-180, out=-90, looseness=0.75] (7.center)
			 to [in=-90, out=0, looseness=0.75] (6.center);
		\draw [style=vt] (1.center) to (5.center);
		\draw [style=vt] (2.center) to (6.center);
		\draw [style=vt-bl-di] (18.center) to (19.center);
	\end{pgfonlayer}
\end{tikzpicture}}
  {%
  }}%
} ~~=~a \,.
  \label{eq:counit_a}
  \ee
Let us explain this prescription in more detail: The idempotents that specify the 
objects $\imath_\ell^{}\circ\df_{\!\ell}(a)$ and $a$, respectively, are drawn 
vertically, while in order to facilitate the visualization, the string net that
represents the morphism $(\varepsilon_\ell^{})_{a}^{}$ is drawn horizontally, 
aligned along the arrow.  The latter string net is obtained by modifying the 
string net representing the idempotent for the object $a$ in a similar way as in 
the prescription \eqref{eq:imath_.F}
for $\imath_\ell^{}\circ\df_{\!\ell}(a)$, but now instead of turning the entirety of
$a$ into a string-net with 2-cells labeled by the trivial Frobenius algebra, the
procedure is only applied to the half of $a$ that starts at its domain. (It is also
understood that the label for the vertex is only partially $u$-conjugated, with the 
label of the non-contractible circle being the same as the one in the idempotent 
for $a$.)

Using the fact that the forgetful functor \eqref{eq:ufr} from $\cfrc$ to $\cbc$ 
preserves the graphical calculus up to the insertion of Frobenius graphs 
\Cite{Ex.\ 2.29}{fusY2} and an argument similar to showing that $a$ (and likewise,
$\imath_\ell^{}\circ\df{\!_\ell}^{}(a)$, is an idempotent, one sees that 
the morphism $(\varepsilon_\ell)_{a}^{}$ displayed in \eqref{eq:counit_a}
that is obtained in this way is well-defined and has the correct domain and codomain.
For the very same reason, $(\varepsilon_\ell)_{a}^{}$ is an \emph{iso}morphism: the 
inverse of $a$ is given by its `transpose', i.e.\ by the string net that 
results from applying the modification procedure 
to the other half of $a$, i.e.\ to the one that ends at its codomain.

Next we establish naturality of the family $\{ (\varepsilon_\ell^{})_{a}^{} \}$ of
morphisms. To this end we consider generic morphisms in $\SNfrc(\ell)$ between 
any two objects $a$ and $b$ of the form \eqref{eq:pic:a}; such a morphism looks like
(again suppressing all labels)
  \be
  f \Colon a ~=~ \scalebox{0.6}{%
{\tikzstyle{every picture}=[tikzfig]
  {\begin{tikzpicture}
	\begin{pgfonlayer}{nodelayer}
		\node [style=none] (25) at (0, -0.5) {};
		\node [style=none] (26) at (-1.75, 0.25) {};
		\node [style=none] (27) at (1.75, 0.25) {};
		\node [style=none] (28) at (0, 3.75) {};
		\node [style=none] (29) at (-1.75, 3) {};
		\node [style=none] (30) at (1.75, 3) {};
		\node [style=none] (45) at (0, 3.75) {};
		\node [style=none] (46) at (-1.75, 3) {};
		\node [style=none] (47) at (1.75, 3) {};
		\node [style=none] (48) at (0, 2.25) {};
		\node [style=none] (31) at (0, -0.5) {};
		\node [style=none] (32) at (-1.75, 0.25) {};
		\node [style=none] (35) at (1.75, 0.25) {};
		\node [style=none] (39) at (-1.75, -3) {};
		\node [style=none] (40) at (1.75, -3) {};
		\node [style=none] (41) at (0, -3.75) {};
		\node [style=none] (10) at (-1.75, 0.25) {};
		\node [style=none] (11) at (1.75, 0.25) {};
		\node [style=mid-nc] (12) at (0, -0.5) {};
		\node [style=none] (14) at (-0.25, 1) {};
		\node [style=none] (15) at (-1.75, 0.25) {};
		\node [style=none] (16) at (0.25, 1) {};
		\node [style=none] (17) at (1.75, 0.25) {};
		\node [style=none] (0) at (0, 3.75) {};
		\node [style=none] (1) at (-1.75, 3) {};
		\node [style=none] (2) at (1.75, 3) {};
		\node [style=none] (3) at (0, 2.25) {};
		\node [style=none] (5) at (-1.75, -3) {};
		\node [style=none] (6) at (1.75, -3) {};
		\node [style=none] (7) at (0, -3.75) {};
		\node [style=none] (18) at (1.3, -0.25) {};
		\node [style=none] (19) at (1.425, -0.225) {};
	\end{pgfonlayer}
	\begin{pgfonlayer}{edgelayer}
		\draw [style=fi-bl] (29.center)
			 to (26.center)
			 to [in=-180, out=-90, looseness=0.75] (25.center)
			 to [in=-90, out=0, looseness=0.75] (27.center)
			 to (30.center)
			 to [in=0, out=90, looseness=0.75] (28.center)
			 to [in=90, out=180, looseness=0.75] cycle;
		\draw [style=op-lsh] (46.center)
			 to [in=-180, out=-90, looseness=0.75] (48.center)
			 to [in=-90, out=0, looseness=0.75] (47.center)
			 to [in=0, out=90, looseness=0.75] (45.center)
			 to [in=90, out=180, looseness=0.75] cycle;
		\draw [style=fi-bl] (35.center)
			 to (40.center)
			 to [in=0, out=-90, looseness=0.75] (41.center)
			 to [in=-90, out=-180, looseness=0.75] (39.center)
			 to (32.center)
			 to [in=-180, out=-90, looseness=0.75] (31.center)
			 to [in=-90, out=0, looseness=0.75] cycle;
		\draw [style=vt-bl-di] (7.center) to (12);
		\draw [style=vt-bl] (10.center)
			 to [in=-180, out=-90, looseness=0.75] (12.center)
			 to [in=-90, out=0, looseness=0.75] (11.center);
		\draw [style=vt-bl-di] (12) to (3.center);
		\draw [style=vt-bl-da, in=90, out=180, looseness=0.75] (14.center) to (15.center);
		\draw [style=vt-bl-da, in=90, out=0, looseness=0.75] (16.center) to (17.center);
		\draw [style=vt] (1.center)
			 to [in=-180, out=-90, looseness=0.75] (3.center)
			 to [in=-90, out=0, looseness=0.75] (2.center)
			 to [in=0, out=90, looseness=0.75] (0.center)
			 to [in=90, out=180, looseness=0.75] cycle;
		\draw [style=vt] (5.center)
			 to [in=-180, out=-90, looseness=0.75] (7.center)
			 to [in=-90, out=0, looseness=0.75] (6.center);
		\draw [style=vt] (1.center) to (5.center);
		\draw [style=vt] (2.center) to (6.center);
		\draw [style=vt-bl-di] (18.center) to (19.center);
	\end{pgfonlayer}
\end{tikzpicture}}
  {%
  }}%
}
  ~\xrightarrow{\qquad{\scalebox{0.7}{%
{\tikzstyle{every picture}=[tikzfig]
  {\begin{tikzpicture}
	\begin{pgfonlayer}{nodelayer}
		\node [style=none] (25) at (0.5, 0) {};
		\node [style=none] (26) at (-0.25, 1.75) {};
		\node [style=none] (27) at (-0.25, -1.75) {};
		\node [style=none] (28) at (-3.75, 0) {};
		\node [style=none] (29) at (-3, 1.75) {};
		\node [style=none] (30) at (-3, -1.75) {};
		\node [style=none] (45) at (-3.75, 0) {};
		\node [style=none] (46) at (-3, 1.75) {};
		\node [style=none] (47) at (-3, -1.75) {};
		\node [style=none] (48) at (-2.25, 0) {};
		\node [style=none] (31) at (0.5, 0) {};
		\node [style=none] (32) at (-0.25, 1.75) {};
		\node [style=none] (35) at (-0.25, -1.75) {};
		\node [style=none] (39) at (3, 1.75) {};
		\node [style=none] (40) at (3, -1.75) {};
		\node [style=none] (41) at (3.75, 0) {};
		\node [style=none] (10) at (-0.25, 1.75) {};
		\node [style=none] (11) at (-0.25, -1.75) {};
		\node [style=mid-nc] (12) at (0.5, 0) {};
		\node [style=none] (14) at (-1, 0.25) {};
		\node [style=none] (15) at (-0.25, 1.75) {};
		\node [style=none] (16) at (-1, -0.25) {};
		\node [style=none] (17) at (-0.25, -1.75) {};
		\node [style=none] (0) at (-3.75, 0) {};
		\node [style=none] (1) at (-3, 1.75) {};
		\node [style=none] (2) at (-3, -1.75) {};
		\node [style=none] (3) at (-2.25, 0) {};
		\node [style=none] (5) at (3, 1.75) {};
		\node [style=none] (6) at (3, -1.75) {};
		\node [style=none] (7) at (3.75, 0) {};
		\node [style=none] (18) at (0.25, -1.3) {};
		\node [style=none] (19) at (0.225, -1.425) {};
	\end{pgfonlayer}
	\begin{pgfonlayer}{edgelayer}
		\draw [style=fi-bl] (29.center)
			 to (26.center)
			 to [in=90, out=0, looseness=0.75] (25.center)
			 to [in=0, out=-90, looseness=0.75] (27.center)
			 to (30.center)
			 to [in=-90, out=-180, looseness=0.75] (28.center)
			 to [in=-180, out=90, looseness=0.75] cycle;
		\draw [style=op-lsh] (46.center)
			 to [in=90, out=0, looseness=0.75] (48.center)
			 to [in=0, out=-90, looseness=0.75] (47.center)
			 to [in=-90, out=-180, looseness=0.75] (45.center)
			 to [in=-180, out=90, looseness=0.75] cycle;
		\draw [style=fi-gr] (35.center)
			 to (40.center)
			 to [in=-90, out=0, looseness=0.75] (41.center)
			 to [in=0, out=90, looseness=0.75] (39.center)
			 to (32.center)
			 to [in=90, out=0, looseness=0.75] (31.center)
			 to [in=0, out=-90, looseness=0.75] cycle;
		\draw [style=vt-bl-di] (12) to (7.center);
		\draw [style=vt-bl] (10.center)
			 to [in=90, out=0, looseness=0.75] (12.center)
			 to [in=0, out=-90, looseness=0.75] (11.center);
		\draw [style=vt-bl-di] (3.center) to (12);
		\draw [style=vt-bl-da, in=-180, out=90, looseness=0.75] (14.center) to (15.center);
		\draw [style=vt-bl-da, in=-180, out=-90, looseness=0.75] (16.center) to (17.center);
		\draw [style=vt] (1.center)
			 to [in=90, out=0, looseness=0.75] (3.center)
			 to [in=0, out=-90, looseness=0.75] (2.center)
			 to [in=-90, out=-180, looseness=0.75] (0.center)
			 to [in=-180, out=90, looseness=0.75] cycle;
		\draw [style=vt] (5.center)
			 to [in=90, out=0, looseness=0.75] (7.center)
			 to [in=0, out=-90, looseness=0.75] (6.center);
		\draw [style=vt] (1.center) to (5.center);
		\draw [style=vt] (2.center) to (6.center);
		\draw [style=vt-bl-di] (18.center) to (19.center);
	\end{pgfonlayer}
\end{tikzpicture}}
  {%
  }}%
}}_{\phantom|}\qquad}~
  \scalebox{0.6}{%
{\tikzstyle{every picture}=[tikzfig]
  {\begin{tikzpicture}
	\begin{pgfonlayer}{nodelayer}
		\node [style=none] (25) at (0, -0.5) {};
		\node [style=none] (26) at (-1.75, 0.25) {};
		\node [style=none] (27) at (1.75, 0.25) {};
		\node [style=none] (28) at (0, 3.75) {};
		\node [style=none] (29) at (-1.75, 3) {};
		\node [style=none] (30) at (1.75, 3) {};
		\node [style=none] (45) at (0, 3.75) {};
		\node [style=none] (46) at (-1.75, 3) {};
		\node [style=none] (47) at (1.75, 3) {};
		\node [style=none] (48) at (0, 2.25) {};
		\node [style=none] (31) at (0, -0.5) {};
		\node [style=none] (32) at (-1.75, 0.25) {};
		\node [style=none] (35) at (1.75, 0.25) {};
		\node [style=none] (39) at (-1.75, -3) {};
		\node [style=none] (40) at (1.75, -3) {};
		\node [style=none] (41) at (0, -3.75) {};
		\node [style=none] (10) at (-1.75, 0.25) {};
		\node [style=none] (11) at (1.75, 0.25) {};
		\node [style=mid-nc] (12) at (0, -0.5) {};
		\node [style=none] (14) at (-0.25, 1) {};
		\node [style=none] (15) at (-1.75, 0.25) {};
		\node [style=none] (16) at (0.25, 1) {};
		\node [style=none] (17) at (1.75, 0.25) {};
		\node [style=none] (0) at (0, 3.75) {};
		\node [style=none] (1) at (-1.75, 3) {};
		\node [style=none] (2) at (1.75, 3) {};
		\node [style=none] (3) at (0, 2.25) {};
		\node [style=none] (5) at (-1.75, -3) {};
		\node [style=none] (6) at (1.75, -3) {};
		\node [style=none] (7) at (0, -3.75) {};
		\node [style=none] (18) at (1.3, -0.25) {};
		\node [style=none] (19) at (1.425, -0.225) {};
	\end{pgfonlayer}
	\begin{pgfonlayer}{edgelayer}
		\draw [style=fi-gr] (29.center)
			 to (26.center)
			 to [in=-180, out=-90, looseness=0.75] (25.center)
			 to [in=-90, out=0, looseness=0.75] (27.center)
			 to (30.center)
			 to [in=0, out=90, looseness=0.75] (28.center)
			 to [in=90, out=180, looseness=0.75] cycle;
		\draw [style=op-lsh] (46.center)
			 to [in=-180, out=-90, looseness=0.75] (48.center)
			 to [in=-90, out=0, looseness=0.75] (47.center)
			 to [in=0, out=90, looseness=0.75] (45.center)
			 to [in=90, out=180, looseness=0.75] cycle;
		\draw [style=fi-gr] (35.center)
			 to (40.center)
			 to [in=0, out=-90, looseness=0.75] (41.center)
			 to [in=-90, out=-180, looseness=0.75] (39.center)
			 to (32.center)
			 to [in=-180, out=-90, looseness=0.75] (31.center)
			 to [in=-90, out=0, looseness=0.75] cycle;
		\draw [style=vt-bl-di] (7.center) to (12);
		\draw [style=vt-bl] (10.center)
			 to [in=-180, out=-90, looseness=0.75] (12.center)
			 to [in=-90, out=0, looseness=0.75] (11.center);
		\draw [style=vt-bl-di] (12) to (3.center);
		\draw [style=vt-bl-da, in=90, out=180, looseness=0.75] (14.center) to (15.center);
		\draw [style=vt-bl-da, in=90, out=0, looseness=0.75] (16.center) to (17.center);
		\draw [style=vt] (1.center)
			 to [in=-180, out=-90, looseness=0.75] (3.center)
			 to [in=-90, out=0, looseness=0.75] (2.center)
			 to [in=0, out=90, looseness=0.75] (0.center)
			 to [in=90, out=180, looseness=0.75] cycle;
		\draw [style=vt] (5.center)
			 to [in=-180, out=-90, looseness=0.75] (7.center)
			 to [in=-90, out=0, looseness=0.75] (6.center);
		\draw [style=vt] (1.center) to (5.center);
		\draw [style=vt] (2.center) to (6.center);
		\draw [style=vt-bl-di] (18.center) to (19.center);
	\end{pgfonlayer}
\end{tikzpicture}}
  {%
  }}%
} ~=~ b
  \ee
Note that this is indeed a morphism in $\SNfrc(\ell)$, since, by definition of the 
idempotent completion, we have
  \be
  f\cir a =  f = b\cir f \,.
  \label{eq:fa=f=bf}
  \ee
The naturality of the counit at the morphism $f$ is now seen as follows:
  \be
  (\varepsilon_\ell^{})_{b}^{} \circ (\imath_\ell^{}\cir \df_{\!\ell}^{}(f)) 
  ~=~ \scalebox{0.5}{%
{\tikzstyle{every picture}=[tikzfig]
  {\begin{tikzpicture}
	\begin{pgfonlayer}{nodelayer}
		\node [style=none] (25) at (0, -0.5) {};
		\node [style=none] (26) at (-1.75, 0.25) {};
		\node [style=none] (27) at (1.75, 0.25) {};
		\node [style=none] (28) at (0, 3.75) {};
		\node [style=none] (29) at (-1.75, 3) {};
		\node [style=none] (30) at (1.75, 3) {};
		\node [style=none] (42) at (0, 3.75) {};
		\node [style=none] (43) at (-1.75, 3) {};
		\node [style=none] (44) at (1.75, 3) {};
		\node [style=none] (45) at (0, 2.25) {};
		\node [style=none] (31) at (0, -0.5) {};
		\node [style=none] (32) at (-1.75, 0.25) {};
		\node [style=none] (35) at (1.75, 0.25) {};
		\node [style=none] (39) at (-1.75, -3) {};
		\node [style=none] (40) at (1.75, -3) {};
		\node [style=none] (41) at (0, -3.75) {};
		\node [style=none] (14) at (-0.25, 1) {};
		\node [style=none] (15) at (-1.75, 0.25) {};
		\node [style=none] (16) at (0.25, 1) {};
		\node [style=none] (17) at (1.75, 0.25) {};
		\node [style=none] (18) at (1.3, -0.25) {};
		\node [style=none] (19) at (1.425, -0.225) {};
		\node [style=none] (66) at (-0.5, -0.5) {};
		\node [style=none] (77) at (-0.975, -0.375) {};
		\node [style=none] (78) at (-1.15, -2.95) {};
		\node [style=none] (79) at (-0.5, -1.25) {};
		\node [style=none] (80) at (-1, -1.5) {};
		\node [style=none] (81) at (-1.75, -1.75) {};
		\node [style=none] (82) at (-0.575, -2.9) {};
		\node [style=none] (83) at (0, -2.5) {};
		\node [style=none] (84) at (-1.025, -3.6) {};
		\node [style=none] (85) at (0.75, -2) {};
		\node [style=none] (86) at (0.75, -1) {};
		\node [style=none] (87) at (1.75, -2.475) {};
		\node [style=none] (88) at (1, -2.5) {};
		\node [style=none] (89) at (0, -1.75) {};
		\node [style=none] (90) at (0, -3.25) {};
		\node [style=none] (91) at (1, -3.6) {};
		\node [style=none] (92) at (1.75, -1.25) {};
		\node [style=none] (93) at (-1.75, -0.5) {};
		\node [style=none] (94) at (-1.25, -1) {};
		\node [style=none] (95) at (-1.75, -2.75) {};
		\node [style=none] (96) at (-1.125, -2) {};
		\node [style=none] (97) at (-0.825, -2.475) {};
		\node [style=none] (98) at (0, -1.25) {};
		\node [style=none] (99) at (1.25, -1.15) {};
		\node [style=none] (100) at (0, -2) {};
		\node [style=none] (101) at (0, -2.75) {};
		\node [style=none] (102) at (0.75, -3) {};
		\node [style=none] (0) at (0, 3.75) {};
		\node [style=none] (1) at (-1.75, 3) {};
		\node [style=none] (2) at (1.75, 3) {};
		\node [style=none] (3) at (0, 2.25) {};
		\node [style=none] (5) at (-1.75, -3) {};
		\node [style=none] (6) at (1.75, -3) {};
		\node [style=none] (7) at (0, -3.75) {};
		\node [style=none] (103) at (1, -0.35) {};
		\node [style=none] (10) at (-1.75, 0.25) {};
		\node [style=none] (11) at (1.75, 0.25) {};
		\node [style=mid-nc] (12) at (0, -0.5) {};
	\end{pgfonlayer}
	\begin{pgfonlayer}{edgelayer}
		\draw [style=fi-gr] (29.center)
			 to (26.center)
			 to [in=-180, out=-90, looseness=0.75] (25.center)
			 to [in=-90, out=0, looseness=0.75] (27.center)
			 to (30.center)
			 to [in=0, out=90, looseness=0.75] (28.center)
			 to [in=90, out=180, looseness=0.75] cycle;
		\draw [style=op-lsh] (43.center)
			 to [in=-180, out=-90, looseness=0.75] (45.center)
			 to [in=-90, out=0, looseness=0.75] (44.center)
			 to [in=0, out=90, looseness=0.75] (42.center)
			 to [in=90, out=180, looseness=0.75] cycle;
		\draw [style=fi-sh] (35.center)
			 to (40.center)
			 to [in=0, out=-90, looseness=0.75] (41.center)
			 to [in=-90, out=-180, looseness=0.75] (39.center)
			 to (32.center)
			 to [in=-180, out=-90, looseness=0.75] (31.center)
			 to [in=-90, out=0, looseness=0.75] cycle;
		\draw [style=vt-bl-da, in=90, out=180, looseness=0.75] (14.center) to (15.center);
		\draw [style=vt-bl-da, in=90, out=0, looseness=0.75] (16.center) to (17.center);
		\draw [style=vt-bl-di] (18.center) to (19.center);
		\draw [style=vt-bl-di] (12) to (3.center);
		\draw [style=vt-lgr] (93.center) to (94.center);
		\draw [style=vt-lgr] (94.center) to (77.center);
		\draw [style=vt-lgr] (94.center) to (80.center);
		\draw [style=vt-lgr] (80.center) to (79.center);
		\draw [style=vt-lgr] (79.center) to (89.center);
		\draw [style=vt-lgr] (79.center) to (66.center);
		\draw [style=vt-lgr] (95.center) to (78.center);
		\draw [style=vt-lgr] (78.center) to (84.center);
		\draw [style=vt-lgr] (78.center) to (82.center);
		\draw [style=vt-lgr] (96.center) to (82.center);
		\draw [style=vt-lgr] (82.center) to (90.center);
		\draw [style=vt-lgr] (97.center) to (83.center);
		\draw [style=vt-lgr] (98.center) to (86.center);
		\draw [style=vt-lgr] (86.center) to (92.center);
		\draw [style=vt-lgr] (99.center) to (85.center);
		\draw [style=vt-lgr] (100.center) to (85.center);
		\draw [style=vt-lgr] (85.center) to (88.center);
		\draw [style=vt-lgr] (88.center) to (87.center);
		\draw [style=vt-lgr] (101.center) to (102.center);
		\draw [style=vt-lgr] (102.center) to (88.center);
		\draw [style=vt-lgr] (102.center) to (91.center);
		\draw [style=vt-bl-di] (7.center) to (12);
		\draw [style=vt] (1.center)
			 to [in=-180, out=-90, looseness=0.75] (3.center)
			 to [in=-90, out=0, looseness=0.75] (2.center)
			 to [in=0, out=90, looseness=0.75] (0.center)
			 to [in=90, out=180, looseness=0.75] cycle;
		\draw [style=vt] (5.center)
			 to [in=-180, out=-90, looseness=0.75] (7.center)
			 to [in=-90, out=0, looseness=0.75] (6.center);
		\draw [style=vt-lgr] (103.center) to (86.center);
		\draw [style=vt-bl] (10.center)
			 to [in=-180, out=-90, looseness=0.75] (12.center)
			 to [in=-90, out=0, looseness=0.75] (11.center);
		\draw [style=vt-lgr] (81.center) to (96.center);
		\draw [style=vt-lgr] (96.center) to (80.center);
		\draw [style=vt] (1.center) to (5.center);
		\draw [style=vt] (2.center) to (6.center);
	\end{pgfonlayer}
\end{tikzpicture}}
  {%
  }}%
} ~\circ~ \scalebox{0.5}{%
{\tikzstyle{every picture}=[tikzfig]
  {\begin{tikzpicture}
	\begin{pgfonlayer}{nodelayer}
		\node [style=none] (25) at (0, -0.5) {};
		\node [style=none] (26) at (-1.75, 0.25) {};
		\node [style=none] (27) at (1.75, 0.25) {};
		\node [style=none] (28) at (0, 3.75) {};
		\node [style=none] (29) at (-1.75, 3) {};
		\node [style=none] (30) at (1.75, 3) {};
		\node [style=none] (42) at (0, 3.75) {};
		\node [style=none] (43) at (-1.75, 3) {};
		\node [style=none] (44) at (1.75, 3) {};
		\node [style=none] (45) at (0, 2.25) {};
		\node [style=none] (31) at (0, -0.5) {};
		\node [style=none] (32) at (-1.75, 0.25) {};
		\node [style=none] (35) at (1.75, 0.25) {};
		\node [style=none] (39) at (-1.75, -3) {};
		\node [style=none] (40) at (1.75, -3) {};
		\node [style=none] (41) at (0, -3.75) {};
		\node [style=none] (14) at (-0.25, 1) {};
		\node [style=none] (15) at (-1.75, 0.25) {};
		\node [style=none] (16) at (0.25, 1) {};
		\node [style=none] (17) at (1.75, 0.25) {};
		\node [style=none] (18) at (1.3, -0.25) {};
		\node [style=none] (19) at (1.425, -0.225) {};
		\node [style=none] (46) at (-1.5, 2.575) {};
		\node [style=none] (47) at (-1, 2.45) {};
		\node [style=none] (48) at (-0.5, 2.275) {};
		\node [style=none] (49) at (0.5, 2.25) {};
		\node [style=none] (50) at (1.25, 2.45) {};
		\node [style=none] (52) at (-1.25, 1.75) {};
		\node [style=none] (53) at (-1, 1.25) {};
		\node [style=none] (54) at (0, 1.75) {};
		\node [style=none] (55) at (0, 1.5) {};
		\node [style=none] (56) at (-0.75, 0.5) {};
		\node [style=none] (57) at (-1.75, 1) {};
		\node [style=none] (58) at (-1.2, -0.3) {};
		\node [style=none] (59) at (0, 0.5) {};
		\node [style=none] (60) at (-0.5, 1.25) {};
		\node [style=none] (61) at (1.75, 2.25) {};
		\node [style=none] (62) at (1.025, 1.85) {};
		\node [style=none] (63) at (0.75, 1) {};
		\node [style=none] (64) at (0.5, 0.5) {};
		\node [style=none] (65) at (-1.25, 0.25) {};
		\node [style=none] (66) at (-0.5, -0.5) {};
		\node [style=none] (67) at (0, 0) {};
		\node [style=none] (68) at (-0.625, 0) {};
		\node [style=none] (69) at (0.5, 1.5) {};
		\node [style=none] (70) at (0, 0.25) {};
		\node [style=none] (71) at (0.65, -0.425) {};
		\node [style=none] (72) at (1.75, 0.75) {};
		\node [style=none] (73) at (1.75, 1.25) {};
		\node [style=none] (74) at (0.9, 1.45) {};
		\node [style=none] (75) at (-0.25, 1.5) {};
		\node [style=none] (76) at (0.8, 0.55) {};
		\node [style=none] (77) at (-0.975, -0.375) {};
		\node [style=none] (78) at (-1.15, -2.95) {};
		\node [style=none] (79) at (-0.5, -1.25) {};
		\node [style=none] (80) at (-1, -1.5) {};
		\node [style=none] (81) at (-1.75, -1.75) {};
		\node [style=none] (82) at (-0.575, -2.9) {};
		\node [style=none] (83) at (0, -2.5) {};
		\node [style=none] (84) at (-1.025, -3.6) {};
		\node [style=none] (85) at (0.75, -2) {};
		\node [style=none] (86) at (0.75, -1) {};
		\node [style=none] (87) at (1.75, -2.475) {};
		\node [style=none] (88) at (1, -2.5) {};
		\node [style=none] (89) at (0, -1.75) {};
		\node [style=none] (90) at (0, -3.25) {};
		\node [style=none] (91) at (1, -3.6) {};
		\node [style=none] (92) at (1.75, -1.25) {};
		\node [style=none] (93) at (-1.75, -0.5) {};
		\node [style=none] (94) at (-1.25, -1) {};
		\node [style=none] (95) at (-1.75, -2.75) {};
		\node [style=none] (96) at (-1.125, -2) {};
		\node [style=none] (97) at (-0.825, -2.475) {};
		\node [style=none] (98) at (0, -1.25) {};
		\node [style=none] (99) at (1.25, -1.15) {};
		\node [style=none] (100) at (0, -2) {};
		\node [style=none] (101) at (0, -2.75) {};
		\node [style=none] (102) at (0.75, -3) {};
		\node [style=none] (0) at (0, 3.75) {};
		\node [style=none] (1) at (-1.75, 3) {};
		\node [style=none] (2) at (1.75, 3) {};
		\node [style=none] (3) at (0, 2.25) {};
		\node [style=none] (5) at (-1.75, -3) {};
		\node [style=none] (6) at (1.75, -3) {};
		\node [style=none] (7) at (0, -3.75) {};
		\node [style=none] (103) at (1, -0.35) {};
		\node [style=none] (10) at (-1.75, 0.25) {};
		\node [style=none] (11) at (1.75, 0.25) {};
		\node [style=mid-nc] (12) at (0, -0.5) {};
	\end{pgfonlayer}
	\begin{pgfonlayer}{edgelayer}
		\draw [style=fi-sh] (29.center)
			 to (26.center)
			 to [in=-180, out=-90, looseness=0.75] (25.center)
			 to [in=-90, out=0, looseness=0.75] (27.center)
			 to (30.center)
			 to [in=0, out=90, looseness=0.75] (28.center)
			 to [in=90, out=180, looseness=0.75] cycle;
		\draw [style=op-lsh] (43.center)
			 to [in=-180, out=-90, looseness=0.75] (45.center)
			 to [in=-90, out=0, looseness=0.75] (44.center)
			 to [in=0, out=90, looseness=0.75] (42.center)
			 to [in=90, out=180, looseness=0.75] cycle;
		\draw [style=fi-sh] (35.center)
			 to (40.center)
			 to [in=0, out=-90, looseness=0.75] (41.center)
			 to [in=-90, out=-180, looseness=0.75] (39.center)
			 to (32.center)
			 to [in=-180, out=-90, looseness=0.75] (31.center)
			 to [in=-90, out=0, looseness=0.75] cycle;
		\draw [style=vt-bl-da, in=90, out=180, looseness=0.75] (14.center) to (15.center);
		\draw [style=vt-bl-da, in=90, out=0, looseness=0.75] (16.center) to (17.center);
		\draw [style=vt-bl-di] (18.center) to (19.center);
		\draw [style=vt-lgr] (46.center) to (52.center);
		\draw [style=vt-lgr] (52.center) to (47.center);
		\draw [style=vt-lgr] (52.center) to (57.center);
		\draw [style=vt-lgr] (52.center) to (53.center);
		\draw [style=vt-lgr] (53.center) to (60.center);
		\draw [style=vt-lgr] (60.center) to (54.center);
		\draw [style=vt-lgr] (60.center) to (59.center);
		\draw [style=vt-lgr] (53.center) to (56.center);
		\draw [style=vt-lgr] (57.center) to (65.center);
		\draw [style=vt-lgr] (65.center) to (56.center);
		\draw [style=vt-lgr] (65.center) to (58.center);
		\draw [style=vt-lgr] (56.center) to (66.center);
		\draw [style=vt-lgr] (68.center) to (67.center);
		\draw [style=vt-lgr] (55.center) to (69.center);
		\draw [style=vt-lgr] (69.center) to (49.center);
		\draw [style=vt-lgr] (69.center) to (63.center);
		\draw [style=vt-lgr] (63.center) to (64.center);
		\draw [style=vt-lgr] (63.center) to (62.center);
		\draw [style=vt-lgr] (62.center) to (50.center);
		\draw [style=vt-lgr] (62.center) to (61.center);
		\draw [style=vt-lgr] (70.center) to (64.center);
		\draw [style=vt-lgr] (64.center) to (72.center);
		\draw [style=vt-lgr] (74.center) to (73.center);
		\draw [style=vt-bl-di] (12) to (3.center);
		\draw [style=vt-lgr] (48.center) to (75.center);
		\draw [style=vt-lgr] (76.center) to (71.center);
		\draw [style=vt-lbl] (93.center) to (94.center);
		\draw [style=vt-lbl] (94.center) to (77.center);
		\draw [style=vt-lbl] (94.center) to (80.center);
		\draw [style=vt-lbl] (80.center) to (79.center);
		\draw [style=vt-lbl] (79.center) to (89.center);
		\draw [style=vt-lbl] (79.center) to (66.center);
		\draw [style=vt-lbl] (95.center) to (78.center);
		\draw [style=vt-lbl] (78.center) to (84.center);
		\draw [style=vt-lbl] (78.center) to (82.center);
		\draw [style=vt-lbl] (96.center) to (82.center);
		\draw [style=vt-lbl] (82.center) to (90.center);
		\draw [style=vt-lbl] (97.center) to (83.center);
		\draw [style=vt-lbl] (98.center) to (86.center);
		\draw [style=vt-lbl] (86.center) to (92.center);
		\draw [style=vt-lbl] (99.center) to (85.center);
		\draw [style=vt-lbl] (100.center) to (85.center);
		\draw [style=vt-lbl] (85.center) to (88.center);
		\draw [style=vt-lbl] (88.center) to (87.center);
		\draw [style=vt-lbl] (101.center) to (102.center);
		\draw [style=vt-lbl] (102.center) to (88.center);
		\draw [style=vt-lbl] (102.center) to (91.center);
		\draw [style=vt-bl-di] (7.center) to (12);
		\draw [style=vt] (1.center)
			 to [in=-180, out=-90, looseness=0.75] (3.center)
			 to [in=-90, out=0, looseness=0.75] (2.center)
			 to [in=0, out=90, looseness=0.75] (0.center)
			 to [in=90, out=180, looseness=0.75] cycle;
		\draw [style=vt] (5.center)
			 to [in=-180, out=-90, looseness=0.75] (7.center)
			 to [in=-90, out=0, looseness=0.75] (6.center);
		\draw [style=vt-lbl] (103.center) to (86.center);
		\draw [style=vt-bl] (10.center)
			 to [in=-180, out=-90, looseness=0.75] (12.center)
			 to [in=-90, out=0, looseness=0.75] (11.center);
		\draw [style=vt-lbl] (81.center) to (96.center);
		\draw [style=vt-lbl] (96.center) to (80.center);
		\draw [style=vt] (1.center) to (5.center);
		\draw [style=vt] (2.center) to (6.center);
	\end{pgfonlayer}
\end{tikzpicture}}
  {%
  }}%
} 
  ~=~ \scalebox{0.5}{%
{\tikzstyle{every picture}=[tikzfig]
  {\begin{tikzpicture}
	\begin{pgfonlayer}{nodelayer}
		\node [style=none] (25) at (0, -0.5) {};
		\node [style=none] (26) at (-1.75, 0.25) {};
		\node [style=none] (27) at (1.75, 0.25) {};
		\node [style=none] (28) at (0, 3.75) {};
		\node [style=none] (29) at (-1.75, 3) {};
		\node [style=none] (30) at (1.75, 3) {};
		\node [style=none] (42) at (0, 3.75) {};
		\node [style=none] (43) at (-1.75, 3) {};
		\node [style=none] (44) at (1.75, 3) {};
		\node [style=none] (45) at (0, 2.25) {};
		\node [style=none] (31) at (0, -0.5) {};
		\node [style=none] (32) at (-1.75, 0.25) {};
		\node [style=none] (35) at (1.75, 0.25) {};
		\node [style=none] (39) at (-1.75, -3) {};
		\node [style=none] (40) at (1.75, -3) {};
		\node [style=none] (41) at (0, -3.75) {};
		\node [style=none] (14) at (-0.25, 1) {};
		\node [style=none] (15) at (-1.75, 0.25) {};
		\node [style=none] (16) at (0.25, 1) {};
		\node [style=none] (17) at (1.75, 0.25) {};
		\node [style=none] (18) at (1.3, -0.25) {};
		\node [style=none] (19) at (1.425, -0.225) {};
		\node [style=none] (66) at (-0.5, -0.5) {};
		\node [style=none] (77) at (-0.975, -0.375) {};
		\node [style=none] (78) at (-1.15, -2.95) {};
		\node [style=none] (79) at (-0.5, -1.25) {};
		\node [style=none] (80) at (-1, -1.5) {};
		\node [style=none] (81) at (-1.75, -1.75) {};
		\node [style=none] (82) at (-0.575, -2.9) {};
		\node [style=none] (83) at (0, -2.5) {};
		\node [style=none] (84) at (-1.025, -3.6) {};
		\node [style=none] (85) at (0.75, -2) {};
		\node [style=none] (86) at (0.75, -1) {};
		\node [style=none] (87) at (1.75, -2.475) {};
		\node [style=none] (88) at (1, -2.5) {};
		\node [style=none] (89) at (0, -1.75) {};
		\node [style=none] (90) at (0, -3.25) {};
		\node [style=none] (91) at (1, -3.6) {};
		\node [style=none] (92) at (1.75, -1.25) {};
		\node [style=none] (93) at (-1.75, -0.5) {};
		\node [style=none] (94) at (-1.25, -1) {};
		\node [style=none] (95) at (-1.75, -2.75) {};
		\node [style=none] (96) at (-1.125, -2) {};
		\node [style=none] (97) at (-0.825, -2.475) {};
		\node [style=none] (98) at (0, -1.25) {};
		\node [style=none] (99) at (1.25, -1.15) {};
		\node [style=none] (100) at (0, -2) {};
		\node [style=none] (101) at (0, -2.75) {};
		\node [style=none] (102) at (0.75, -3) {};
		\node [style=none] (0) at (0, 3.75) {};
		\node [style=none] (1) at (-1.75, 3) {};
		\node [style=none] (2) at (1.75, 3) {};
		\node [style=none] (3) at (0, 2.25) {};
		\node [style=none] (5) at (-1.75, -3) {};
		\node [style=none] (6) at (1.75, -3) {};
		\node [style=none] (7) at (0, -3.75) {};
		\node [style=none] (103) at (1, -0.35) {};
		\node [style=none] (10) at (-1.75, 0.25) {};
		\node [style=none] (11) at (1.75, 0.25) {};
		\node [style=mid-nc] (12) at (0, -0.5) {};
	\end{pgfonlayer}
	\begin{pgfonlayer}{edgelayer}
		\draw [style=fi-gr] (29.center)
			 to (26.center)
			 to [in=-180, out=-90, looseness=0.75] (25.center)
			 to [in=-90, out=0, looseness=0.75] (27.center)
			 to (30.center)
			 to [in=0, out=90, looseness=0.75] (28.center)
			 to [in=90, out=180, looseness=0.75] cycle;
		\draw [style=op-lsh] (43.center)
			 to [in=-180, out=-90, looseness=0.75] (45.center)
			 to [in=-90, out=0, looseness=0.75] (44.center)
			 to [in=0, out=90, looseness=0.75] (42.center)
			 to [in=90, out=180, looseness=0.75] cycle;
		\draw [style=fi-sh] (35.center)
			 to (40.center)
			 to [in=0, out=-90, looseness=0.75] (41.center)
			 to [in=-90, out=-180, looseness=0.75] (39.center)
			 to (32.center)
			 to [in=-180, out=-90, looseness=0.75] (31.center)
			 to [in=-90, out=0, looseness=0.75] cycle;
		\draw [style=vt-bl-da, in=90, out=180, looseness=0.75] (14.center) to (15.center);
		\draw [style=vt-bl-da, in=90, out=0, looseness=0.75] (16.center) to (17.center);
		\draw [style=vt-bl-di] (18.center) to (19.center);
		\draw [style=vt-bl-di] (12) to (3.center);
		\draw [style=vt-lbl] (93.center) to (94.center);
		\draw [style=vt-lbl] (94.center) to (77.center);
		\draw [style=vt-lbl] (94.center) to (80.center);
		\draw [style=vt-lbl] (80.center) to (79.center);
		\draw [style=vt-lbl] (79.center) to (89.center);
		\draw [style=vt-lbl] (79.center) to (66.center);
		\draw [style=vt-lbl] (95.center) to (78.center);
		\draw [style=vt-lbl] (78.center) to (84.center);
		\draw [style=vt-lbl] (78.center) to (82.center);
		\draw [style=vt-lbl] (96.center) to (82.center);
		\draw [style=vt-lbl] (82.center) to (90.center);
		\draw [style=vt-lbl] (97.center) to (83.center);
		\draw [style=vt-lbl] (98.center) to (86.center);
		\draw [style=vt-lbl] (86.center) to (92.center);
		\draw [style=vt-lbl] (99.center) to (85.center);
		\draw [style=vt-lbl] (100.center) to (85.center);
		\draw [style=vt-lbl] (85.center) to (88.center);
		\draw [style=vt-lbl] (88.center) to (87.center);
		\draw [style=vt-lbl] (101.center) to (102.center);
		\draw [style=vt-lbl] (102.center) to (88.center);
		\draw [style=vt-lbl] (102.center) to (91.center);
		\draw [style=vt-bl-di] (7.center) to (12);
		\draw [style=vt] (1.center)
			 to [in=-180, out=-90, looseness=0.75] (3.center)
			 to [in=-90, out=0, looseness=0.75] (2.center)
			 to [in=0, out=90, looseness=0.75] (0.center)
			 to [in=90, out=180, looseness=0.75] cycle;
		\draw [style=vt] (5.center)
			 to [in=-180, out=-90, looseness=0.75] (7.center)
			 to [in=-90, out=0, looseness=0.75] (6.center);
		\draw [style=vt-lbl] (103.center) to (86.center);
		\draw [style=vt-bl] (10.center)
			 to [in=-180, out=-90, looseness=0.75] (12.center)
			 to [in=-90, out=0, looseness=0.75] (11.center);
		\draw [style=vt-lbl] (81.center) to (96.center);
		\draw [style=vt-lbl] (96.center) to (80.center);
		\draw [style=vt] (1.center) to (5.center);
		\draw [style=vt] (2.center) to (6.center);
	\end{pgfonlayer}
\end{tikzpicture}}
  {%
  }}%
}
  \stackrel{\eqref{eq:fa=f=bf}}=
  \scalebox{0.5}{%
{\tikzstyle{every picture}=[tikzfig]
  {\begin{tikzpicture}
	\begin{pgfonlayer}{nodelayer}
		\node [style=none] (25) at (0, -0.5) {};
		\node [style=none] (26) at (-1.75, 0.25) {};
		\node [style=none] (27) at (1.75, 0.25) {};
		\node [style=none] (28) at (0, 3.75) {};
		\node [style=none] (29) at (-1.75, 3) {};
		\node [style=none] (30) at (1.75, 3) {};
		\node [style=none] (45) at (0, 3.75) {};
		\node [style=none] (46) at (-1.75, 3) {};
		\node [style=none] (47) at (1.75, 3) {};
		\node [style=none] (48) at (0, 2.25) {};
		\node [style=none] (31) at (0, -0.5) {};
		\node [style=none] (32) at (-1.75, 0.25) {};
		\node [style=none] (35) at (1.75, 0.25) {};
		\node [style=none] (39) at (-1.75, -3) {};
		\node [style=none] (40) at (1.75, -3) {};
		\node [style=none] (41) at (0, -3.75) {};
		\node [style=none] (10) at (-1.75, 0.25) {};
		\node [style=none] (11) at (1.75, 0.25) {};
		\node [style=mid-nc] (12) at (0, -0.5) {};
		\node [style=none] (14) at (-0.25, 1) {};
		\node [style=none] (15) at (-1.75, 0.25) {};
		\node [style=none] (16) at (0.25, 1) {};
		\node [style=none] (17) at (1.75, 0.25) {};
		\node [style=none] (0) at (0, 3.75) {};
		\node [style=none] (1) at (-1.75, 3) {};
		\node [style=none] (2) at (1.75, 3) {};
		\node [style=none] (3) at (0, 2.25) {};
		\node [style=none] (5) at (-1.75, -3) {};
		\node [style=none] (6) at (1.75, -3) {};
		\node [style=none] (7) at (0, -3.75) {};
		\node [style=none] (18) at (1.3, -0.25) {};
		\node [style=none] (19) at (1.425, -0.225) {};
	\end{pgfonlayer}
	\begin{pgfonlayer}{edgelayer}
		\draw [style=fi-gr] (29.center)
			 to (26.center)
			 to [in=180, out=-90, looseness=0.75] (25.center)
			 to [in=-90, out=0, looseness=0.75] (27.center)
			 to (30.center)
			 to [in=0, out=90, looseness=0.75] (28.center)
			 to [in=90, out=180, looseness=0.75] cycle;
		\draw [style=op-lsh] (46.center)
			 to [in=180, out=-90, looseness=0.75] (48.center)
			 to [in=-90, out=0, looseness=0.75] (47.center)
			 to [in=0, out=90, looseness=0.75] (45.center)
			 to [in=90, out=180, looseness=0.75] cycle;
		\draw [style=fi-bl] (35.center)
			 to (40.center)
			 to [in=0, out=-90, looseness=0.75] (41.center)
			 to [in=-90, out=180, looseness=0.75] (39.center)
			 to (32.center)
			 to [in=180, out=-90, looseness=0.75] (31.center)
			 to [in=-90, out=0, looseness=0.75] cycle;
		\draw [style=vt-bl-di] (7.center) to (12);
		\draw [style=vt-bl] (10.center)
			 to [in=180, out=-90, looseness=0.75] (12.center)
			 to [in=-90, out=0, looseness=0.75] (11.center);
		\draw [style=vt-bl-di] (12) to (3.center);
		\draw [style=vt-bl-da, in=90, out=180, looseness=0.75] (14.center) to (15.center);
		\draw [style=vt-bl-da, in=90, out=0, looseness=0.75] (16.center) to (17.center);
		\draw [style=vt] (1.center)
			 to [in=180, out=-90, looseness=0.75] (3.center)
			 to [in=-90, out=0, looseness=0.75] (2.center)
			 to [in=0, out=90, looseness=0.75] (0.center)
			 to [in=90, out=180, looseness=0.75] cycle;
		\draw [style=vt] (5.center)
			 to [in=180, out=-90, looseness=0.75] (7.center)
			 to [in=-90, out=0, looseness=0.75] (6.center);
		\draw [style=vt] (1.center) to (5.center);
		\draw [style=vt] (2.center) to (6.center);
		\draw [style=vt-bl-di] (18.center) to (19.center);
	\end{pgfonlayer}
\end{tikzpicture}}
  {%
  }}%
} ~\circ~ \scalebox{0.5}{%
{\tikzstyle{every picture}=[tikzfig]
  {\begin{tikzpicture}
	\begin{pgfonlayer}{nodelayer}
		\node [style=none] (25) at (0, -0.5) {};
		\node [style=none] (26) at (-1.75, 0.25) {};
		\node [style=none] (27) at (1.75, 0.25) {};
		\node [style=none] (28) at (0, 3.75) {};
		\node [style=none] (29) at (-1.75, 3) {};
		\node [style=none] (30) at (1.75, 3) {};
		\node [style=none] (42) at (0, 3.75) {};
		\node [style=none] (43) at (-1.75, 3) {};
		\node [style=none] (44) at (1.75, 3) {};
		\node [style=none] (45) at (0, 2.25) {};
		\node [style=none] (31) at (0, -0.5) {};
		\node [style=none] (32) at (-1.75, 0.25) {};
		\node [style=none] (35) at (1.75, 0.25) {};
		\node [style=none] (39) at (-1.75, -3) {};
		\node [style=none] (40) at (1.75, -3) {};
		\node [style=none] (41) at (0, -3.75) {};
		\node [style=none] (14) at (-0.25, 1) {};
		\node [style=none] (15) at (-1.75, 0.25) {};
		\node [style=none] (16) at (0.25, 1) {};
		\node [style=none] (17) at (1.75, 0.25) {};
		\node [style=none] (18) at (1.3, -0.25) {};
		\node [style=none] (19) at (1.425, -0.225) {};
		\node [style=none] (66) at (-0.5, -0.5) {};
		\node [style=none] (77) at (-0.975, -0.375) {};
		\node [style=none] (78) at (-1.15, -2.95) {};
		\node [style=none] (79) at (-0.5, -1.25) {};
		\node [style=none] (80) at (-1, -1.5) {};
		\node [style=none] (81) at (-1.75, -1.75) {};
		\node [style=none] (82) at (-0.575, -2.9) {};
		\node [style=none] (83) at (0, -2.5) {};
		\node [style=none] (84) at (-1.025, -3.6) {};
		\node [style=none] (85) at (0.75, -2) {};
		\node [style=none] (86) at (0.75, -1) {};
		\node [style=none] (87) at (1.75, -2.475) {};
		\node [style=none] (88) at (1, -2.5) {};
		\node [style=none] (89) at (0, -1.75) {};
		\node [style=none] (90) at (0, -3.25) {};
		\node [style=none] (91) at (1, -3.6) {};
		\node [style=none] (92) at (1.75, -1.25) {};
		\node [style=none] (93) at (-1.75, -0.5) {};
		\node [style=none] (94) at (-1.25, -1) {};
		\node [style=none] (95) at (-1.75, -2.75) {};
		\node [style=none] (96) at (-1.125, -2) {};
		\node [style=none] (97) at (-0.825, -2.475) {};
		\node [style=none] (98) at (0, -1.25) {};
		\node [style=none] (99) at (1.25, -1.15) {};
		\node [style=none] (100) at (0, -2) {};
		\node [style=none] (101) at (0, -2.75) {};
		\node [style=none] (102) at (0.75, -3) {};
		\node [style=none] (0) at (0, 3.75) {};
		\node [style=none] (1) at (-1.75, 3) {};
		\node [style=none] (2) at (1.75, 3) {};
		\node [style=none] (3) at (0, 2.25) {};
		\node [style=none] (5) at (-1.75, -3) {};
		\node [style=none] (6) at (1.75, -3) {};
		\node [style=none] (7) at (0, -3.75) {};
		\node [style=none] (103) at (1, -0.35) {};
		\node [style=none] (10) at (-1.75, 0.25) {};
		\node [style=none] (11) at (1.75, 0.25) {};
		\node [style=mid-nc] (12) at (0, -0.5) {};
	\end{pgfonlayer}
	\begin{pgfonlayer}{edgelayer}
		\draw [style=fi-bl] (29.center)
			 to (26.center)
			 to [in=-180, out=-90, looseness=0.75] (25.center)
			 to [in=-90, out=0, looseness=0.75] (27.center)
			 to (30.center)
			 to [in=0, out=90, looseness=0.75] (28.center)
			 to [in=90, out=180, looseness=0.75] cycle;
		\draw [style=op-lsh] (43.center)
			 to [in=-180, out=-90, looseness=0.75] (45.center)
			 to [in=-90, out=0, looseness=0.75] (44.center)
			 to [in=0, out=90, looseness=0.75] (42.center)
			 to [in=90, out=180, looseness=0.75] cycle;
		\draw [style=fi-sh] (35.center)
			 to (40.center)
			 to [in=0, out=-90, looseness=0.75] (41.center)
			 to [in=-90, out=-180, looseness=0.75] (39.center)
			 to (32.center)
			 to [in=-180, out=-90, looseness=0.75] (31.center)
			 to [in=-90, out=0, looseness=0.75] cycle;
		\draw [style=vt-bl-da, in=90, out=180, looseness=0.75] (14.center) to (15.center);
		\draw [style=vt-bl-da, in=90, out=0, looseness=0.75] (16.center) to (17.center);
		\draw [style=vt-bl-di] (18.center) to (19.center);
		\draw [style=vt-bl-di] (12) to (3.center);
		\draw [style=vt-lbl] (93.center) to (94.center);
		\draw [style=vt-lbl] (94.center) to (77.center);
		\draw [style=vt-lbl] (94.center) to (80.center);
		\draw [style=vt-lbl] (80.center) to (79.center);
		\draw [style=vt-lbl] (79.center) to (89.center);
		\draw [style=vt-lbl] (79.center) to (66.center);
		\draw [style=vt-lbl] (95.center) to (78.center);
		\draw [style=vt-lbl] (78.center) to (84.center);
		\draw [style=vt-lbl] (78.center) to (82.center);
		\draw [style=vt-lbl] (96.center) to (82.center);
		\draw [style=vt-lbl] (82.center) to (90.center);
		\draw [style=vt-lbl] (97.center) to (83.center);
		\draw [style=vt-lbl] (98.center) to (86.center);
		\draw [style=vt-lbl] (86.center) to (92.center);
		\draw [style=vt-lbl] (99.center) to (85.center);
		\draw [style=vt-lbl] (100.center) to (85.center);
		\draw [style=vt-lbl] (85.center) to (88.center);
		\draw [style=vt-lbl] (88.center) to (87.center);
		\draw [style=vt-lbl] (101.center) to (102.center);
		\draw [style=vt-lbl] (102.center) to (88.center);
		\draw [style=vt-lbl] (102.center) to (91.center);
		\draw [style=vt-bl-di] (7.center) to (12);
		\draw [style=vt] (1.center)
			 to [in=-180, out=-90, looseness=0.75] (3.center)
			 to [in=-90, out=0, looseness=0.75] (2.center)
			 to [in=0, out=90, looseness=0.75] (0.center)
			 to [in=90, out=180, looseness=0.75] cycle;
		\draw [style=vt] (5.center)
			 to [in=-180, out=-90, looseness=0.75] (7.center)
			 to [in=-90, out=0, looseness=0.75] (6.center);
		\draw [style=vt-lbl] (103.center) to (86.center);
		\draw [style=vt-bl] (10.center)
			 to [in=-180, out=-90, looseness=0.75] (12.center)
			 to [in=-90, out=0, looseness=0.75] (11.center);
		\draw [style=vt-lbl] (81.center) to (96.center);
		\draw [style=vt-lbl] (96.center) to (80.center);
		\draw [style=vt] (1.center) to (5.center);
		\draw [style=vt] (2.center) to (6.center);
	\end{pgfonlayer}
\end{tikzpicture}}
  {%
  }}%
}
  ~=~ f \circ (\varepsilon_\ell^{})_{a}^{} \,.
  \ee

\medskip

Putting these findings together, we arrive at:

\begin{thm} \label{thm:Feq}
Let $\cc$ be a \ko-linear pivotal tensor category and $\ell$ a compact oriented
one-ma\-ni\-fold with possibly non-empty boundary. There is an adjoint equivalence 
  \be
  \df{\!_\ell}^{} ~\colon~ \SNfrc(\ell) 
  \begin{tikzcd}[column sep=2.2em]
  \hspace*{-0.4em} \ar[yshift=4.5pt]{r} & \hspace*{-0.4em}
  \ar[yshift=-3.4pt]{l}[swap,yshift=-0.9pt]{\simeq}
  \end{tikzcd}
  \SNc(\ell) ~\colon\, \imath_\ell^{} 
  \ee
of idempotent-completed cylinder categories, with $\df_{\!\ell}$ the field functor 
\eqref{eq:def(ell)} and $\imath$ the functor \eqref{eq:def:imath}.
\end{thm}

\begin{cor} \label{cor:Feq}
Let $\cc$ be a pivotal fusion category. The idempotent-completed cylinder categories
for the interval $\di$ and the circle $\ds^1$ obey the equivalences
  \be
  \SNfrc(\di) \simeq \SNc(\di)\simeq\cc \qquad \text{and} \qquad
  \SNfrc(\ds^{1}) \simeq \SNc(\ds^1)\simeq\czc
  \ee
of \ko-linear categories.
\end{cor}

Note that in Corollary \ref{cor:Feq} the statement for the interval $\di$ is trivial:
it is immediate that $\SNfrc(\di)\,{\simeq}\, \hfrc(\tu,\tu) \,{\simeq}~ \cc$.

\begin{rem}
Superficially, the horizontally drawn string net in \ref{eq:counit_a} looks like 
consisting of two halves that carry different types of string nets. This is not
correct -- we just deal with a particular string net based on the bicategory $\cfrc$.
On 
the other hand, at least for spherical fusion categories $\cc$, the modular functors 
that are constructed from $\cc$ by the string-net construction and by the Turaev-Viro 
state-sum construction are expected to be equivalent \cite{kirI24,bart8}. Moreover,
topological defects between Turaev-Viro theories are well understood. It is therefore
reasonable to believe that there also exists a satisfactory theory of topological 
defects between bicategorical string nets. We anticipate that such defects amount
to string nets with one side colored by a bicategory $\sphmod$, which should be some
variant of the spherical bicategory of spherical $\cc$-module categories
considered in \cite{fgjs}, and the other side by the delooping of $\cc$. There 
should then further exist forgetful functors generalizing the functor \eqref{eq:ufr}
that give rise to such topological defects. By combining these ingredients we expect 
to obtain a linear functor
  \be
  \snb(\ell\Times \Di) \Colon \snc(\ell)^{\op} \oti \snsphmod(\ell) \rarr~ \vct
  \ee
and thus a functor
  \be
  \snsphmod(\ell) \rarr~ \mathrm{PSh}(\snc(\ell))
  \ee
that naturally provides field objects as objects in the free cocompletion, the
presheaf category. These ideas are, however, beyond the scope of the present paper.
\end{rem}

%%%%%%%%%%%%%%%%%%%%%%%%%%%%%%%%%%%%%%%%%%%%%%%%%%%%%%%%%%%%%%%%%%%%%%%%

\subsection{Universal correlators are invertible linear maps}

As recalled in Section \ref{sec:fresh}, for any oriented open-clo\-sed bordism
$\surf\colon\ell\ptO\ell'$ with boundary datum 
$(a,a')\iN\SNfrc(\overline\ell\,{\sqcup}\,\ell')$ there is a linear map from the 
string-net space $\SNfrc(\surf;a,a')$ to $\SNc(\surf;\df_{\!\ell}\,a,\df_{\!\ell'}
\,a')$: the universal correlator $\ucorc(\surf)_{a,a'}$ \eqref{eq:ucorc-surf-aa'}.
For the case that $\surf$ is the cylinder $\ell\Times\Di$ over a one-manifold $\ell$,
Theorem \ref{thm:Feq} implies that the universal correlator in fact constitutes
a linear isomorphism
  \be
  \ucorc(\ell\Times\Di)_{a,a'} \Colon \SNfrc(\ell\Times\Di;a,a')
  \iso \SNc(\ell\Times\Di;\df_{\!\ell}^{}\,a,\df_{\!\ell'}^{}\,a')
  \label{eq:ucorc-aa'}
  \ee
between string-net spaces. The inverse of the morphism \eqref{eq:ucorc-aa'} is
obtained by combining the counit (\ref{eq:counit}), its inverse, and the 
quasi-inverse $\imath_\ell\colon\SNc(\ell)\eqv\SNfrc(\ell)$ of the functor
$\df_{\!\ell}$.
 
We will now establish an extension of this result from cylinders to general bordisms.
To this end we define for any open-closed bordism $\surf\colon\ell\ptO\ell'$
with given in- and outgoing boundary data $(a,a')$ a linear map
  \be
  \imath_\surf^{} \Colon \SNc(\surf;\df_{\!\ell}\,a,\df_{\!\ell'}\,a') \rarr~
  \SNfrc(\surf;\imath_{\ell}\cir\df_{\!\ell}\,a,\imath_{\ell'}\cir\df_{\!\ell'}\,a')
  \ee
between string-net spaces. As in the case of the functor $\imath_\ell$ 
\eqref{eq:def:imath}, $\imath_\surf^{}$ interprets a $\cc$-colored string net on 
$\surf$ as an $\cfrc$-colored one, making use of the inclusion 
$\cbc \,{\hookrightarrow}\, \cfrc$ of pointed bicategories. That is, the image
under $\imath_\surf^{}$ of a representative graph $\grph$ of a string net
$[\grph] \iN \SNc(\surf;\df_{\!\ell}\,a,\df_{\!\ell'}\,a')$ is the string net in
$\SNfrc(\surf;\imath_{\ell}\cir\df_{\!\ell}\,a,\imath_{\ell'}\cir\df_{\!\ell'}\,a')$
that is represented by the same graph $\grph$ supplemented by a coloring of each
2-cell, as well as of any adjacent 1-cell on the boundary, by the trivial Frobenius 
algebra $\tu$. Let us illustrate this map for one specific graph (omitting all labels):
  \be
{\tikzstyle{every picture}=[tikzfig]
  {\input{UCI3.tikz}}
  {%
  }}%
 
  ~~~~\xmapsto{~~\imath_\surf^{}~~}~~
{\tikzstyle{every picture}=[tikzfig]
  {\input{UCI2.tikz}}
  {%
  }}%

  \ee
~

\begin{thm} \label{thm:UcorIso}
Let $\cc$ be a \ko-linear pivotal tensor category
and $\surf\colon\ell\ptO\ell'$ an oriented open-clo\-sed bordism. Then
for every boundary datum $(a,a')\iN\SNfrc(\overline\ell\,{\sqcup}\,\ell')$,
the universal correlator  \ref{eq:ucorc-surf-aa'} provides an isomorphism 
  \be
  \ucorc(\surf)_{a,a'} \Colon
  \SNfrc(\surf;a,a')\iso\SNc(\surf;\df_{\!\ell}^{}\,a,\df_{\!\ell'}^{}\,a')  
  \label{eq:UCor-rhs}
  \ee
of string-net spaces.
\end{thm}

\begin{proof}
We show that an inverse of the linear map $\ucorc(\surf)_{a,a'}$ is given by
  \be
  \begin{aligned}
  \Phi_{a,a'} \Colon \SNc(\surf;\df_{\!\ell}^{}\,a,\df_{\!\ell'}^{}\,a')
  & \rarr~ \SNfrc(\surf;a,a') \,,
  \\
  \grph & \xmapsto{\phantom{xw}} (\varepsilon_\ell)_{a}^{-1}
  \cdot \imath_\surf^{}(\grph) \cdot (\varepsilon_{\ell'})_{a'} \,.
  \end{aligned}
  \ee
Let us first show that the linear map $\Phi \,{\equiv}\, \Phi_{a,a'}$ is a right 
inverse of $\ucorc \,{\equiv}\, \ucorc(\surf)_{a,a'}$. Noting that
$\ucorc((\varepsilon_\ell)_{a}^{}) \eq \df_{\!\ell}\,a$
and that $\ucorc$ is compatible with sewing, we have 
  \be
  \begin{aligned}
  \ucorc\circ\Phi(\grph) & =\ucorc((\varepsilon_\ell)_{a}^{-1})
  \cdot \ucorc(\imath_\surf(\grph)) \cdot \ucorc((\varepsilon_{\ell'})_{a'})
  \\[2pt]
  & = (\df_{\!\ell}^{}\,a)^{\vee}_{} \cdot \grph \cdot\df_{\!\ell'}^{}\,a'
  = \grph .
  \end{aligned}
  \ee
Here the second equality also makes use of the fact that the composite 
$\ucorc\cir\imath_\surf$ acts as the identity, while the last equality uses that 
$\SNc(\surf,\df_{\!\ell}\,a,\df_{\!\ell'}\,a')$ is by definition 
invariant under sewing with the idempotents underlying the boundary data.
 \\
To show that $\Phi$ is also a left inverse of $\ucorc$, i.e.\ that
  \be
  (\varepsilon _\ell)_{a}^{-1} \cdot \imath_\surf (\ucorc(\grph')) \cdot
  (\varepsilon_{\ell'})_{a'} = \grph'
  \label{eq:PhiLI}
  \ee
for every string net $\grph'\iN\SNfrc(\surf;a,a')$, we argue pictorially:
For a specific choice of quadruple $(\surf,a,a',\grph')$, \eqref{eq:PhiLI}
reads (omitting again all labels)
  \be
{\tikzstyle{every picture}=[tikzfig]
  {\input{UCI0.tikz}}
  {%
  }}%
 ~\quad=\quad %
{\tikzstyle{every picture}=[tikzfig]
  {\begin{tikzpicture}
	\begin{pgfonlayer}{nodelayer}
		\node [style=none] (37) at (-3.25, -1) {};
		\node [style=none] (38) at (-4, -2.5) {};
		\node [style=none] (39) at (-2.5, -2.5) {};
		\node [style=none] (40) at (-3.25, -4) {};
		\node [style=none] (33) at (-3.25, 4) {};
		\node [style=none] (34) at (-4, 2.5) {};
		\node [style=none] (35) at (-2.5, 2.5) {};
		\node [style=none] (36) at (-3.25, 1) {};
		\node [style=none] (83) at (-0.775, -0.75) {};
		\node [style=none] (84) at (0.025, -0.975) {};
		\node [style=none] (68) at (-0.775, -0.75) {};
		\node [style=none] (69) at (-2.625, -1.725) {};
		\node [style=none] (71) at (-2.7, 1.5) {};
		\node [style=none] (72) at (-0.9, -0.1) {};
		\node [style=none] (73) at (1.25, -0.75) {};
		\node [style=none] (74) at (-0.275, -0.05) {};
		\node [style=none] (75) at (0.225, -0.575) {};
		\node [style=none] (76) at (-2.625, -3.25) {};
		\node [style=none] (78) at (-2.5, 2.5) {};
		\node [style=none] (79) at (-2.5, -2.5) {};
		\node [style=none] (80) at (-0.775, -0.75) {};
		\node [style=none] (81) at (-1.175, -1.1) {};
		\node [style=none] (82) at (0.025, -0.975) {};
		\node [style=none] (63) at (1.25, -0.75) {};
		\node [style=none] (64) at (-2.625, -3.25) {};
		\node [style=none] (65) at (-3.25, -4) {};
		\node [style=none] (66) at (4, 0) {};
		\node [style=none] (67) at (3.25, -1.5) {};
		\node [style=none] (50) at (0.025, 0.275) {};
		\node [style=none] (51) at (1.475, 0.275) {};
		\node [style=none] (52) at (1.25, -0.75) {};
		\node [style=none] (53) at (-0.275, -0.05) {};
		\node [style=none] (54) at (0.225, -0.575) {};
		\node [style=none] (55) at (-3.25, 4) {};
		\node [style=none] (56) at (-2.5, 2.5) {};
		\node [style=none] (58) at (4, 0) {};
		\node [style=none] (59) at (3.25, 1.5) {};
		\node [style=none] (60) at (4, 0) {};
		\node [style=none] (61) at (0.75, 0) {};
		\node [style=none] (62) at (1.475, 0.275) {};
		\node [style=none] (42) at (-0.775, -0.75) {};
		\node [style=none] (43) at (-2.625, -1.725) {};
		\node [style=none] (44) at (-1.175, -1.1) {};
		\node [style=none] (45) at (-2.7, 1.5) {};
		\node [style=none] (46) at (-0.9, -0.1) {};
		\node [style=none] (47) at (-3.25, 1) {};
		\node [style=none] (48) at (-3.25, -1) {};
		\node [style=none] (12) at (-0.25, 0.75) {};
		\node [style=none] (13) at (1.75, 0.75) {};
		\node [style=none] (14) at (0.75, 0) {};
		\node [style=none] (15) at (0.025, 0.275) {};
		\node [style=none] (16) at (1.475, 0.275) {};
		\node [style=antipode] (17) at (-0.775, -0.75) {};
		\node [style=none] (22) at (-2.625, -1.725) {};
		\node [style=none] (23) at (-1.175, -1.1) {};
		\node [style=none] (25) at (-2.7, 1.5) {};
		\node [style=none] (26) at (-0.9, -0.1) {};
		\node [style=antipode] (27) at (1.25, -0.75) {};
		\node [style=none] (28) at (-0.275, -0.05) {};
		\node [style=none] (29) at (0.225, -0.575) {};
		\node [style=none] (30) at (-2.625, -3.25) {};
		\node [style=none] (32) at (0.025, -0.975) {};
		\node [style=none] (8) at (3.25, 1.5) {};
		\node [style=none] (10) at (4, 0) {};
		\node [style=none] (11) at (3.25, -1.5) {};
		\node [style=none] (85) at (-2.25, -1.65) {};
		\node [style=none] (86) at (-2.175, -1.625) {};
		\node [style=none] (87) at (-2.25, -3.225) {};
		\node [style=none] (88) at (-2.175, -3.225) {};
		\node [style=none] (89) at (0, -1) {};
		\node [style=none] (90) at (0.05, -0.875) {};
		\node [style=none] (91) at (-3.25, 4) {};
		\node [style=none] (92) at (-4, 2.5) {};
		\node [style=none] (93) at (-2.5, 2.5) {};
		\node [style=none] (94) at (-3.25, 1) {};
		\node [style=none] (95) at (-3.25, -1) {};
		\node [style=none] (96) at (-4, -2.5) {};
		\node [style=none] (97) at (-2.5, -2.5) {};
		\node [style=none] (98) at (-3.25, -4) {};
		\node [style=none] (0) at (-3.25, 4) {};
		\node [style=none] (1) at (-4, 2.5) {};
		\node [style=none] (2) at (-2.5, 2.5) {};
		\node [style=none] (3) at (-3.25, 1) {};
		\node [style=none] (4) at (-3.25, -1) {};
		\node [style=none] (5) at (-4, -2.5) {};
		\node [style=none] (6) at (-2.5, -2.5) {};
		\node [style=none] (7) at (-3.25, -4) {};
	\end{pgfonlayer}
	\begin{pgfonlayer}{edgelayer}
		\draw [style=fi-bl] (37.center)
			 to [in=90, out=-180, looseness=0.75] (38.center)
			 to [in=180, out=-90, looseness=0.75] (40.center)
			 to [in=-90, out=0, looseness=0.75] (39.center)
			 to [in=0, out=90, looseness=0.75] cycle;
		\draw [style=fi-bl] (34.center)
			 to [in=180, out=-90, looseness=0.75] (36.center)
			 to [in=-90, out=0, looseness=0.75] (35.center)
			 to [in=0, out=90, looseness=0.75] (33.center)
			 to [in=90, out=-180, looseness=0.75] cycle;
		\draw [style=fi-gr] (84.center)
			 to [in=-60, out=-105, looseness=1.25] (83.center)
			 to [in=75, out=30, looseness=1.25] cycle;
		\draw [style=fi-pu] (82.center)
			 to [in=30, out=75, looseness=1.25] (68.center)
			 to (72.center)
			 to [in=-15, out=105] (71.center)
			 to [in=-90, out=60, looseness=0.75] (78.center)
			 to [in=120, out=-15] (74.center)
			 to [in=150, out=-60] (75.center)
			 to [in=-180, out=-30] (73.center)
			 to [in=0, out=-120, looseness=0.75] (76.center)
			 to [in=-90, out=75] (79.center)
			 to [in=-75, out=90] (69.center)
			 to [bend right=15, looseness=0.75] (81.center)
			 to [in=225, out=30, looseness=0.75] (80.center)
			 to [in=-105, out=-60, looseness=1.25] cycle;
		\draw [style=fi-bl] (67.center)
			 to [in=0, out=-180] (65.center)
			 to [in=-105, out=0] (64.center)
			 to [in=-120, out=0, looseness=0.75] (63.center)
			 to [bend left=15] (66.center)
			 to [in=0, out=-90, looseness=0.75] cycle;
		\draw [style=fi-bl] (61.center)
			 to [in=-45, out=-180] (50.center)
			 to [bend left=75, looseness=0.75] (51.center)
			 to (60.center)
			 to [in=0, out=90, looseness=0.75] (59.center)
			 to [in=0, out=180] (55.center)
			 to [in=90, out=0, looseness=0.75] (56.center)
			 to [in=120, out=-15] (53.center)
			 to [in=150, out=-60] (54.center)
			 to [in=-180, out=-30] (52.center)
			 to [bend left=15] (58.center)
			 to (62.center)
			 to [in=0, out=-135] cycle;
		\draw [style=fi-bl] (42.center)
			 to (46.center)
			 to [in=-15, out=105] (45.center)
			 to [in=15, out=-120] (47.center)
			 to [bend left=90, looseness=2.75] (48.center)
			 to [in=105, out=0, looseness=0.75] (43.center)
			 to [bend right=15, looseness=0.75] (44.center)
			 to [in=225, out=30, looseness=0.75] cycle;
		\draw [style=vt, bend left=90, looseness=2.75] (3.center) to (4.center);
		\draw [style=vt] (12.center)
			 to [in=-180, out=-75] (14.center)
			 to [in=255, out=0] (13.center);
		\draw [style=vt, bend left=75, looseness=0.75] (15.center) to (16.center);
		\draw [style=vt-bl, bend right=15, looseness=0.75] (22.center) to (23.center);
		\draw [style=vt-bl, in=225, out=30, looseness=0.75] (23.center) to (17);
		\draw [style=vt-bl-di, in=-15, out=105] (26.center) to (25.center);
		\draw [style=vt-bl] (26.center) to (17);
		\draw [style=vt-bl-di, in=-15, out=120] (28.center) to (2.center);
		\draw [style=vt-bl, in=150, out=-60] (28.center) to (29.center);
		\draw [style=vt-bl, in=-180, out=-30] (29.center) to (27);
		\draw [style=vt-bl, in=-120, out=0, looseness=0.75] (30.center) to (27);
		\draw [style=vt-bl-di, bend left=15] (27) to (10.center);
		\draw [style=vt-bl, in=75, out=30, looseness=1.25] (17) to (32.center);
		\draw [style=vt-bl, in=-105, out=-60, looseness=1.25] (17) to (32.center);
		\draw [style=vt] (7.center)
			 to [in=-180, out=0] (11.center)
			 to [in=-90, out=0, looseness=0.75] (10.center)
			 to [in=0, out=90, looseness=0.75] (8.center)
			 to [in=0, out=180] (0.center);
		\draw [style=vt-bl-di] (85.center) to (86.center);
		\draw [style=vt-bl-di] (87.center) to (88.center);
		\draw [style=vt-bl-di] (89.center) to (90.center);
		\draw [style=op-lsh] (92.center)
			 to [in=180, out=-90, looseness=0.75] (94.center)
			 to [in=-90, out=0, looseness=0.75] (93.center)
			 to [in=0, out=90, looseness=0.75] (91.center)
			 to [in=90, out=-180, looseness=0.75] cycle;
		\draw [style=op-lsh] (97.center)
			 to [in=0, out=90, looseness=0.75] (95.center)
			 to [in=90, out=-180, looseness=0.75] (96.center)
			 to [in=180, out=-90, looseness=0.75] (98.center)
			 to [in=-90, out=0, looseness=0.75] cycle;
		\draw [style=vt] (1.center)
			 to [in=180, out=-90, looseness=0.75] (3.center)
			 to [in=-90, out=0, looseness=0.75] (2.center)
			 to [in=0, out=90, looseness=0.75] (0.center)
			 to [in=90, out=-180, looseness=0.75] cycle;
		\draw [style=vt] (5.center)
			 to [in=180, out=-90, looseness=0.75] (7.center)
			 to [in=-90, out=0, looseness=0.75] (6.center)
			 to [in=0, out=90, looseness=0.75] (4.center)
			 to [in=90, out=-180, looseness=0.75] cycle;
	\end{pgfonlayer}
\end{tikzpicture}}
  {%
  }}%

  \ee
This equality holds because the structure morphisms (multiplication, 
comultiplication, module action etc.) of all Frobenius algebras and bimodules 
involved descend to bimodule morphisms between relative tensor products of bimodules.
For instance, the multiplication of a Frobenius algebra $A\iN\cfrc$ descends to 
an $A$-$A$-bimodule isomorphism $A\,{\otimes_{A}}\,A \,{\cong}\, A$. 
As a consequence, we can turn any  2-cell that is labeled by the trivial Frobenius
algebra $\tu$ and contains a Frobenius graph labeled by a Frobenius algebra
$A$ into an $A$-labeled 2-cell. As an illustrative example, consider a region that
locally looks like the left hand side of the following picture:
  \be
{\tikzstyle{every picture}=[tikzfig]
  {\begin{tikzpicture}
	\begin{pgfonlayer}{nodelayer}
		\node [style=none] (28) at (-2, 2) {};
		\node [style=none] (29) at (-2, -2) {};
		\node [style=none] (30) at (0, 1) {};
		\node [style=none] (31) at (0, -1) {};
		\node [style=none] (32) at (0, 2) {};
		\node [style=none] (33) at (0, -2) {};
		\node [style=none] (34) at (-1, 0) {};
		\node [style=none] (35) at (2, 2) {};
		\node [style=none] (36) at (2, -2) {};
		\node [style=none] (37) at (0, 1) {};
		\node [style=none] (38) at (0, -1) {};
		\node [style=none] (39) at (0, 2) {};
		\node [style=none] (40) at (0, -2) {};
		\node [style=none] (41) at (1, 0) {};
		\node [style=none] (42) at (0, 1) {};
		\node [style=none] (43) at (0, -1) {};
		\node [style=none] (44) at (-1, 0) {};
		\node [style=none] (45) at (1, 0) {};
		\node [style=none] (12) at (-2, 2) {};
		\node [style=none] (13) at (-2, -2) {};
		\node [style=none] (18) at (2, -2) {};
		\node [style=none] (19) at (2, 2) {};
		\node [style=none] (0) at (-2, 2) {};
		\node [style=none] (1) at (-2, -2) {};
		\node [style=none] (2) at (2, -2) {};
		\node [style=none] (3) at (2, 2) {};
		\node [style=none] (20) at (0, 1) {};
		\node [style=none] (21) at (0, -1) {};
		\node [style=none] (22) at (0, 2) {};
		\node [style=none] (23) at (0, -2) {};
		\node [style=none] (24) at (-1, 0) {};
		\node [style=none] (25) at (1, 0) {};
	\end{pgfonlayer}
	\begin{pgfonlayer}{edgelayer}
		\draw [style=fi-bl] (29.center)
			 to (28.center)
			 to (32.center)
			 to (30.center)
			 to [bend right=45] (34.center)
			 to [bend right=45] (31.center)
			 to (33.center)
			 to cycle;
		\draw [style=fi-sh] (36.center)
			 to (35.center)
			 to (39.center)
			 to (37.center)
			 to [bend left=45] (41.center)
			 to [bend left=45] (38.center)
			 to (40.center)
			 to cycle;
		\draw [style=fi-sh] (44.center)
			 to [bend right=45] (43.center)
			 to [bend right=45] (45.center)
			 to [bend right=45] (42.center)
			 to [bend right=45] cycle;
		\draw [style=vt-lbl, bend right=45] (20.center) to (24.center);
		\draw [style=vt-lbl, bend right=45] (24.center) to (21.center);
		\draw [style=vt-lbl, bend left=45] (20.center) to (25.center);
		\draw [style=vt-lbl, bend left=45] (25.center) to (21.center);
		\draw [style=vt-lbl] (21.center) to (23.center);
		\draw [style=vt-lbl] (20.center) to (22.center);
	\end{pgfonlayer}
\end{tikzpicture}}
  {%
  }}%
 ~~=~~ %
{\tikzstyle{every picture}=[tikzfig]
  {\begin{tikzpicture}
	\begin{pgfonlayer}{nodelayer}
		\node [style=none] (41) at (-2, 2) {};
		\node [style=none] (42) at (-2, -2) {};
		\node [style=none] (43) at (0, 2) {};
		\node [style=none] (44) at (0, -2) {};
		\node [style=none] (45) at (0, 2) {};
		\node [style=none] (46) at (0, -2) {};
		\node [style=none] (47) at (2, 2) {};
		\node [style=none] (48) at (2, -2) {};
	\end{pgfonlayer}
	\begin{pgfonlayer}{edgelayer}
		\draw [style=fi-bl] (43.center)
			 to (41.center)
			 to (42.center)
			 to (44.center)
			 to cycle;
		\draw [style=fi-sh] (47.center)
			 to (45.center)
			 to (46.center)
			 to (48.center)
			 to cycle;
		\draw [style=vt-lbl] (46.center) to (45.center);
	\end{pgfonlayer}
\end{tikzpicture}}
  {%
  }}%
 ~~=~~ %
{\tikzstyle{every picture}=[tikzfig]
  {\begin{tikzpicture}
	\begin{pgfonlayer}{nodelayer}
		\node [style=none] (28) at (-2, 2) {};
		\node [style=none] (29) at (-2, -2) {};
		\node [style=none] (30) at (0, 1) {};
		\node [style=none] (31) at (0, -1) {};
		\node [style=none] (32) at (0, 2) {};
		\node [style=none] (33) at (0, -2) {};
		\node [style=none] (34) at (-1, 0) {};
		\node [style=none] (35) at (2, 2) {};
		\node [style=none] (36) at (2, -2) {};
		\node [style=none] (37) at (0, 1) {};
		\node [style=none] (38) at (0, -1) {};
		\node [style=none] (39) at (0, 2) {};
		\node [style=none] (40) at (0, -2) {};
		\node [style=none] (41) at (1, 0) {};
		\node [style=none] (42) at (0, 1) {};
		\node [style=none] (43) at (0, -1) {};
		\node [style=none] (44) at (-1, 0) {};
		\node [style=none] (45) at (1, 0) {};
		\node [style=none] (12) at (-2, 2) {};
		\node [style=none] (13) at (-2, -2) {};
		\node [style=none] (18) at (2, -2) {};
		\node [style=none] (19) at (2, 2) {};
		\node [style=none] (0) at (-2, 2) {};
		\node [style=none] (1) at (-2, -2) {};
		\node [style=none] (2) at (2, -2) {};
		\node [style=none] (3) at (2, 2) {};
		\node [style=none] (20) at (0, 1) {};
		\node [style=none] (21) at (0, -1) {};
		\node [style=none] (22) at (0, 2) {};
		\node [style=none] (23) at (0, -2) {};
		\node [style=none] (24) at (-1, 0) {};
		\node [style=none] (25) at (1, 0) {};
	\end{pgfonlayer}
	\begin{pgfonlayer}{edgelayer}
		\draw [style=fi-bl] (29.center)
			 to (28.center)
			 to (32.center)
			 to (30.center)
			 to [bend right=45] (34.center)
			 to [bend right=45] (31.center)
			 to (33.center)
			 to cycle;
		\draw [style=fi-sh] (36.center)
			 to (35.center)
			 to (39.center)
			 to (37.center)
			 to [bend left=45] (41.center)
			 to [bend left=45] (38.center)
			 to (40.center)
			 to cycle;
		\draw [style=fi-bl] (44.center)
			 to [bend right=45] (43.center)
			 to [bend right=45] (45.center)
			 to [bend right=45] (42.center)
			 to [bend right=45] cycle;
		\draw [style=vt-lbl, bend right=45] (20.center) to (24.center);
		\draw [style=vt-lbl, bend right=45] (24.center) to (21.center);
		\draw [style=vt-lbl, bend left=45] (20.center) to (25.center);
		\draw [style=vt-lbl, bend left=45] (25.center) to (21.center);
		\draw [style=vt-lbl] (21.center) to (23.center);
		\draw [style=vt-lbl] (20.center) to (22.center);
	\end{pgfonlayer}
\end{tikzpicture}}
  {%
  }}%

  \label{eq:absorbA}
  \ee
Here the bent edge of the graph that lies in the $\tu$-labeled area is labeled by 
$A$ as an object in $\cc$, while the other edges are labeled by $A$ as an 
$A$-$\tu$-bimodule. Invoking $\Delta$-separability of $A$ then gives the first
equality in \eqref{eq:absorbA}. The second equality follows in the same way,
except that in the picture on the right hand side, the left-most bent edge
is labeled by $A$ as a bimodule over itself.
\end{proof}

\begin{rem}
Expressions of the form of the right hand side of \eqref{eq:UCor-rhs} have been 
studied in \cite[Def.\,3.4]{fuSc22} under the name \emph{pinned block functor}. In
\cite{fuSc22} a candidate field object for bulk fields is selected, and the block 
functor is then evaluated after inserting that specific object at all arguments.
\end{rem}

%%%%%%%%%%%%%%%%%%%%%%%%%%%%%%%%%%%%%%%%%%%%%%%%%%%%%%%%%%%%%%%%%%%%%%%%

\subsection{Conformal field theory correlators} \label{sec:cft1}

A motivating application of the string-net constructions outlined in Section
\ref{sec:fresh} is a well-struc\-tured approach to the construction of consistent
systems of correlators for rational conformal field theories. This application was
developed in \cite{yangYa3,fusY,fusY2}. Here we recall a few of its main aspects
and provide an interpretation of Theorem \ref{thm:UcorIso} in that context.

There are two distinct entities that are referred to as a
two-dimensional conformal field theory, or CFT, for short: chiral CFT and full CFT; 
see e.g.\ \cite{fswy} for a short review. A \emph{chiral} CFT is governed by a
vertex operator algebra $\mathfrak{V}$, which reflects conservation laws of the CFT.
The representations of $\mathfrak{V}$ form a linear monoidal category 
$\cc \eq \mathrm{Rep}(\mathfrak{V})$.
We focus our attention on \emph{rational} conformal field theories, in which case 
$\cc$ has the structure of a modular fusion category. The string-net construction 
described in Section \ref{sec:fresh} works in fact in the much more general setting
that $\cb$ is any \ko-linear pivotal bicategory, but at present its conformal field
theory ramifications are fully understood only in the rational case. (Considerable
efforts go, however, at the time of writing in getting control over larger classes of
two-dimensional conformal field theories see e.g.\ 
\cite{fuSc27,fssw2,woik7,hoRun2,yera2}.) 
A \emph{full} CFT combines a left and a right chiral CFT, which are controlled by
two modular fusion categories $\cc$ and $\cc'$. We restrict our attention to full 
CFTs for which $\cc'$ equals $\cc$ up to a reversal of the braiding and the ribbon
twist, i.e.\ $\cc' \eq \cc^{\rm rev} \eq \mathrm{Rep}(\mathfrak{V})^{\rm rev}$. 
In this case
we have, thanks to modularity, $\cc \,{\boxtimes}\, \cc^{\rm rev} \,{\simeq}\, \czc$
as modular fusion categories, so that the string-net construction produces the
appropriate cylinder category for the circle, see \eqref{eq:blc-I-S1}.

Having fixed the modular category $\cc$, we can relate string-net constructions
to further CFT terminology. In physics terms, a surface $\surf$ carrying a
$\cfrc$-labeled string-net is called a \emph{\worldsheet}. For example, the labels of
the 1-cells on the boundary of $\surf$ describe boundary conditions for the CFT on
the \worldsheet\ that are compatible with the chiral symmetry encoded in $\cc$, while
labels for 1-cells in the interior of $\surf$ describe types of topological defect 
lines. 

A central goal of any quantum field theory is to establish the existence of 
correlation functions and understand there properties. For full CFTs with 
$\cc' \eq \cc^{\rm rev}$ this has been achieved in the following setting: 
Starting from the vertex algebra $\mathfrak{V}$ whose representation category is 
$\cc$, one can construct \emph{conformal blocks}. These are vector
bundles, endowed with a projectively flat connection, over the moduli space of
complex structures and marked points on the \worldsheet\ $\surf$. Thanks to the
projectively flat connection, the conformal blocks are characterized, up to 
isomorphism, by a representation of the mapping class group of $\surf$. The monodromy
data provided by the conformal blocks are captured by an open-closed modular functor
$\SNb$ \eqref{eq:def:SNb}, with the bicategory $\cb$ being the delooping $\cbc$ of
the modular fusion category $\cc$.

The correlation functions of the full CFT are global horizontal sections in these
bundles of conformal blocks. In the description of conformal blocks via an
open-closed modular functor, the correlation functions are distinguished elements in
vector spaces of conformal blocks that are invariant under the action of a suitable
subgroup \Cite{Eq.\,(6.11)}{fusY} of the mapping class group and are compatible 
with sewing, i.e.\ with composition of 1-morphisms in the bordism category.
Such a specific vector is obtained by regarding the \worldsheet\ $\surf$ as a 
$\cfrc$-co\-lored string net with boundary data $a$ and $a'$ that encode the fields
whose correlator is considered, and thereby as an element of the string-net space
$\SNfrc(\surf;a,a')$. (This element was called a \emph{quantum \worldsheet} in 
\cite{fusY}. As seen in \Cite{Rem.\,4.9}{fusY2}, the modular functor $\SNfrc$ 
classifies decorated \worldsheet s up to a local graphical calculus of defects.)
By applying the field functor to the boundary values and the universal
correlator \eqref{eq:ucorc-surf-aa'} to the distinguished vector, one obtains an
element of the vector space $\SNc(\surf;\df_{\!\ell}\,a,\df_{\!\ell'}\,a')$ that
captures the monodromy data of the conformal blocks with insertions on the complex
double $\surfhat$, and this element is the correlator for the \worldsheet\ $\surf$.

\medskip

The resulting interpretation of Theorem \ref{thm:UcorIso} in the CFT setting 
is as follows. Via the vertex algebra $\mathfrak{V}$, the conservation laws
of the chiral CFT with $\cc \eq \mathrm{Rep}(\mathfrak{V})$ can be expressed
as linear differential equations, known as Ward 
identities, that are satisfied by the conformal blocks.
It is a priori by no means clear whether every solution to these 
equations plays a role in full conformal field theory. The surjectivity of the
universal correlator $\ucorc(\surf)_{a,a'}$ stated in Theorem \ref{thm:UcorIso}
means that this is in fact the case: \emph{any} $\cc$-labeled string net on a 
\worldsheet\ $\surf$ with boundary data $(a,a')$, can be expressed as a linear 
combination of CFT correlators, provided that
correlators involving arbitrary networks of defect lines, and thereby also arbitrary 
types of defect fields, are admitted. Thinking about spaces of $\cc$-labeled string 
nets as spaces of solutions to the Ward identities, this implies that every 
solution to these differential equations contributes to some CFT correlator --
there are no spurious solutions. In short, \emph{CFT correlators exhaust the 
complete space of conformal blocks.} This is a very strong statement: in contrast,
the correlators for ordinary bulk fields, to which attention has traditionally often
been confined in the past, typically make use of only a small subspace of the space
of all solutions. (In the early literature dealing with concrete solutions to the 
Ward identities, the latter feature was well known, but not picked out as an
important theme.)

To understand the bearing of injectivity of the universal correlator in the CFT 
context, consider two \worldsheet s that share the same surface $\surf$,
which give rise to two vectors in $\SNfrc(\surf;a,a')$. Suppose that 
the universal correlator maps these two vectors to one and the same
$\cc$-colored string net. This means that the two \worldsheet s have
the same correlator, and hence they cannot be distinguished by any `measurement'.
Injectivity of the universal correlator means that the quantum \worldsheet s
determined by the two \worldsheet s must be the same, and hence that they are
related by a move in the bicategorical graphical calculus of defects. 
It follows that all relations among the correlators of a full 
CFT are local field-the\-o\-retical relations that come from the skein calculus 
of topological defects -- there aren't any additional relations between correlators
beyond the local relations implied by the string-net construction.
In this sense, Theorem \ref{thm:UcorIso} leads to direct
physical consequences, in terms of correlators, of the calculus of defects.

%%%%%%%%%%%%%%%%%%%%%%%%%%%%%%%%%%%%%%%%%%%%%%%%%%%%%%%%%%%%%%%%%%%%%%%%

\section{Double categorical structures} \label{sec:dbl-1}

In this Section we summarize some known results from the theory of (monoidal) double
categories. Section \ref{sec:dbl-1.1} contains basic definitions, in particular the 
one of a vertical transformation between double functors, that will be crucial for the
combined conceptual understanding of the universal correlator and the field functor.
Section \ref{sec:dbl-1.2} collects facts about monoidal double categories and
introduces the two basic examples of symmetric monoidal double categories we need:
the double category $\dbord$ of open-closed two-dimensional bordisms in Example
\ref{exa:dbord}, and the double category $\dprof$ of profunctors in Example
\ref{exa:dprof}. Section \ref{sec:dbl-1.3} explains the notions of companions and
conjoints, which are used, following \cite{wesSh}, to obtain bicategories from
double categories.

%%%%%%%%%%%%%%%%%%%%%%%%%%%%%%%%%%%%%%%%%%%%%%%%%%%%%%%%%%%%%%%%%%%%%%%%

\subsection{Double categories} \label{sec:dbl-1.1}

A \emph{$($pseudo$)$ double category} is a pseudocategory object internal to the
2-category $\ccat$ of small categories, functors and natural
transformations; a \emph{strict double category} is a category object internal to 
the category $\cat$ of small categories and functors. 
Every double category can be (functorially) strictified \Cite{Thm.\,7.5}{grPar};
accordingly we will treat every double category, as well as every double functor,
as if they were strict. Partly unraveling the condensed definition, we have:

\begin{defn}
A double category $\da$ consists of the following data:
\Itemize
  \item 
a collection of \emph{objects} $a,b,c,...\iN\da$;
  \item 
a collection of \emph{vertical $1$-morphisms} $f\colon a\Rarr~ b$, 
$g\colon b\Rarr~ c ,...$
  \item 
a composition of vertical $1$-morphisms;
  \item 
a collection of \emph{horizontal $1$-morphisms} $P\colon a\ptO b$,
$Q\colon b\ptO c, ...$ 
  \item 
a composition of horizontal $1$-morphisms;
  \item 
a collection of \emph{$2$-morphisms}\,
 $\begin{tikzcd}[row sep=1.1em,column sep=1.3em]	
 a & b \\ {a'} & {b'} 	
 \arrow[""{name=0,anchor=center,inner sep=0}, "P", "\shortmid"{marking}, from=1-1, to=1-2] 
 \arrow[""{name=1,anchor=center,inner sep=0}, "Q"', "\shortmid"{marking}, from=2-1, to=2-2] 
 \arrow["f"', from=1-1, to=2-1] \arrow["g", from=1-2, to=2-2]
 \arrow["\displaystyle\alpha"{description}, draw=none, from=0, to=1]
 \end{tikzcd}$, ... 
  \item 
vertical and horizontal compositions of $2$-morphisms, e.g.\ (suppressing labels of
1-mor\-phisms)
  \be
  \begin{tikzcd}[row sep=2.1em]
  a & b \\ ~ & ~ \\ {a''} & {b''} 
  \arrow[""{name=0,anchor=center,inner sep=0}, "\shortmid"{marking}, from=1-1, to=1-2]
  \arrow[from=1-2, to=3-2] 
  \arrow[""{name=1,anchor=center,inner sep=0}, "\shortmid"{marking}, from=3-1, to=3-2]
  \arrow[from=1-1, to=3-1]
  \arrow["\displaystyle\alpha'\cir\alpha"{description}, draw=none, from=0, to=1]
  \end{tikzcd}
  \quad \raisebox{-4pt}{$=$} \quad
  \begin{tikzcd}
  a & b \\ {a'} & {b'} \\ {a''} & {b''} 
  \arrow[""{name=0,anchor=center,inner sep=0}, "\shortmid"{marking}, from=1-1, to=1-2]
  \arrow[from=1-2, to=2-2] \arrow[from=2-2, to=3-2]
  \arrow[""{name=1,anchor=center,inner sep=0}, "\shortmid"{marking}, from=3-1, to=3-2]
  \arrow[from=1-1, to=2-1] \arrow[from=2-1, to=3-1] 
  \arrow[""{name=6,anchor=center,inner sep=0}, "\shortmid"{marking}, from=2-1, to=2-2]
  \arrow["\displaystyle\alpha"{description}, draw=none, from=0, to=6]
  \arrow["\displaystyle{\alpha'}"{description}, draw=none, from=6, to=1] 
  \end{tikzcd}
  \ee
and
  \be
  \begin{tikzcd}
  a & ~ & c \\ {a'} & \phantom{b'} & {c'} 
  \arrow[from=1-1, to=2-1] \arrow[from=1-3, to=2-3]
  \arrow[""{name=2,anchor=center,inner sep=0}, "\shortmid"{marking}, from=2-1, to=2-3]
  \arrow[""{name=4,anchor=center,inner sep=0}, "\shortmid"{marking}, from=1-1, to=1-3]
  \arrow["\displaystyle{\alpha\,{\cdot}\,\beta}"{description},draw=none, from=4, to=2]
  \end{tikzcd}
  \quad \raisebox{-2pt}{$=$} \quad
  \begin{tikzcd}
  a & b & c \\ {a'} & {b'} & {c'} 
  \arrow[from=1-1, to=2-1]
  \arrow[""{name=2,anchor=center,inner sep=0}, "\shortmid"{marking}, from=2-1, to=2-2]
  \arrow[""{name=3,anchor=center,inner sep=0}, "\shortmid"{marking}, from=2-2, to=2-3]
  \arrow[""{name=4,anchor=center,inner sep=0}, "\shortmid"{marking}, from=1-1, to=1-2]
  \arrow[""{name=5,anchor=center,inner sep=0}, "\shortmid"{marking}, from=1-2, to=1-3]
  \arrow[from=1-2, to=2-2] \arrow[from=1-3, to=2-3]
  \arrow["\displaystyle\alpha"{description}, draw=none, from=4, to=2]
  \arrow["\displaystyle\beta"{description}, draw=none, from=5, to=3]
  \end{tikzcd}
  \ee
\end{itemize}

\noindent
The composition of vertical $1$-morphisms, which we denote by `$\circ$', is strictly 
unital and associative, while the composition of horizontal $1$-morphisms, to be 
denoted by `$\cdot$', is weakly unital and associative.
We denote the weak identity horizontal $1$-morphism for an object $a$ by $U_a$,
and the identity vertical $1$-morphism for $a$ by $1_a$.
 \\[2pt]
Similarly, the vertical composition of 2-morphisms is strictly unital and associative,
while their horizontal composition is unital and associative only up to coherent 
isomorphisms. Moreover the two types of compositions are compatible, as expressed by
a \emph{middle-four interchange law} akin to the one for bicategories.
\end{defn}

A $2$-morphism for which the two vertical $1$-morphisms are identities is called
\emph{globular}.
Note that there is no composition of a vertical and a horizontal 1-morphism. Thus in
particular the squares used to depict 2-morphisms must not be interpreted as 
commutative diagrams.\,%
 \footnote{~For any category $\ca$ there is a double category $\square\ca$ whose
 vertical and horizontal 1-morphisms are both just the morphisms in $\ca$ and
 whose 2-morphisms are given by commutative squares in $\ca$. However,
 the composition of a vertical morphism with a horizontal morphism,
 though possible owing to the contents of $\square\ca$, is not
 intrinsic to the framework of double categories.}.

The framework of internal categories also provides the notions of a
\emph{$($pseudo$)$ double functor} between double categories and of a \emph{vertical 
transformation} (or \emph{double transformation}) between double functors.
For the full specification of the notion of a double functor 
we refer to Definition A.2.8 in \cite{cours}. In a nutshell, we have:

\begin{defn} \label{def:doublefunctor}
Let $\da$ and $\db$ be double categories. A \emph{double functor}
$F\colon\da\Rarr~\db$ maps objects, vertical 1-morphisms and horizontal 1-morphisms
in $\da$ to the corresponding items in $\db$ in a manner compatible with their
compositions. Further, it maps
a 2-morphism
in $\da$ to a 2-morphism in $\db$ according to
  \be
  \begin{tikzcd}[column sep=2.7em]
  a & b & Fa & Fb 
  \\
  {a'} & {b'} & {Fa'} & {Fb'}
  \arrow["f"', from=1-1, to=2-1]
  \arrow[""{name=0,anchor=center,inner sep=0},"Q"', "\shortmid"{marking}, from=2-1, to=2-2]
  \arrow[""{name=1,anchor=center,inner sep=0},"P", "\shortmid"{marking}, from=1-1, to=1-2]
  \arrow[""{name=2,anchor=center,inner sep=0},"g", from=1-2, to=2-2]
  \arrow[""{name=3,anchor=center,inner sep=0},"FP", "\shortmid"{marking}, from=1-3, to=1-4]
  \arrow[""{name=4,anchor=center,inner sep=0},"Ff"', from=1-3, to=2-3]
  \arrow[""{name=5,anchor=center,inner sep=0},"FQ"', "\shortmid"{marking}, from=2-3, to=2-4]
  \arrow["Fg", from=1-4, to=2-4]
  \arrow["\displaystyle\alpha"{description}, draw=none, from=1, to=0]
  \arrow["\!\!\!\xmapsto{\phantom{ww}}"{description}, draw=none, from=2, to=4]
  \arrow["\displaystyle F\alpha"{description}, draw=none, from=3, to=5]
  \end{tikzcd}
  \ee
in such a way that (along with the preservation of units) both the vertical
and the horizontal compositions are preserved, i.e.
  \be
  \begin{tikzcd}[column sep=3.7em]
  Fa & Fb & Fa & Fb \\ {Fa'} & {Fb'} & {Fa'} & {Fb'} 
  \\
  {Fa''} & {Fb''} & {Fa''} & {Fb''}
  \arrow[""{name=0,anchor=center,inner sep=0},"\shortmid"{marking}, from=1-1, to=1-2]
  \arrow[""{name=1,anchor=center,inner sep=0},"\shortmid"{marking}, from=1-3, to=1-4]
  \arrow[""{name=2,anchor=center,inner sep=0},"\shortmid"{marking}, from=2-1, to=2-2]
  \arrow[from=1-1, to=2-1]
  \arrow[from=1-2, to=2-2]
  \arrow[from=1-3, to=2-3]
  \arrow["\displaystyle ~=~"{description}, draw=none, from=2-2, to=2-3]
  \arrow[from=1-4, to=2-4]
  \arrow[from=2-4, to=3-4]
  \arrow[""{name=3,anchor=center,inner sep=0},"\shortmid"{marking}, from=3-3, to=3-4]
  \arrow[from=2-3, to=3-3]
  \arrow[from=2-2, to=3-2]
  \arrow[""{name=4,anchor=center,inner sep=0},"\shortmid"{marking}, from=3-1, to=3-2]
  \arrow[from=2-1, to=3-1]
  \arrow["\displaystyle F\alpha"{description}, draw=none, from=0, to=2]
  \arrow["\displaystyle{F(\alpha'{\circ}\alpha)}"{description},draw=none,from=1,to=3]
  \arrow["\displaystyle{F\alpha'}"{description}, draw=none, from=2, to=4]
  \end{tikzcd}
  \label{eq:Fver}
  \ee
and
  \be
  \begin{tikzcd}
  Fa & Fb & Fc & Fa & Fb & Fc
  \\
  {Fa'} & {Fb'} & {Fc'} & {Fa'} & {Fb'} & {Fc'}
  \arrow[""{name=0, anchor=center, inner sep=0}, "\shortmid"{marking}, from=1-1, to=1-2]
  \arrow["\shortmid"{marking}, from=1-4, to=1-5]
  \arrow[""{name=1,anchor=center,inner sep=0}, "\shortmid"{marking}, from=2-1, to=2-2]
  \arrow[from=1-1, to=2-1]
  \arrow[from=1-2, to=2-2]
  \arrow[""{name=2,anchor=center,inner sep=0}, from=1-4, to=2-4]
  \arrow["\shortmid"{marking}, from=2-4, to=2-5]
  \arrow[""{name=3,anchor=center,inner sep=0}, "\shortmid"{marking}, from=1-2, to=1-3]
  \arrow[""{name=4,anchor=center,inner sep=0}, "\shortmid"{marking}, from=2-2, to=2-3]
  \arrow[""{name=5,anchor=center,inner sep=0}, from=1-3, to=2-3]
  \arrow["\shortmid"{marking}, from=1-5, to=1-6]
  \arrow[from=1-6, to=2-6]
  \arrow["\shortmid"{marking}, from=2-5, to=2-6]
  \arrow["\displaystyle{F(\alpha\,{\cdot}\,\beta)}"{description},draw=none,from=1-5,to=2-5]
  \arrow["\displaystyle F\alpha"{description}, draw=none, from=0, to=1]
  \arrow["\displaystyle F\beta"{description}, draw=none, from=3, to=4]
  \arrow["\displaystyle ~=~"{description}, draw=none, from=5, to=2]
  \end{tikzcd}
  \label{eq:Fhor}
  \ee
for all composable pairs of 2-morphisms. 
\end{defn}

\begin{rem}
As before, in Definition \ref{def:doublefunctor} we tacitly invoke the
strictification theorem \Cite{Thm.\,7.5}{grPar} for double categories and double
functors. In full generality, the equality \eqref{eq:Fhor}
as well as the preservation of the weak identity horizontal 1-morphisms are
instead valid only up to insertions of appropriate globular comparison 
constraints.  If there are such constraints $U_{Fa} \Rarr~ F(U_a)$ for every 
weak identity horizontal morphism $U_a$ and 
$F(P) \,{\cdot}\, F(Q) \Rarr~ F(P \,{\cdot}\, Q)$ for every composable pair $P,Q$ of 
horizontal 1-cells, then we deal with a \emph{lax} double functor, while if there are
constraints in the opposite direction, then we have an \emph{oplax} double functor.
For a \emph{pseudo} double functor the constraints are isomorphisms, and for a
\emph{strong} double functor they are globular isomorphisms.
\end{rem}
 
Vertical transformations between double functors are described as follows
\Cite{Def.\,A.2.9}{cours} (again we invoke the strictification
of double categories as well as of double functors):

\begin{defn}
\label{def:doubletrans}
Let $F,G\colon\da\Rarr~\db$ be two double functors. A \emph{vertical transformation}
(or \emph{double transformation}) $\theta\colon F\,{\xRightarrow{\phantom{xx}}}\, G$
consists of two pieces of data: a family
  \be
  \{\theta_{a}\colon Fa\Rarr~ Ga\}_{a\in\da}^{}
  \label{eq:family-theta}
  \ee
of vertical 1-morphisms in $\db$, and a family 
  \be
  \begin{tikzcd}
  Fa & Fb \\ Ga & Gb
  \arrow[""{name=0,anchor=center,inner sep=0},"FP", "\shortmid"{marking}, from=1-1, to=1-2]
  \arrow[""{name=1,anchor=center,inner sep=0},"GP"', "\shortmid"{marking}, from=2-1, to=2-2]
  \arrow["{\theta_a}"', from=1-1, to=2-1]
  \arrow["{\theta_b}", from=1-2, to=2-2]
  \arrow["\displaystyle{\theta_P}"{description}, draw=none, from=0, to=1] 
  \end{tikzcd}
  \label{eq:family-2mora}
  \ee
of 2-morphisms in $\db$ that is parametrized by the horizontal 1-morphisms 
$P\colon a\ptO b$ in $\da$.
These data are required to satisfy the following conditions:
 \Itemize
 \item
\emph{Horizontal functoriality}: 
  \be
  \begin{tikzcd}
  Fa & Fb & Fc & Fa && Fc \\ Ga & Gb & Gc & Ga && Gc
  \arrow[""{name=0,anchor=center,inner sep=0},"FP", "\shortmid"{marking}, from=1-1, to=1-2]
  \arrow[""{name=1,anchor=center,inner sep=0},"GP"', "\shortmid"{marking}, from=2-1, to=2-2]
  \arrow["{\theta_a}"', from=1-1, to=2-1]
  \arrow["{\theta_b}", from=1-2, to=2-2]
  \arrow[""{name=2,anchor=center,inner sep=0},"GQ"', "\shortmid"{marking}, from=2-2, to=2-3]
  \arrow[""{name=3,anchor=center,inner sep=0},"FQ", "\shortmid"{marking}, from=1-2, to=1-3]
  \arrow[""{name=4,anchor=center,inner sep=0}, "{\theta_c}", from=1-3, to=2-3]
  \arrow[""{name=5,anchor=center,inner sep=0}, "{\theta_a}"', from=1-4, to=2-4]
  \arrow[""{name=6,anchor=center,inner sep=0}, "{F(P\cdot Q)}", "\shortmid"{marking}, from=1-4, to=1-6]
  \arrow["{\theta_c}", from=1-6, to=2-6]
  \arrow[""{name=7,anchor=center,inner sep=0},"{G(P\cdot Q)}"', "\shortmid"{marking}, from=2-4, to=2-6]
  \arrow["\displaystyle{\theta_P}"{description}, draw=none, from=0, to=1]
  \arrow["\displaystyle{\theta_Q}"{description}, draw=none, from=3, to=2]
  \arrow["\displaystyle{\theta_{P\cdot Q}}"{description}, draw=none, from=6, to=7]
  \arrow["\displaystyle ~=~"{description}, draw=none, from=4, to=5]
  \end{tikzcd}
  \ee
 \item
\emph{Vertical naturality}: 
  \be
  \begin{tikzcd}
  Fa & Fb & Fa & Fb \\ {Fa'} & {Fb'} & Ga & Gb \\ {Ga'} & {Gb'} & {Ga'} & {Gb'}
  \arrow[""{name=0,anchor=center,inner sep=0},"{FP'}"', "\shortmid"{marking}, from=2-1, to=2-2]
  \arrow[""{name=1,anchor=center,inner sep=0},"{GP'}"', "\shortmid"{marking}, from=3-1, to=3-2]
  \arrow["{\theta_{a'}}"', from=2-1, to=3-1]
  \arrow["{\theta_{b'}}", from=2-2, to=3-2]
  \arrow[""{name=2,anchor=center,inner sep=0},"FP", "\shortmid"{marking}, from=1-3, to=1-4]
  \arrow[""{name=3,anchor=center,inner sep=0},"FP", "\shortmid"{marking}, from=1-1, to=1-2]
  \arrow["Fg", from=1-2, to=2-2]
  \arrow["Ff"', from=1-1, to=2-1]
  \arrow[""{name=4,anchor=center,inner sep=0},"{GP'}"', "\shortmid"{marking}, from=3-3, to=3-4]
  \arrow["{\theta_a}"', from=1-3, to=2-3]
  \arrow["{\theta_b}", from=1-4, to=2-4]
  \arrow["Gf"', from=2-3, to=3-3]
  \arrow["Gg", from=2-4, to=3-4]
  \arrow[""{name=5, anchor=center, inner sep=0}, "GP"', "\shortmid"{marking}, from=2-3, to=2-4]
  \arrow["\displaystyle ~~=~~"{description}, draw=none, from=2-2, to=2-3]
  \arrow["\displaystyle{\theta_{P'}}"{description, pos=0.6}, draw=none, from=0, to=1]
  \arrow["\displaystyle{\theta_P}"{description}, draw=none, from=2, to=5]
  \arrow["\displaystyle F\alpha"{description}, draw=none, from=3, to=0]
  \arrow["\displaystyle G\alpha"{description, pos=0.6}, draw=none, from=5, to=4]
  \end{tikzcd}
  \label{eq:vertnat}
  \ee
for any 2-morphism $\alpha$.
 \item
\emph{Horizontal unitality}: 
  \be
  \begin{tikzcd}
  Fa & Fa & Fa & Fa \\ Ga & Ga & Ga & Ga
  \arrow[""{name=0, anchor=center, inner sep=0}, "~{F(U_a)}", "\shortmid"{marking}, from=1-1, to=1-2]
  \arrow[""{name=1, anchor=center, inner sep=0}, "{\theta_a}", from=1-2, to=2-2]
  \arrow["{\theta_a}"', from=1-1, to=2-1]
  \arrow[""{name=2, anchor=center, inner sep=0}, "\,{G(U_a)}"', "\shortmid"{marking}, from=2-1, to=2-2]
  \arrow[""{name=3, anchor=center, inner sep=0}, "{~~U_{\!Fa}}", "\shortmid"{marking}, from=1-3, to=1-4]
  \arrow["{\theta_a}", from=1-4, to=2-4]
  \arrow[""{name=4, anchor=center, inner sep=0}, "{\theta_a}"', from=1-3, to=2-3]
  \arrow[""{name=5, anchor=center, inner sep=0}, "{U_{Ga}}"', "\shortmid"{marking}, from=2-3, to=2-4]
  \arrow["\displaystyle ~~=~~"{description}, draw=none, from=1, to=4]
  \arrow["\displaystyle{\theta_{U_a}^{\phantom x}}"{description}, draw=none, from=0, to=2]
  \arrow["\displaystyle{U_{\theta_a}}"{description}, draw=none, from=3, to=5]
  \end{tikzcd}
  \ee
\end{itemize}
\end{defn}

To a double category $\da$ there are naturally associated two ordinary categories
$\da_0$ and $\da_1$: 

\begin{defn} \label{def:A0,A1}
Let $\da$ be a double category. The category $\da_0$ is the category with objects 
being the objects of $\da$ and morphisms being the vertical 1-morphisms of $\da$. 
The category $\da_1$ has as objects the horizontal 1-morphisms of $\da$, and as 
morphisms the globular 2-morphisms.
\end{defn} 

\begin{rem} \label{rem:theta1theta0}
Thinking about a double category as a category internal to $\ccat$, $\da_0$ is the
object category and $\da_1$ the morphism category.
 \\
It follows directly from the definition that a double functor $F\colon\da\Rarr~\db$
gives in particular rise to two ordinary functors $F_0\colon\da_0\Rarr~\db_0$ and
$F_1\colon\da_1\Rarr~\db_1$ between the categories $\da_0$ and $\da_1$.
 \\
The vertical naturality requirement \eqref{eq:vertnat} expresses the fact that
a vertical transformation $\theta_1 \colon F\,{\xRightarrow{\phantom{xx}}}\, G$ 
between double functors $F,G\colon\da\Rarr~\db$ furnishes in particular an ordinary 
natural transformation $\theta\colon F_1\,{\xRightarrow{\phantom{xx}}}\, G_1$ between
functors from $\da_1$ to $\db_1$. Moreover, by taking all horizontal morphisms in
\eqref{eq:vertnat} to be identities, $\theta$ also gives rise to an ordinary natural
transformation $\theta_0 \colon F_0\,{\xRightarrow{\phantom{xx}}}\, G_0$ 
between the corresponding functors from $\da_0$ to $\db_0$.
\end{rem}

%%%%%%%%%%%%%%%%%%%%%%%%%%%%%%%%%%%%%%%%%%%%%%%%%%%%%%%%%%%%%%%%%%%%%%%%
 
\subsection{Monoidal double categories} \label{sec:dbl-1.2}
 
We consider double categories together with double functors and vertical 
transformations, as a symmetric monoidal 2-category $\dbl$ by using the Cartesian
product. A symmetric monoidal double category is a symmetric pseudomonoid in
$\dbl$. Slightly more explicitly, we have (for a fully unpacked description see
Definitions A.2.12 and A.2.13 of \cite{cours}):
A \emph{pseudomonoid} internal to a monoidal 2-category $(\cb,\otimes,\tu)$
is an object $m$ in $\cb$ together with morphisms $m \oti m \Rarr~ m$ and 
$\tu \Rarr~ m$ such that the usual associativity and unit properties hold up to
specified 2-isomorphisms that satisfy appropriate equalities \Cite{Sect.\,3}{daSt}.
A \emph{braided pseudomonoid} in $(\cb,\otimes,\tu)$ is a pseudomonoid $m$
equipped with an isomorphism between $\otimes$ and the opposite tensor product
that obeys the usual hexagon identities of a braided monoidal category.
 % Definition A.2.11 in \cite{cours}

\begin{defn} \label{def:mondblecat}
A \emph{monoidal double category} is a double category $\da$ endowed with double 
functors $\otimes\colon \da\Times\da \Rarr~ \da$ and $I\colon {*} \Rarr~ \da$,
with ${*}$ the terminal double category, and with invertible vertical transformations
that are analogous to the associator and unitors of a monoidal category and obey
the same pentagon axiom and triangle identities as those.
\end{defn}

Note that in a vertical transformation we use vertical 1-morphisms. Accordingly, the
structural 1-morphisms in a monoidal double category are \emph{vertical} morphisms.

\begin{rem} \label{rem:A0,A1-mon}
If the double category $\da$ is monoidal, then the two categories $\da_0$ and $\da_1$
introduced in Definition \ref{def:A0,A1} are monoidal categories.
\end{rem}

We denote by $\tau$ the double functor $\da\Times\da \Rarr~ \da\Times\da$ that flips 
the order of the members of any pair of objects and of morphisms in $\da$ (i.e., of
pairs of objects and of morphisms in the categories $\da_0$ and $\da_1$).

\begin{defn}
A \emph{braided monoidal double category} is a monoidal double category $\da$ 
equipped with a \emph{braiding}, that is, with an invertible vertical transformation
$\beta\colon {\otimes} \,{\xRightarrow{\phantom{xx}}}\, {\otimes} \cir \tau$
that satisfies (on the nose) the same hexagon identities as the braiding in a monoidal
category.
 % see Definition A.2.13 in \cite{cours}.
 \\
A \emph{symmetric monoidal double category} is a braided one for which $\beta$
squares to the identity vertical transformation.
\end{defn}

The following two instances of symmetric monoidal double categories
are directly relevant to us.

\begin{example} \label{exa:dbord}
There is a symmetric monoidal double category $\dbord$ with the following data:
  \Itemize
  \item 
objects: compact oriented one-manifolds with possibly non-empty boundary;
  \item 
vertical 1-morphisms: orientation preserving smooth embeddings;
  \item 
horizontal 1-morphisms: oriented open-closed bordisms, with composition
given by sewing;
  \item 2-morphisms: isotopy classes of orientation preserving embeddings
that restrict to the embeddings of the parametrizing one-manifolds.
For example, there is a 2-morphism 
  \be
  \begin{tikzcd}
  {\mathbb{I}} & {\mathbb{I}} \\ {\mathbb{S}^1} & {(\mathbb{S}^1)^{\sqcup 2}}
  \arrow["f"', from=1-1, to=2-1]
  \arrow["g", from=1-2, to=2-2]
  \arrow[""{name=0,anchor=center,inner sep=0}, "\varSigma", "\shortmid"{marking}, from=1-1, to=1-2]
  \arrow[""{name=1,anchor=center,inner sep=0}, "{\varSigma'}"', "\shortmid"{marking}, from=2-1, to=2-2]
  \arrow["\displaystyle\xi"{description}, draw=none, from=0, to=1] \end{tikzcd}
  \ee
that looks like
  \be
  \raisebox{-0.9em}{\scalebox{0.7}{%
{\tikzstyle{every picture}=[tikzfig]
  
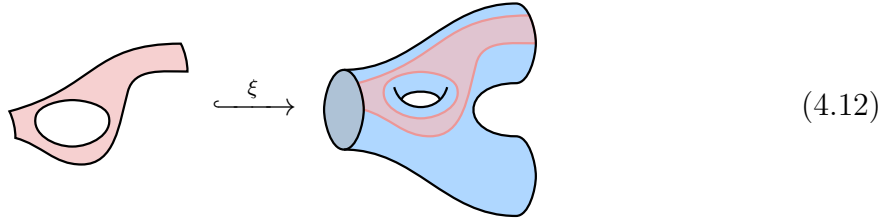
{\begin{tikzpicture}
	\begin{pgfonlayer}{nodelayer}
		\node [style=none] (132) at (-2.5, 0) {};
		\node [style=none] (133) at (4, 2.5) {};
		\node [style=none] (134) at (-2.725, 1) {};
		\node [style=none] (135) at (-1.725, 0.825) {};
		\node [style=none] (136) at (-0.35, 1.425) {};
		\node [style=none] (137) at (1.025, 0.825) {};
		\node [style=none] (139) at (3.75, 3.55) {};
		\node [style=none] (141) at (1.75, 1.75) {};
		\node [style=none] (142) at (1.025, 0.825) {};
		\node [style=none] (143) at (-0.35, -0.325) {};
		\node [style=none] (144) at (0, -1) {};
		\node [style=none] (145) at (1.75, 1.75) {};
		\node [style=none] (170) at (-2.5, 0) {};
		\node [style=none] (171) at (4, 2.5) {};
		\node [style=none] (172) at (-2.725, 1) {};
		\node [style=none] (173) at (-1.725, 0.825) {};
		\node [style=none] (174) at (-0.35, 1.425) {};
		\node [style=none] (176) at (3.75, 3.55) {};
		\node [style=none] (178) at (1.025, 0.825) {};
		\node [style=none] (179) at (-0.35, -0.325) {};
		\node [style=none] (180) at (0, -1) {};
		\node [style=none] (181) at (1.75, 1.75) {};
	\end{pgfonlayer}
	\begin{pgfonlayer}{edgelayer}
		\draw [style=fi-pi] (144.center)
			 to [in=-105, out=0] (145.center)
			 to (142.center)
			 to [in=0, out=-75] (143.center)
			 to [in=-105, out=180] (135.center)
			 to [in=-180, out=75, looseness=0.75] (136.center)
			 to [in=105, out=0, looseness=0.75] (137.center)
			 to (141.center)
			 to [in=-180, out=75] (133.center)
			 to [bend right=15, looseness=0.75] (139.center)
			 to [in=15, out=-180, looseness=1.25] (134.center)
			 to [bend left=15, looseness=0.75] (132.center)
			 to [in=180, out=-15] cycle;
		\draw [style=vt] (172.center)
			 to [bend left=15, looseness=0.75] (170.center)
			 to [in=180, out=-15] (180.center)
			 to [in=-105, out=0] (181.center)
			 to [in=-180, out=75] (171.center)
			 to [bend right=15, looseness=0.75] (176.center)
			 to [in=15, out=-180, looseness=1.25] cycle;
		\draw [style=vt] (179.center)
			 to [in=-105, out=180] (173.center)
			 to [in=-180, out=75, looseness=0.75] (174.center)
			 to [in=105, out=0, looseness=0.75] (178.center)
			 to [in=0, out=-75] cycle;
	\end{pgfonlayer}
\end{tikzpicture}}
  {%
  }}%
}}
  \xhookrightarrow{~~\,\xi~~~}~ \scalebox{0.7}{%
{\tikzstyle{every picture}=[tikzfig]
  {\begin{tikzpicture}
	\begin{pgfonlayer}{nodelayer}
		\node [style=none] (146) at (-4, 0) {};
		\node [style=none] (149) at (-3.25, -1.5) {};
		\node [style=none] (154) at (-1.075, 0.375) {};
		\node [style=none] (155) at (3.25, 4) {};
		\node [style=none] (156) at (4, 2.5) {};
		\node [style=none] (157) at (3.25, 1) {};
		\node [style=none] (158) at (3.25, -1) {};
		\node [style=none] (159) at (4, -2.5) {};
		\node [style=none] (160) at (3.25, -4) {};
		\node [style=none] (165) at (-0.35, 0.1) {};
		\node [style=none] (166) at (0.375, 0.375) {};
		\node [style=none] (167) at (-1.075, 0.375) {};
		\node [style=none] (168) at (-4, 0) {};
		\node [style=none] (169) at (-3.25, 1.5) {};
		\node [style=none] (111) at (0.375, 0.375) {};
		\node [style=none] (112) at (-1.075, 0.375) {};
		\node [style=none] (161) at (-4, 0) {};
		\node [style=none] (162) at (-2.5, 0) {};
		\node [style=none] (163) at (-3.25, 1.5) {};
		\node [style=none] (164) at (-3.25, -1.5) {};
		\node [style=none] (182) at (-2.5, 0) {};
		\node [style=none] (183) at (4, 2.5) {};
		\node [style=none] (184) at (-2.725, 1) {};
		\node [style=none] (185) at (-1.725, 0.825) {};
		\node [style=none] (186) at (-0.35, 1.425) {};
		\node [style=none] (187) at (1.025, 0.825) {};
		\node [style=none] (188) at (3.75, 3.55) {};
		\node [style=none] (189) at (1.75, 1.75) {};
		\node [style=none] (190) at (1.025, 0.825) {};
		\node [style=none] (191) at (-0.35, -0.325) {};
		\node [style=none] (192) at (0, -1) {};
		\node [style=none] (193) at (1.75, 1.75) {};
		\node [style=none] (108) at (0.65, 0.85) {};
		\node [style=none] (109) at (-1.35, 0.85) {};
		\node [style=none] (110) at (-0.35, 0.1) {};
		\node [style=none] (194) at (-2.5, 0) {};
		\node [style=none] (195) at (4, 2.5) {};
		\node [style=none] (196) at (-2.725, 1) {};
		\node [style=none] (197) at (-1.725, 0.825) {};
		\node [style=none] (198) at (-0.35, 1.425) {};
		\node [style=none] (199) at (3.75, 3.55) {};
		\node [style=none] (200) at (1.025, 0.825) {};
		\node [style=none] (201) at (-0.35, -0.325) {};
		\node [style=none] (202) at (0, -1) {};
		\node [style=none] (203) at (1.75, 1.75) {};
		\node [style=none] (100) at (-4, 0) {};
		\node [style=none] (101) at (-2.5, 0) {};
		\node [style=none] (103) at (-3.25, 1.5) {};
		\node [style=none] (105) at (-3.25, -1.5) {};
		\node [style=none] (115) at (3.25, 4) {};
		\node [style=none] (117) at (4, 2.5) {};
		\node [style=none] (118) at (3.25, 1) {};
		\node [style=none] (119) at (3.25, -1) {};
		\node [style=none] (121) at (4, -2.5) {};
		\node [style=none] (122) at (3.25, -4) {};
	\end{pgfonlayer}
	\begin{pgfonlayer}{edgelayer}
		\draw [style=fi-bl] (169.center)
			 to [in=180, out=0] (155.center)
			 to [in=90, out=0, looseness=0.75] (156.center)
			 to [in=0, out=-90, looseness=0.75] (157.center)
			 to [bend right=90, looseness=2.75] (158.center)
			 to [in=90, out=0, looseness=0.75] (159.center)
			 to [in=0, out=-90, looseness=0.75] (160.center)
			 to [in=0, out=180] (149.center)
			 to [in=-90, out=180, looseness=0.75] (146.center)
			 to (167.center)
			 to [in=180, out=-45, looseness=0.75] (165.center)
			 to [in=-135, out=0, looseness=0.75] (166.center)
			 to [bend right=75, looseness=0.75] (154.center)
			 to (168.center)
			 to [in=-180, out=90, looseness=0.75] cycle;
		\draw [style=op-lsh] (162.center)
			 to [in=0, out=90, looseness=0.75] (163.center)
			 to [in=90, out=-180, looseness=0.75] (161.center)
			 to [in=180, out=-90, looseness=0.75] (164.center)
			 to [in=-90, out=0, looseness=0.75] cycle;
		\draw [style=op-pi] (192.center)
			 to [in=-105, out=0] (193.center)
			 to (190.center)
			 to [in=0, out=-75] (191.center)
			 to [in=-105, out=180] (185.center)
			 to [in=-180, out=75, looseness=0.75] (186.center)
			 to [in=105, out=0, looseness=0.75] (187.center)
			 to (189.center)
			 to [in=-180, out=75] (183.center)
			 to [bend right=15, looseness=0.75] (188.center)
			 to [in=15, out=-180, looseness=1.25] (184.center)
			 to [bend left=15, looseness=0.75] (182.center)
			 to [in=180, out=-15] cycle;
		\draw [style=vt] (108.center)
			 to [in=0, out=-105] (110.center)
			 to [in=-75, out=180] (109.center);
		\draw [style=vt, bend right=75, looseness=0.75] (111.center) to (112.center);
		\draw [style=vt-pi] (196.center)
			 to [bend left=15, looseness=0.75] (194.center)
			 to [in=180, out=-15] (202.center)
			 to [in=-105, out=0] (203.center)
			 to [in=-180, out=75] (195.center)
			 to [bend right=15, looseness=0.75] (199.center)
			 to [in=15, out=-180, looseness=1.25] cycle;
		\draw [style=vt-pi] (201.center)
			 to [in=-105, out=180] (197.center)
			 to [in=-180, out=75, looseness=0.75] (198.center)
			 to [in=105, out=0, looseness=0.75] (200.center)
			 to [in=0, out=-75] cycle;
		\draw [style=vt] (103.center)
			 to [in=180, out=0] (115.center)
			 to [in=90, out=0, looseness=0.75] (117.center)
			 to [in=0, out=-90, looseness=0.75] (118.center)
			 to [bend right=90, looseness=2.75] (119.center)
			 to [in=90, out=0, looseness=0.75] (121.center)
			 to [in=0, out=-90, looseness=0.75] (122.center)
			 to [in=0, out=180] (105.center);
		\draw [style=vt] (101.center)
			 to [in=0, out=90, looseness=0.75] (103.center)
			 to [in=90, out=-180, looseness=0.75] (100.center)
			 to [in=180, out=-90, looseness=0.75] (105.center)
			 to [in=-90, out=0, looseness=0.75] cycle;
	\end{pgfonlayer}
\end{tikzpicture}}
  {%
  }}%
}
  \ee
 \item 
monoidal product: disjoint union.
\end{itemize}
\end{example}

\begin{rem}
The possibility to work with a double category to capture features of bordisms 
(specifically, to circumvent the use of collars, which strongly relies on invoking 
the axiom of choice) was also noticed, albeit not pursued, in  Section 4.2 of
\cite{schom}.
 % Morton cite not sensible 
\end{rem}

\begin{rem}
In a sense, the double category $\dbord$ encodes two types of \emph{locality}
properties: the locality via cutting and pasting, which is a feature of 
functorial quantum field theories, and the one via embeddings, which is crucial
in operator algebraic approaches to quantum field theories.
In the operator algebraic setting, besides the inclusions of local algebras coming 
from inclusions of space-time regions, also bimodules over those algebras appear.
It has recently been proposed \cite{komal} that the compatibility conditions 
between these two types of structures are captured by a double functor from a 
double category of space-times to a double category of von Neumann algebras.
\end{rem}

\begin{example} \label{exa:dprof}
For any field \ko\ there is a symmetric 
monoidal double category $\dprof$ with the following defining data:
  \Itemize
  \item 
objects: essentially small \ko-linear categories;
  \item 
vertical 1-morphisms: \ko-linear functors;
  \item 
horizontal 1-morphisms: \ko-linear profunctors\,%
 \footnote{~Recall that in our convention, a profunctor $P\colon A\ptO B$ is a 
 functor $A^{\op}\otimes B\to\vct$,}
with composition given by coends;
  \item
2-morphisms: natural transformations, with functors inserted in the
target profunctor; for instance, the 2-morphism
  \be
  \begin{tikzcd}
  A & B \\ {A'} & {B'}
  \arrow["F"', from=1-1, to=2-1]
  \arrow["G", from=1-2, to=2-2]
  \arrow[""{name=0,anchor=center,inner sep=0},"P", "\shortmid"{marking},from=1-1,to=1-2]
  \arrow[""{name=1,anchor=center,inner sep=0},"Q"',"\shortmid"{marking},from=2-1,to=2-2]
  \arrow["\displaystyle\varphi"{description}, draw=none, from=0, to=1]
  \end{tikzcd}
  \ee
is given by a natural transformation 
  \be
  \varphi\Colon P(-,\sim)\xRightarrow{\phantom{xx}} Q(F(-),G(\sim)).
  \ee
  \item 
monoidal product: the naive tensor product for enriched categories.
\end{itemize}
\end{example}

\begin{defn} \label{def:monoidalverttrafo}
A \emph{monoidal vertical transformation} between monoidal double functors
$F\colon \da\Rarr~\db$ and $G\colon \da\Rarr~\db$ is a vertical transformation 
$\theta\colon F\,{\xRightarrow{\phantom{-}}}\, G$ for which the natural 
transformations $\theta_0\colon F_0\,{\xRightarrow{\phantom{...}}}\, G_0$ and 
$\theta_1\colon F_1\,{\xRightarrow{\phantom{...}}}\, G_1$ 
described in Remark \ref{rem:theta1theta0} are monoidal.
\end{defn}
 % see \Cite{Def.\,A.2.15}{cours}

There is also the notion of a symmetric monoidal double functor. 
This is a morphism of commutative pseudomonoids in $\dbl$. For a complete
definition we refer to Definition A.2.14 of \cite{cours}; 
 % and 2.14 of \cite{wesSh}
a condensed version is:

\begin{defn} \label{def:monoidaldoublefunctor}
A \emph{monoidal double functor} between monoidal double categories $\da$ and $\db$ 
is a double functor $F\colon\da\Rarr~\db$ together with 
a vertical 1-isomorphism $\tu_{\db} \Rarr~ F(\tu_{\da})$ and
with vertical 1-isomorphisms $F(a) \oti F(b) \Rarr~ F(a\oti b)$ for each pair 
of objects $a,b$ in $\da$ that are natural with respect to horizontal 1-morphisms
 % in \cite{cours} this is just called a natural isomorphism
as well as with an invertible 2-cell
$F(P) \oti F(Q) \,{\xRightarrow{\phantom{w}}}\, F(P\oti Q)$
for each pair of horizontal 1-mor\-phisms
$P,Q$ in $\da$, and with isomorphisms involving the respective monoidal units, that 
satisfy various compatibility conditions which extend those of monoidal bicategories.
\end{defn}

\begin{rem} \label{rem:F0,F1-mon}
For a monoidal double functor $F$, both $F_0$ and $F_1$ are monoidal functors.
\end{rem} 

\begin{defn} 
A \emph{braided monoidal double functor} is a monoidal double functor $F$ for which
the functors $F_0$ and $F_1$ are braided monoidal.
A \emph{symmetric monoidal double functor} is a monoidal double functor for which
$F_0$ and $F_1$ are symmetric monoidal.
\end{defn}

%%%%%%%%%%%%%%%%%%%%%%%%%%%%%%%%%%%%%%%%%%%%%%%%%%%%%%%%%%%%%%%%%%%%%%%%

\subsection{Companions, conjoints, and graphical calculus} \label{sec:dbl-1.3}

There is a natural notion of horizontal bicategory of a double category, see
Definition \ref{def:HA,VA} below.
In \cite{wesSh} sufficient conditions on symmetric monoidal double categories,
symmetric mo\-no\-i\-dal double functors and monoidal vertical transformations were
identified under which the (symmetric) monoidal structures can be lifted to the
horizontal bicategories, as well as to the functors and transformations between them. 
Doing so is not trivial, because by taking horizontal bicategories one discards
all the non-globular 2-cells, among which there are the 2-cells mentioned after 
Definition \ref{def:mondblecat} that mediate the symmetric monoidal structures. 

First we need the notions of companion and conjoint of a vertical 1-morphism in a
double category. These can be regarded as the double categorical analogues of 
isomorphisms and adjunctions between 1-morphisms in a 2-category. (More 
specifically, for any 2-category $\ca$, the companion pairs in the double category
$\mathbb{Q}\ca$ of \emph{quintets} are precisely the isomorphisms between
1-morphisms of $\ca$, and the conjoint pairs in $\mathbb{Q}\ca$ are precisely
the adjoint pairs in $\ca$.)

In diagrams for 2-morphisms we denote identity vertical 1-morphisms by an equality 
sign, and weak identity horizontal 1-morphisms by the symbol
``\hspace*{-1.1em}$\begin{tikzcd}[column sep=1.1em]~ 
\ar["\shortmid"{marking},Rightarrow,no head]{r}{} &~ \end{tikzcd}$\hspace*{-1.1em}''.

\begin{defn}
Let $f\colon a\to b$ be a vertical 1-morphism in a double category $\dd$. 
 \\[3pt]
A \emph{companion} of $f$ is a horizontal 1-morphism $\widehat{f}\colon a\ptO b$
together with 2-morphisms 
  \be
  \begin{tikzcd}
  a & b & & a & a
  \\
  b & b & & a & b
  \arrow[""{name=0, anchor=center, inner sep=0}, "{\widehat{f}}~", "\shortmid"{marking}, from=1-1, to=1-2]
  \arrow["f"', from=1-1, to=2-1]
  \arrow[""{name=1, anchor=center, inner sep=0}, "\shortmid"{marking}, Rightarrow, no head, from=2-1, to=2-2]
  \arrow[Rightarrow, no head, from=1-2, to=2-2]
  \arrow[Rightarrow, no head, from=1-4, to=2-4]
  \arrow[""{name=2, anchor=center, inner sep=0}, "{~\widehat{f}}"', "\shortmid"{marking}, from=2-4, to=2-5]
  \arrow[""{name=3, anchor=center, inner sep=0}, "\shortmid"{marking}, Rightarrow, no head, from=1-4, to=1-5]
  \arrow["f", from=1-5, to=2-5]
  \arrow["\displaystyle\varepsilon"{description}, draw=none, from=0, to=1]
  \arrow["\displaystyle\eta"{description}, draw=none, from=3, to=2]
  \arrow["{\displaystyle\text{~and~}}"{description}, draw=none, from=1, to=3]
  \end{tikzcd} 
  \ee
(called \emph{counit} and \emph{unit}) such that 
  \be
  \begin{tikzcd}
  a & a \\ a & b \\ b & b
  \arrow[Rightarrow, no head, from=1-1, to=2-1]
  \arrow["f"', from=2-1, to=3-1]
  \arrow[Rightarrow, no head, from=2-2, to=3-2]
  \arrow[""{name=0, anchor=center, inner sep=0}, "{\widehat{f}}"{description}, "\shortmid"{marking}, from=2-1, to=2-2]
  \arrow[""{name=1, anchor=center, inner sep=0}, "\shortmid"{marking}, Rightarrow, no head, from=3-1, to=3-2]
  \arrow[""{name=2, anchor=center, inner sep=0}, "\shortmid"{marking}, Rightarrow, no head, from=1-1, to=1-2]
  \arrow[""{name=3, anchor=center, inner sep=0}, "f", from=1-2, to=2-2]
  \arrow["\displaystyle\eta"{description}, draw=none, from=2, to=0]
  \arrow["\displaystyle\varepsilon"{description}, draw=none, from=0, to=1]
  \end{tikzcd}
  \raisebox{-1.8pt}{ ~~=~~ }
  \begin{tikzcd}[row sep=5.0em]
  a & a \\ b & b
  \arrow["f", from=1-2, to=2-2]
  \arrow[""{name=4, anchor=center, inner sep=0}, "f"', from=1-1, to=2-1]
  \arrow["\shortmid"{marking}, Rightarrow, no head, from=1-1, to=1-2]
  \arrow["\shortmid"{marking}, Rightarrow, no head, from=2-1, to=2-2]
  \end{tikzcd}
  ~~~~~\mbox{and}~~~~~
  \begin{tikzcd}[row sep=2.4em]
  a & a & b & a & b \\ a & b & b & a & b
  \arrow[""{name=5, anchor=center, inner sep=0}, "\shortmid"{marking}, Rightarrow, no head, from=1-1, to=1-2]
  \arrow[""{name=2, anchor=center, inner sep=0}, "f"{description}, from=1-2, to=2-2]
  \arrow[""{name=0, anchor=center, inner sep=0}, Rightarrow, no head, from=1-1, to=2-1]
  \arrow[""{name=1, anchor=center, inner sep=0}, "{~~\widehat{f}}"', "\shortmid"{marking}, from=2-1, to=2-2]
  \arrow[""{name=7, anchor=center, inner sep=0}, "{\widehat{f}~}", "\shortmid"{marking}, from=1-2, to=1-3]
  \arrow[""{name=8, anchor=center, inner sep=0}, Rightarrow, no head, from=1-3, to=2-3] 
  \arrow[""{name=9, anchor=center, inner sep=0}, "\shortmid"{marking}, Rightarrow, no head, from=2-2, to=2-3] 
  \arrow["{~~\widehat{f}}"', "\shortmid"{marking}, from=2-4, to=2-5]
  \arrow["{\widehat{f}~}", "\shortmid"{marking}, from=1-4, to=1-5]
  \arrow[Rightarrow, no head, from=1-5, to=2-5]
  \arrow[""{name=10, anchor=center, inner sep=0}, Rightarrow, no head, from=1-4, to=2-4]
  \arrow["\displaystyle\eta"{description}, draw=none, from=0, to=2]
  \arrow["\displaystyle\varepsilon"{description}, draw=none, from=7, to=9]
  \arrow["\displaystyle ~=~"{description}, draw=none, from=8, to=10]
  \end{tikzcd}
  ~~~~~
  \ee
where the unlabeled 2-cells are units of 1-cells.
 \\[3pt]
A \emph{conjoint} of $f$, denoted by $\check{f}\colon b\ptO a$,
is a companion of $f$ in the horizontal opposite $\dd^{\mathrm{h}\text{-}\op}$
of the double category $\dd$.
 \\[3pt]
A double category in which every vertical 1-morphism has both a companion and a 
conjoint is called a \emph{framed bicategory} \cite{shul4} or, equivalently,
a \emph{proarrow equipment} \cite{wood}. 
\end{defn}

Much like their bicategorical counterparts, a companion or conjoint of a vertical
1-mor\-phism in a double category, if it exists, is unique up to a unique
globular 2-isomorphism.

\begin{example}
The bordism double category $\dbord$ from Example \ref{exa:dbord} is a framed
bicategory: the companion and conjoint of an embedding are given by mapping 
cylinders regarded as bordisms.
\end{example}

\begin{example}
The profunctor double category $\dprof$ is a framed bicategory  -- it is in fact a 
prototypical one. The companion of a linear functor $F\colon A\Rarr~ B$ in $\dprof$ 
is the profunctor $F_{*}\Coloneqq \Hom_B(F{\sim},\backsim)\colon A \ptO B$, while the
conjoint of $F$ is the profunctor $F^{*}\,{\coloneqq}\,
\Hom_B(\sim,F{\backsim})\colon 
        $\linebreak[0]$
	B \ptO A$.
\end{example}

Double categories admit a graphical string calculus that works well with companions 
and conjoints. The following is a brief summary of this calculus; for the details 
see \cite{myerD}.\,%
 \footnote{~Beware, however, that our convention for the direction of 1-cells differs 
 from the one in \cite{myerD}: instead of transposing the Poincar\'e dual of a
 pasting diagram, we allow a horizontal 1-morphism to be drawn vertically, keeping
 in mind that it separates the 2-dimensional regions horizontally.}
Consider a 2-morphism
  \be
  \begin{tikzcd}
  a & b \\ c & d
  \arrow[""{name=0, anchor=center, inner sep=0}, "\,P", "\shortmid"{marking}, from=1-1, to=1-2]
  \arrow[""{name=1, anchor=center, inner sep=0}, "~\,Q"', "\shortmid"{marking}, from=2-1, to=2-2]
  \arrow["f"', from=1-1, to=2-1]
  \arrow["g", from=1-2, to=2-2]
  \arrow["\displaystyle\alpha"{description}, draw=none, from=0, to=1]
  \end{tikzcd} 
  \label{eq:pasti}
  \ee
in a double category $\dd$. The string diagram presentation
of $\alpha$ is the Poincar\'e dual of its pasting diagram \eqref{eq:pasti}: 
  \be
  \scalebox{1.2}{%
{\tikzstyle{every picture}=[tikzfig]
  {\begin{tikzpicture}
	\begin{pgfonlayer}{nodelayer}
		\node [style=none] (28) at (-1.5, 1.5) {};
		\node [style=none] (29) at (-1.5, -1.5) {};
		\node [style=none] (35) at (1.5, 1.5) {};
		\node [style=none] (36) at (1.5, -1.5) {};
		\node [style=nc] (37) at (0, 0) {};
		\node [style=none] (38) at (0, 1.5) {};
		\node [style=none] (39) at (-1.5, 0) {};
		\node [style=none] (40) at (1.5, 0) {};
		\node [style=none] (41) at (0, -1.5) {};
		\node [style=none] (42) at (-1.5, 1.5) {};
		\node [style=none] (43) at (0, 0) {};
		\node [style=none] (44) at (0, 1.5) {};
		\node [style=none] (45) at (-1.5, 0) {};
		\node [style=none] (46) at (0, 1.5) {};
		\node [style=none] (47) at (1.5, 0) {};
		\node [style=none] (48) at (1.5, 1.5) {};
		\node [style=none] (49) at (0, 0) {};
		\node [style=none] (50) at (0, 0) {};
		\node [style=none] (51) at (1.5, -1.5) {};
		\node [style=none] (52) at (1.5, 0) {};
		\node [style=none] (53) at (0, -1.5) {};
		\node [style=none] (54) at (-1.5, 0) {};
		\node [style=none] (55) at (0, -1.5) {};
		\node [style=none] (56) at (0, 0) {};
		\node [style=none] (57) at (-1.5, -1.5) {};
		\node [style=none] (58) at (0, 0) {$\alpha$};
		\node [style=none] (59) at (-2, 0) {$f$};
		\node [style=none] (60) at (0, 2) {$P$};
		\node [style=none] (61) at (2, 0) {$g$};
		\node [style=none] (62) at (0, -2) {$Q$};
	\end{pgfonlayer}
	\begin{pgfonlayer}{edgelayer}
		\draw [style=fi-pi] (44.center)
			 to (43.center)
			 to (45.center)
			 to (42.center)
			 to cycle;
		\draw [style=fi-bl] (48.center)
			 to (47.center)
			 to (49.center)
			 to (46.center)
			 to cycle;
		\draw [style=fi-pu] (52.center)
			 to (51.center)
			 to [in=360, out=180] (53.center)
			 to (50.center)
			 to cycle;
		\draw [style=fi-gr] (56.center)
			 to (55.center)
			 to (57.center)
			 to (54.center)
			 to cycle;
		\draw [style=vt-bl-di] (38.center) to (37);
		\draw [style=vt-bl-di] (37) to (41.center);
		\draw [style=vt-bl] (37) to (39.center);
		\draw [style=vt-bl] (40.center) to (37);
	\end{pgfonlayer}
\end{tikzpicture}}
  {%
  }}%
}
  \label{eq:nonpasti}
  \ee
Vertical and horizontal compositions are performed by concatenating the string 
diagrams vertically and horizontally, respectively. We omit the unit 1- or
2-morphisms in the string diagrams, analogously as we do with bicategories. We
depict the counit $\varepsilon$ and unit $\eta$ of a companion pair $(f,\widehat{f})$ by
  \be
{\tikzstyle{every picture}=[tikzfig]
  {\begin{tikzpicture}
	\begin{pgfonlayer}{nodelayer}
		\node [style=none] (63) at (-1.5, 1.5) {};
		\node [style=none] (65) at (-1.5, 0) {};
		\node [style=none] (66) at (0, 1.5) {};
		\node [style=none] (67) at (0, 0.5) {};
		\node [style=none] (68) at (-0.5, 0) {};
		\node [style=none] (69) at (-1.5, -1.5) {};
		\node [style=none] (70) at (1.5, 1.5) {};
		\node [style=none] (71) at (1.5, -1.5) {};
		\node [style=none] (73) at (-1.5, 0) {};
		\node [style=none] (74) at (1.5, 0) {};
		\node [style=none] (75) at (0, 1.5) {};
		\node [style=none] (76) at (0, 0.5) {};
		\node [style=none] (77) at (-0.5, 0) {};
		\node [style=none] (28) at (-1.5, 1.5) {};
		\node [style=none] (29) at (-1.5, -1.5) {};
		\node [style=none] (35) at (1.5, 1.5) {};
		\node [style=none] (36) at (1.5, -1.5) {};
		\node [style=none] (38) at (0, 1.5) {};
		\node [style=none] (39) at (-1.5, 0) {};
		\node [style=none] (40) at (1.5, 0) {};
		\node [style=none] (46) at (0, 1.5) {};
		\node [style=none] (59) at (-2, 0) {$f$};
		\node [style=none] (60) at (0, 2) {$\widehat{f}$};
		\node [style=none] (61) at (0, 0.5) {};
		\node [style=none] (62) at (-0.5, 0) {};
	\end{pgfonlayer}
	\begin{pgfonlayer}{edgelayer}
		\draw [style=fi-pi] (67.center)
			 to (66.center)
			 to (63.center)
			 to (65.center)
			 to (68.center)
			 to [bend right=45, looseness=1.25] cycle;
		\draw [style=fi-bl] (77.center)
			 to (73.center)
			 to (69.center)
			 to (71.center)
			 to (74.center)
			 to (70.center)
			 to (75.center)
			 to (76.center)
			 to [bend left=45, looseness=1.25] cycle;
		\draw [style=vt-bl] (46.center) to (61.center);
		\draw [style=vt-bl-di] (62.center) to (39.center);
		\draw [style=vt-bl, bend left=45, looseness=1.25] (61.center) to (62.center);
	\end{pgfonlayer}
\end{tikzpicture}}
  {%
  }}%
\qquad\text{and}\qquad%
{\tikzstyle{every picture}=[tikzfig]
  {\begin{tikzpicture}
	\begin{pgfonlayer}{nodelayer}
		\node [style=none] (63) at (1.5, -1.5) {};
		\node [style=none] (65) at (1.5, 0) {};
		\node [style=none] (66) at (0, -1.5) {};
		\node [style=none] (67) at (0, -0.5) {};
		\node [style=none] (68) at (0.5, 0) {};
		\node [style=none] (69) at (1.5, 1.5) {};
		\node [style=none] (70) at (-1.5, -1.5) {};
		\node [style=none] (71) at (-1.5, 1.5) {};
		\node [style=none] (73) at (1.5, 0) {};
		\node [style=none] (74) at (-1.5, 0) {};
		\node [style=none] (75) at (0, -1.5) {};
		\node [style=none] (76) at (0, -0.5) {};
		\node [style=none] (77) at (0.5, 0) {};
		\node [style=none] (28) at (1.5, -1.5) {};
		\node [style=none] (29) at (1.5, 1.5) {};
		\node [style=none] (35) at (-1.5, -1.5) {};
		\node [style=none] (36) at (-1.5, 1.5) {};
		\node [style=none] (38) at (0, -1.5) {};
		\node [style=none] (39) at (1.5, 0) {};
		\node [style=none] (40) at (-1.5, 0) {};
		\node [style=none] (46) at (0, -1.5) {};
		\node [style=none] (59) at (2, 0) {$f$};
		\node [style=none] (60) at (0, -2) {$\widehat{f}$};
		\node [style=none] (61) at (0, -0.5) {};
		\node [style=none] (62) at (0.5, 0) {};
	\end{pgfonlayer}
	\begin{pgfonlayer}{edgelayer}
		\draw [style=fi-bl] (67.center)
			 to (66.center)
			 to (63.center)
			 to (65.center)
			 to (68.center)
			 to [bend right=45, looseness=1.25] cycle;
		\draw [style=fi-pi] (77.center)
			 to (73.center)
			 to (69.center)
			 to (71.center)
			 to (74.center)
			 to (70.center)
			 to (75.center)
			 to (76.center)
			 to [bend left=45, looseness=1.25] cycle;
		\draw [style=vt-bl-di] (61.center) to (46.center);
		\draw [style=vt-bl] (62.center) to (39.center);
		\draw [style=vt-bl, bend left=45, looseness=1.25] (61.center) to (62.center);
	\end{pgfonlayer}
\end{tikzpicture}}
  {%
  }}%

  \label{eq:pic:companion1}
  \ee
Then the \emph{yanking equations} they satisfy can be expressed as 
  \be
{\tikzstyle{every picture}=[tikzfig]
  {\begin{tikzpicture}
	\begin{pgfonlayer}{nodelayer}
		\node [style=none] (63) at (1.5, 0) {};
		\node [style=none] (65) at (1.5, 0.75) {};
		\node [style=none] (66) at (0, 0) {};
		\node [style=none] (67) at (0, 0.25) {};
		\node [style=none] (68) at (0.5, 0.75) {};
		\node [style=none] (69) at (1.5, 1.5) {};
		\node [style=none] (70) at (-1.5, 0) {};
		\node [style=none] (71) at (-1.5, 1.5) {};
		\node [style=none] (73) at (1.5, 0.75) {};
		\node [style=none] (74) at (-1.5, 0.75) {};
		\node [style=none] (75) at (0, 0) {};
		\node [style=none] (76) at (0, 0.25) {};
		\node [style=none] (77) at (0.5, 0.75) {};
		\node [style=none] (28) at (1.5, 0) {};
		\node [style=none] (29) at (1.5, 1.5) {};
		\node [style=none] (35) at (-1.5, 0) {};
		\node [style=none] (36) at (-1.5, 1.5) {};
		\node [style=none] (38) at (0, 0) {};
		\node [style=none] (39) at (1.5, 0.75) {};
		\node [style=none] (40) at (-1.5, 0.75) {};
		\node [style=none] (46) at (0, 0) {};
		\node [style=none] (59) at (-2, -0.75) {$f$};
		\node [style=none] (61) at (0, 0.25) {};
		\node [style=none] (62) at (0.5, 0.75) {};
		\node [style=none] (78) at (-1.5, 0) {};
		\node [style=none] (79) at (-1.5, -0.75) {};
		\node [style=none] (80) at (0, 0) {};
		\node [style=none] (81) at (0, -0.25) {};
		\node [style=none] (82) at (-0.5, -0.75) {};
		\node [style=none] (83) at (-1.5, -1.5) {};
		\node [style=none] (84) at (1.5, 0) {};
		\node [style=none] (85) at (1.5, -1.5) {};
		\node [style=none] (86) at (-1.5, -0.75) {};
		\node [style=none] (87) at (1.5, -0.75) {};
		\node [style=none] (88) at (0, 0) {};
		\node [style=none] (89) at (0, -0.25) {};
		\node [style=none] (90) at (-0.5, -0.75) {};
		\node [style=none] (91) at (-1.5, 0) {};
		\node [style=none] (92) at (-1.5, -1.5) {};
		\node [style=none] (93) at (1.5, 0) {};
		\node [style=none] (94) at (1.5, -1.5) {};
		\node [style=none] (96) at (-1.5, -0.75) {};
		\node [style=none] (97) at (1.5, -0.75) {};
		\node [style=none] (99) at (0, -0.25) {};
		\node [style=none] (100) at (-0.5, -0.75) {};
	\end{pgfonlayer}
	\begin{pgfonlayer}{edgelayer}
		\draw [style=fi-bl] (67.center)
			 to (66.center)
			 to (63.center)
			 to (65.center)
			 to (68.center)
			 to [bend right=45, looseness=1.25] cycle;
		\draw [style=fi-pi] (77.center)
			 to (73.center)
			 to (69.center)
			 to (71.center)
			 to (74.center)
			 to (70.center)
			 to (75.center)
			 to (76.center)
			 to [bend left=45, looseness=1.25] cycle;
		\draw [style=vt-bl] (62.center) to (39.center);
		\draw [style=vt-bl, bend left=45, looseness=1.25] (61.center) to (62.center);
		\draw [style=fi-pi] (81.center)
			 to (80.center)
			 to (78.center)
			 to (79.center)
			 to (82.center)
			 to [bend right=45, looseness=1.25] cycle;
		\draw [style=fi-bl] (90.center)
			 to (86.center)
			 to (83.center)
			 to (85.center)
			 to (87.center)
			 to (84.center)
			 to (88.center)
			 to (89.center)
			 to [bend left=45, looseness=1.25] cycle;
		\draw [style=vt-bl-di] (100.center) to (96.center);
		\draw [style=vt-bl, bend left=45, looseness=1.25] (99.center) to (100.center);
		\draw [style=vt-bl] (61.center) to (99.center);
	\end{pgfonlayer}
\end{tikzpicture}}
  {%
  }}%
\quad=\quad%
{\tikzstyle{every picture}=[tikzfig]
  {\begin{tikzpicture}
	\begin{pgfonlayer}{nodelayer}
		\node [style=none] (28) at (-1.5, 1.5) {};
		\node [style=none] (35) at (1.5, 1.5) {};
		\node [style=none] (60) at (-1.5, 0) {};
		\node [style=none] (61) at (1.5, 0) {};
		\node [style=none] (62) at (-1.5, 0) {};
		\node [style=none] (63) at (1.5, 0) {};
		\node [style=none] (64) at (-1.5, -1.5) {};
		\node [style=none] (65) at (1.5, -1.5) {};
		\node [style=none] (66) at (2, 0) {$f$};
	\end{pgfonlayer}
	\begin{pgfonlayer}{edgelayer}
		\draw [style=fi-pi] (35.center)
			 to (61.center)
			 to (60.center)
			 to (28.center)
			 to cycle;
		\draw [style=fi-bl] (63.center)
			 to (65.center)
			 to (64.center)
			 to (62.center)
			 to cycle;
		\draw [style=vt-bl-di] (63.center) to (62.center);
	\end{pgfonlayer}
\end{tikzpicture}}
  {%
  }}%

  \hspace*{2.2em} \text{and} \hspace*{2.6em}
{\tikzstyle{every picture}=[tikzfig]
  {\begin{tikzpicture}
	\begin{pgfonlayer}{nodelayer}
		\node [style=none] (63) at (0, 1.5) {};
		\node [style=none] (65) at (0.75, 1.5) {};
		\node [style=none] (66) at (0, 0) {};
		\node [style=none] (67) at (0.25, 0) {};
		\node [style=none] (68) at (0.75, 0.5) {};
		\node [style=none] (69) at (1.5, 1.5) {};
		\node [style=none] (70) at (0, -1.5) {};
		\node [style=none] (71) at (1.5, -1.5) {};
		\node [style=none] (73) at (0.75, 1.5) {};
		\node [style=none] (74) at (0.75, -1.5) {};
		\node [style=none] (75) at (0, 0) {};
		\node [style=none] (76) at (0.25, 0) {};
		\node [style=none] (77) at (0.75, 0.5) {};
		\node [style=none] (28) at (0, 1.5) {};
		\node [style=none] (29) at (1.5, 1.5) {};
		\node [style=none] (35) at (0, -1.5) {};
		\node [style=none] (36) at (1.5, -1.5) {};
		\node [style=none] (38) at (0, 0) {};
		\node [style=none] (39) at (0.75, 1.5) {};
		\node [style=none] (40) at (0.75, -1.5) {};
		\node [style=none] (46) at (0, 0) {};
		\node [style=none] (59) at (-0.75, -2) {$\widehat{f}$};
		\node [style=none] (61) at (0.25, 0) {};
		\node [style=none] (62) at (0.75, 0.5) {};
		\node [style=none] (78) at (0, -1.5) {};
		\node [style=none] (79) at (-0.75, -1.5) {};
		\node [style=none] (80) at (0, 0) {};
		\node [style=none] (81) at (-0.25, 0) {};
		\node [style=none] (82) at (-0.75, -0.5) {};
		\node [style=none] (83) at (-1.5, -1.5) {};
		\node [style=none] (84) at (0, 1.5) {};
		\node [style=none] (85) at (-1.5, 1.5) {};
		\node [style=none] (86) at (-0.75, -1.5) {};
		\node [style=none] (87) at (-0.75, 1.5) {};
		\node [style=none] (88) at (0, 0) {};
		\node [style=none] (89) at (-0.25, 0) {};
		\node [style=none] (90) at (-0.75, -0.5) {};
		\node [style=none] (91) at (0, -1.5) {};
		\node [style=none] (92) at (-1.5, -1.5) {};
		\node [style=none] (93) at (0, 1.5) {};
		\node [style=none] (94) at (-1.5, 1.5) {};
		\node [style=none] (96) at (-0.75, -1.5) {};
		\node [style=none] (97) at (-0.75, 1.5) {};
		\node [style=none] (99) at (-0.25, 0) {};
		\node [style=none] (100) at (-0.75, -0.5) {};
	\end{pgfonlayer}
	\begin{pgfonlayer}{edgelayer}
		\draw [style=fi-pi] (67.center)
			 to (66.center)
			 to (63.center)
			 to (65.center)
			 to (68.center)
			 to [bend left=45, looseness=1.25] cycle;
		\draw [style=fi-bl] (77.center)
			 to (73.center)
			 to (69.center)
			 to (71.center)
			 to (74.center)
			 to (70.center)
			 to (75.center)
			 to (76.center)
			 to [bend right=45, looseness=1.25] cycle;
		\draw [style=vt-bl] (62.center) to (39.center);
		\draw [style=vt-bl, bend right=45, looseness=1.25] (61.center) to (62.center);
		\draw [style=fi-bl] (81.center)
			 to (80.center)
			 to (78.center)
			 to (79.center)
			 to (82.center)
			 to [bend left=45, looseness=1.25] cycle;
		\draw [style=fi-pi] (90.center)
			 to (86.center)
			 to (83.center)
			 to (85.center)
			 to (87.center)
			 to (84.center)
			 to (88.center)
			 to (89.center)
			 to [bend right=45, looseness=1.25] cycle;
		\draw [style=vt-bl-di] (100.center) to (96.center);
		\draw [style=vt-bl, bend right=45, looseness=1.25] (99.center) to (100.center);
		\draw [style=vt-bl] (61.center) to (99.center);
	\end{pgfonlayer}
\end{tikzpicture}}
  {%
  }}%
\quad=\quad%
{\tikzstyle{every picture}=[tikzfig]
  {\begin{tikzpicture}
	\begin{pgfonlayer}{nodelayer}
		\node [style=none] (28) at (-1.5, -1.5) {};
		\node [style=none] (35) at (-1.5, 1.5) {};
		\node [style=none] (59) at (0, -2) {$\widehat{f}$};
		\node [style=none] (60) at (0, -1.5) {};
		\node [style=none] (61) at (0, 1.5) {};
		\node [style=none] (62) at (0, -1.5) {};
		\node [style=none] (63) at (0, 1.5) {};
		\node [style=none] (64) at (1.5, -1.5) {};
		\node [style=none] (65) at (1.5, 1.5) {};
	\end{pgfonlayer}
	\begin{pgfonlayer}{edgelayer}
		\draw [style=fi-pi] (35.center)
			 to (61.center)
			 to (60.center)
			 to (28.center)
			 to cycle;
		\draw [style=fi-bl] (63.center)
			 to (65.center)
			 to (64.center)
			 to (62.center)
			 to cycle;
		\draw [style=vt-bl-di] (63.center) to (62.center);
	\end{pgfonlayer}
\end{tikzpicture}}
  {%
  }}%
 \hspace*{0.4em}
  \label{eq:pic:companion2+3}
  \ee
The corresponding diagrams for a conjoint pair are obtained by reflecting those in
\eqref{eq:pic:companion1} and \eqref{eq:pic:companion2+3}
horizontally and reversing the arrows. 

\begin{defn} \label{def:HA,VA}
The \emph{horizontal bicategory} $\ch(\da)$ of a double category $\da$ is the 
bicategory with the same class of objects as $\da$, with 1-morphisms the horizontal 
1-morphisms of $\da$, and with 2-morphisms the globular 2-cells of $\da$.
The \emph{vertical $2$-category} $\cv(\da)$ of $\da$ is the 2-category with the same
objects as $\da$, with 1-morphisms the vertical 1-morphisms of $\da$, and with
2-morphisms the 2-cells of $\da$ with trivial horizontal boundaries.
\end{defn}

\begin{rem}
In \cite{wesSh}, horizontal and vertical 1-morphisms of a double category $\da$
are called \emph{loose} and \emph{tight} 1-cells, respectively, and the horizontal
bicategory $\ch(\da)$ of $\da$ is called its \emph{loose bicategory}.
\end{rem}

Note that $\cv(\da)$ is indeed a $2$-category, rather than a generic bicategory,
because by definition a double category is vertically strict. Examples of horizontal
bicategories of our interest are: $\ch(\dbord)$ is the familiar bicategory $\bbord$ 
of extended bordisms, while $\ch(\dprof)$ is the bicategory $\bprof$. The vertical 
2-category $\cv(\dprof)$ is precisely the 2-category $\ccat_{\ko}$ of linear
categories, linear functors and natural transformations.
Indeed, there is an isomorphism 
  \be
  \cv(\dprof) \iso \ccat_\ko
  \label{eq:cv=ccat}
  \ee
of 2-categories that is identity on both objects and 1-morphisms; it acts on 
2-morphisms via 
  \be
  \begin{aligned}
  & \nat_{\ko}(\Hom_A(\sim,\backsim),\Hom_B(F{\sim},G{\backsim}))
  \,\cong\! \int_{a,b}\! \Hom_\ko(\Hom_A(a,b),\Hom_B(Fa,Gb))
  \\
  & \qquad \cong\! \int_{b}\int_{a}\Hom_\ko(\Hom_A(a,b),\Hom_B(Fa,Gb))
  \,\cong\! \int_{b}\Hom_B(Fb,Gb)=\nat_{\ko}(F,G) \,.
  \end{aligned}
  \ee
A double functor $F\colon \da\Rarr~\db$ gives by restriction a functor
$\ch(\da) \Rarr~ \ch(\db)$
of bicategories, which by abuse of notation we denote with the same symbol $F$.

\begin{lem}
Let $(f,\widehat{f})$ be a companion pair in the double category $\dd$, and let
$f\,{\dashv}\, g$ be an adjoint pair in its vertical 2-category
$\cv(\dd)$. Then $\widehat{f}$ is a conjoint of $g$.
\end{lem}

\begin{proof}
The counit and unit of the conjoint pair are given by 
  \be
  \begin{tikzcd}
  a & a & b && b & b & b \\
  & b & b & {\text{and}} & a & a \\
  a & a & a && a & b & b
  \arrow[""{name=0, anchor=center, inner sep=0}, "{\widehat{f}~}", "\shortmid"{marking}, from=1-2, to=1-3]
  \arrow[Rightarrow, no head, from=1-3, to=2-3]
  \arrow["f"', from=1-2, to=2-2]
  \arrow[""{name=1, anchor=center, inner sep=0}, "\shortmid"{marking},Rightarrow, no head, from=2-2, to=2-3]
  \arrow["g"', from=2-2, to=3-2]
  \arrow["\shortmid"{marking}, Rightarrow, no head, from=3-2, to=3-3]
  \arrow["g", from=2-3, to=3-3]
  \arrow[""{name=2, anchor=center, inner sep=0}, "\shortmid"{marking},Rightarrow, no head, from=1-1, to=1-2]
  \arrow[Rightarrow, no head, from=1-1, to=3-1]
  \arrow[""{name=3, anchor=center, inner sep=0}, "\shortmid"{marking},Rightarrow, no head, from=3-1, to=3-2]
  \arrow["g"', from=1-5, to=2-5]
  \arrow[""{name=4, anchor=center, inner sep=0}, "\shortmid"{marking},Rightarrow, no head, from=2-5, to=2-6]
  \arrow["\shortmid"{marking}, Rightarrow, no head, from=1-5, to=1-6]
  \arrow["g", from=1-6, to=2-6]
  \arrow[Rightarrow, no head, from=2-5, to=3-5]
  \arrow[""{name=5, anchor=center, inner sep=0}, "{~\widehat{f}}"', "\shortmid"{marking}, from=3-5, to=3-6]
  \arrow["f", from=2-6, to=3-6]
  \arrow[""{name=6, anchor=center, inner sep=0}, "\shortmid"{marking},Rightarrow, no head, from=1-6, to=1-7]
  \arrow[Rightarrow, no head, from=1-7, to=3-7]
  \arrow[""{name=7, anchor=center, inner sep=0}, "\shortmid"{marking},Rightarrow, no head, from=3-6, to=3-7]
  \arrow["\displaystyle\varepsilon"{description}, draw=none, from=0, to=1]
  \arrow["\displaystyle\eta"{description}, draw=none, from=4, to=5]
  \arrow["\displaystyle{\widetilde{\eta}}"{description}, draw=none, from=2, to=3]
  \arrow["\displaystyle{\widetilde{\varepsilon}}"{description}, draw=none, from=6, to=7]
  \end{tikzcd}
  \label{eq:conjoint(co)unit}
  \ee
where $\varepsilon$ and $\eta$ are the counit and unit of $(f,\widehat{f})$ and
$\widetilde\varepsilon$ and $\widetilde\eta$ the counit and unit for the adjunction
$f\,{\dashv}\, g$ in $\cv(\dd)$, and where again the unlabeled 2-morphisms are unit 
2-morphisms.  The corresponding string diagrams look as follows:
  \be
  \scalebox{1.1}{%
{\tikzstyle{every picture}=[tikzfig]
  {\begin{tikzpicture}
	\begin{pgfonlayer}{nodelayer}
		\node [style=none] (86) at (1.5, 1.5) {};
		\node [style=none] (88) at (0, 0.5) {};
		\node [style=none] (89) at (0.5, 1) {};
		\node [style=none] (90) at (0.5, 1.5) {};
		\node [style=none] (91) at (-0.5, 0.5) {};
		\node [style=none] (92) at (-0.5, 0) {};
		\node [style=none] (93) at (1.5, 0) {};
		\node [style=none] (94) at (-1, 0.25) {};
		\node [style=none] (95) at (1.5, -1.5) {};
		\node [style=none] (97) at (-1.5, -1.5) {};
		\node [style=none] (98) at (-1.5, 1.5) {};
		\node [style=none] (99) at (0, 0.5) {};
		\node [style=none] (100) at (0.5, 1) {};
		\node [style=none] (101) at (0.5, 1.5) {};
		\node [style=none] (102) at (-0.5, 0.5) {};
		\node [style=none] (103) at (-0.5, 0) {};
		\node [style=none] (104) at (1.5, 0) {};
		\node [style=none] (105) at (-1, 0.25) {};
		\node [style=none] (78) at (0, 0.5) {};
		\node [style=none] (79) at (0.5, 1) {};
		\node [style=none] (80) at (0.5, 1.5) {};
		\node [style=none] (81) at (-0.5, 0.5) {};
		\node [style=none] (82) at (-0.5, 0) {};
		\node [style=none] (83) at (1.5, 0) {};
		\node [style=none] (84) at (-1, 0.25) {};
		\node [style=bdot] (85) at (-1, 0.25) {};
		\node [style=none] (106) at (0.5, 2) {$\widehat{f}$};
		\node [style=none] (107) at (2, 0) {$g$};
	\end{pgfonlayer}
	\begin{pgfonlayer}{edgelayer}
		\draw [style=fi-bl] (93.center)
			 to (86.center)
			 to (90.center)
			 to (89.center)
			 to [bend left=45, looseness=1.25] (88.center)
			 to (91.center)
			 to [in=90, out=-180, looseness=1.25] (94.center)
			 to [in=-180, out=-90, looseness=1.25] (92.center)
			 to cycle;
		\draw [style=fi-pi] (95.center)
			 to (104.center)
			 to (103.center)
			 to [in=-90, out=-180, looseness=1.25] (105.center)
			 to [in=-180, out=90, looseness=1.25] (102.center)
			 to (99.center)
			 to [bend right=45, looseness=1.25] (100.center)
			 to (101.center)
			 to (98.center)
			 to (97.center)
			 to cycle;
		\draw [style=vt-bl, bend right=45, looseness=1.25] (78.center) to (79.center);
		\draw [style=vt-bl] (80.center) to (79.center);
		\draw [style=vt-bl-di] (82.center) to (83.center);
		\draw [style=vt-bl] (81.center) to (78.center);
		\draw [style=vt-bl, in=90, out=-180, looseness=1.25] (81.center) to (84.center);
		\draw [style=vt-bl, in=-90, out=-180, looseness=1.25] (82.center) to (84.center);
	\end{pgfonlayer}
\end{tikzpicture}}
  {%
  }}%
}\qquad\text{and}\qquad\scalebox{1.1}{%
{\tikzstyle{every picture}=[tikzfig]
  {\begin{tikzpicture}
	\begin{pgfonlayer}{nodelayer}
		\node [style=none] (86) at (-1.5, -1.5) {};
		\node [style=none] (88) at (0, -0.5) {};
		\node [style=none] (89) at (-0.5, -1) {};
		\node [style=none] (90) at (-0.5, -1.5) {};
		\node [style=none] (91) at (0.5, -0.5) {};
		\node [style=none] (92) at (0.5, 0) {};
		\node [style=none] (93) at (-1.5, 0) {};
		\node [style=none] (94) at (1, -0.25) {};
		\node [style=none] (95) at (-1.5, 1.5) {};
		\node [style=none] (97) at (1.5, 1.5) {};
		\node [style=none] (98) at (1.5, -1.5) {};
		\node [style=none] (99) at (0, -0.5) {};
		\node [style=none] (100) at (-0.5, -1) {};
		\node [style=none] (101) at (-0.5, -1.5) {};
		\node [style=none] (102) at (0.5, -0.5) {};
		\node [style=none] (103) at (0.5, 0) {};
		\node [style=none] (104) at (-1.5, 0) {};
		\node [style=none] (105) at (1, -0.25) {};
		\node [style=none] (78) at (0, -0.5) {};
		\node [style=none] (79) at (-0.5, -1) {};
		\node [style=none] (80) at (-0.5, -1.5) {};
		\node [style=none] (81) at (0.5, -0.5) {};
		\node [style=none] (82) at (0.5, 0) {};
		\node [style=none] (83) at (-1.5, 0) {};
		\node [style=none] (84) at (1, -0.25) {};
		\node [style=bdot] (85) at (1, -0.25) {};
		\node [style=none] (106) at (-0.5, -2) {$\widehat{f}$};
		\node [style=none] (107) at (-2, 0) {$g$};
	\end{pgfonlayer}
	\begin{pgfonlayer}{edgelayer}
		\draw [style=fi-pi] (93.center)
			 to (86.center)
			 to (90.center)
			 to (89.center)
			 to [bend left=45, looseness=1.25] (88.center)
			 to (91.center)
			 to [in=-90, out=0, looseness=1.25] (94.center)
			 to [in=0, out=90, looseness=1.25] (92.center)
			 to cycle;
		\draw [style=fi-bl] (95.center)
			 to (104.center)
			 to (103.center)
			 to [in=90, out=0, looseness=1.25] (105.center)
			 to [in=0, out=-90, looseness=1.25] (102.center)
			 to (99.center)
			 to [bend right=45, looseness=1.25] (100.center)
			 to (101.center)
			 to (98.center)
			 to (97.center)
			 to cycle;
		\draw [style=vt-bl, bend right=45, looseness=1.25] (78.center) to (79.center);
		\draw [style=vt-bl-di] (79.center) to (80.center);
		\draw [style=vt-bl] (82.center) to (83.center);
		\draw [style=vt-bl] (81.center) to (78.center);
		\draw [style=vt-bl, in=-90, out=0, looseness=1.25] (81.center) to (84.center);
		\draw [style=vt-bl, in=90, out=0, looseness=1.25] (82.center) to (84.center);
	\end{pgfonlayer}
\end{tikzpicture}}
  {%
  }}%
}
  \label{eq:pic:conjoint(co)unit}
  \ee
It is an entertaining exercise to use the graphical calculus to check that these
satisfy the yanking equations.
\end{proof}

%%%%%%%%%%%%%%%%%%%%%%%%%%%%%%%%%%%%%%%%%%%%%%%%%%%%%%%%%%%%%%%%%%%%%%%%

\section{A double categorical perspective on string nets} \label{sec:dbl-2}

We are now ready to apply the notions developed in the previous section to
the description of CFT correlators. In particular, we are going to obtain a
conceptual and unified understanding of field functors and universal correlators.

\subsection{String-net modular functors as symmetric monoidal double functors}
\label{sec:SN=sblefun}

Recall that for a pointed pivotal bicategory $(\cb,*_{\cb})$, such as
$(\cfrc,\tu)$ and $(\cbc,*)$ for a pivotal tensor category $\cc$, there is an
open-closed modular functor $\SNb\colon\bbord\Rarr~\bprof$ \eqref{eq:def:SNb}.
$\SNb$ is by definition a symmetric monoidal pseudofunctor. Moreover, the assignment
of cylinder categories $\snb(\ell)$ to one-manifolds (as described after 
\eqref{eq:SNb=(snb)}) is functorial under embeddings. As a consequence, the
pseudofunctor $\SNb$ extends to a double functor between the double categories 
$\dbord$ and $\dprof$. Its action on mapping class group elements
is extended to an action on isotopy classes of embeddings. In fact we have:

\begin{thm} \label{thm:dfunsn}
Let $(\cb,*_{\cb})$ be a pointed pivotal bicategory.
There is a symmetric monoidal double functor 
  \be
  \dSNb\Colon\dbord\rarr~\dprof
  \ee
that extends the open-closed modular functor \eqref{eq:def:SNb}.
\end{thm}

\begin{proof}
The extension is performed by exploiting the functoriality of string-net spaces
with respect to oriented embeddings. The symmetric monoidal structure 
is induced by the canonical isomorphisms $\SNb(\varSigma \,{\sqcup}\, \varSigma')
\,{\simeq}\, \SNb(\varSigma) \oti \SNb(\varSigma')$.
\end{proof}

%%%%%%%%%%%%%%%%%%%%%%%%%%%%%%%%%%%%%%%%%%%%%%%%%%%%%%%%%%%%%%%%%%%%%%%%

\subsection{Field maps and universal correlators as monoidal vertical transformations}
\label{sec:field+Ucor=verttrans}

Next we observe that the 2-morphisms \eqref{eq:ucorc=2morph} fit together to form a
vertical transformation. Indeed, the notion of vertical transformation matches
precisely the concepts of field maps and universal correlators that we introduced
in Section \ref{sec:fresh}: 

\begin{lem} \label{lem:ducorc=verttrans}
The following prescriptions specify a vertical transformation
  \be
  \ducorc\Colon \dSNfrc \xRightarrow{\phantom{xx}} \dSNc 
  \label{eq:dcftc}
  \ee
between string-net double functors.
First, the family \eqref{eq:family-theta} of vertical 1-morphisms is defined
to be given by the field maps \eqref{eq:def(ell)} according to
  \be
  \{\df_{\!\ell}^{}\colon\SNfrc(\ell)\Rarr~\SNc(\ell)\}_{\ell\in\dbord}^{} \,.
  \ee
Second, the family \eqref{eq:family-2mora} of 2-morphisms is defined to 
consist of the 2-morphisms 
  \be
  \begin{tikzcd}[column sep=4.4em, row sep=2.2em]
  \SNfrc(\ell) & \SNfrc(\ell') \\ \SN_{\cbc}(\ell) & \SN_{\cbc}(\ell')
  \arrow[""{name=0, anchor=center, inner sep=0}, "\SNfrc(\surf)~", shorten <=-3pt, "\shortmid"{marking}, from=1-1, to=1-2]
  \arrow[""{name=1, anchor=center, inner sep=0}, "~\SN_{\cbc}(\surf)~"', shorten <=-3pt, "\shortmid"{marking}, from=2-1, to=2-2]
  \arrow["\df_{\!\ell}^{}"', from=1-1, to=2-1, xshift=-8pt]
  \arrow["\df_{\!\ell'}^{}", from=1-2, to=2-2, xshift=-10pt]
  \arrow["\displaystyle\ucorc(\surf)"{description}, draw=none, from=0, to=1] 
  \end{tikzcd} 
  \label{eq:ucorc=2morph-0}
  \ee
which are given by the universal correlators \eqref{eq:ucorc(surf)}.
That is, the components of the natural transformation
  \be
  \ucorc(\surf) \Colon\SNfrc(\surf;-,\sim)
  \xRightarrow{\phantom{xx}} \SNc(\surf,\df_{\!\ell}^{}-,\df_{\!\ell'}^{}{\sim})
  \ee
are the universal correlators.
\end{lem}

\begin{proof}
That this prescription indeed furnishes a vertical transformation from
$\dSNfrc$ to $\dSNc$ is seen as
follows. The horizontal functoriality translates to the compatibility of $\ucorc$ 
with sewing, while the vertical naturality corresponds to the compatibility with
embeddings, which implies invariance under the actions of the mapping class
groups of \worldsheet s when accounting for the graphical calculus of defects.
Finally, the horizontal unity follows from the fact that the actions of field 
functors on morphisms are given by universal correlators on cylinders. 
\end{proof}

Moreover,
the vertical transformation $\ducorc$ \eqref{eq:dcftc} is monoidal in the sense of
Definition \ref{def:monoidalverttrafo}. To verify this is straightforward, using the
fact that both the field functors and the string-net modular functors are monoidal. 
Accordingly we can upgrade Lemma \ref{lem:ducorc=verttrans} to:

\begin{thm} \label{thm:vtransUCor}
Let $\cc$ be a \ko-linear pivotal tensor category. There is a canonical monoidal 
vertical transformation 
  \be
  \ducorc \Colon \dSNfrc \xRightarrow{\phantom{xx}} \dSNc \,.
  \label{eq:def:ducorc}
  \ee
The components of $\ducorc$ at objects are given by the field functors 
$\{\df_{\!\ell}\}_{\ell\in\dbord}$, and its components at horizontal 1-morphisms 
are given by the universal correlators $\ucorc(\surf)$.
\end{thm}

Theorem \ref{thm:vtransUCor} is in fact a consequence of the functoriality
of string-net constructions with respect to rigid separable Frobenius functors
(as defined in \Cite{Def.\,2.28}{fusY2}):

\begin{thm} \label{thm:vtransRSF}
Let $\cb$ and $\cb'$ be pointed pivotal bicategories and $F\colon\cb\Rarr~\cb'$ be a
rigid separable Frobenius functor. There is a canonical vertical transformation 
  \be
  F_{*}\Colon \dSNb \xRightarrow{\phantom{xx}} \mathbb{S}\mathrm{N}_{\cb'} \,.
  \label{eq:dsnf}
  \ee
\end{thm}

\begin{proof}
The corresponding natural transformation between bicategorical string-net
constructions is established in Theorem 4.10 of \cite{fusY2}. It extends to a
vertical transformation between double functors in the same way as obtained above
for the special case that $\cb \eq \dSNfrc$ and $\cb' \eq \dSNc$.
\end{proof}

%%%%%%%%%%%%%%%%%%%%%%%%%%%%%%%%%%%%%%%%%%%%%%%%%%%%%%%%%%%%%%%%%%%%%%%%

\subsection{Universal correlators as a monoidal equivalence between modular functors}

Under suitable conditions involving the existence of
companions and conjoints one can associate symmetric monoidal bicategories to framed
bicategories, symmetric monoidal pseudofunctors to symmetric monoidal double functors
between framed bicategories, and a monoidal pseudonatural equivalence to a monoidal
vertical equivalence between the latter.  Our exposition is based on \cite{wesSh};
the guiding idea is that one can make use of companions to transport structure
morphisms on a double category to structure morphisms on its horizontal bicategory.
The additional conditions are needed because the horizontal bicategory only knows 
about globular morphisms and we have to keep information about non-trivial vertical 
1-morphisms.

A vertical transformation $\alpha\colon F\,{\xRightarrow{\phantom{...}}}\, G$ between
(strong) double functors furnishes \Cite{Thm.\,4.6}{wesSh} an oplax natural
transformation $\widehat{\alpha}\colon F\,{\xRightarrow{\phantom{...}}}\, G$ between 
the underlying pseudofunctors (for the horizontal bicategories), provided
that every vertical component $\alpha_{a}\colon Fa\Rarr~ Ga$ has a companion,
which then becomes the 1-cell component $\widehat{\alpha}_{a}\colon Fa\ptO Ga$
at the same object. The 2-cell components 
  \be
  \begin{tikzcd}
  Fa & Fb & Fa & Fa & Fb & Ga \\
  Ga & Gb & Fa & Ga & Gb & Ga
  \arrow["Ff", "\shortmid"{marking}, from=1-1, to=1-2]
  \arrow[""{name=0, anchor=center, inner sep=0}, "{\widehat{\alpha}_b}", "\shortmid"{marking}, from=1-2, to=2-2]
  \arrow["{\widehat{\alpha}_a}"', "\shortmid"{marking}, from=1-1, to=2-1]
  \arrow["Gg~"', "\shortmid"{marking}, from=2-1, to=2-2]
  \arrow["{\widehat{\alpha}_f\!\!\!}"', shift left=1, shorten <=4pt, shorten >=4pt, Rightarrow, from=1-2, to=2-1]
  \arrow[""{name=1, anchor=center, inner sep=0}, Rightarrow, no head, from=1-3, to=2-3]
  \arrow[""{name=2, anchor=center, inner sep=0}, "\shortmid"{marking},Rightarrow, no head, from=1-3, to=1-4]
  \arrow["{\alpha_a}"', from=1-4, to=2-4]
  \arrow[""{name=3, anchor=center, inner sep=0}, "Ff", "\shortmid"{marking}, from=1-4, to=1-5] 	\arrow["{\alpha_b}", from=1-5, to=2-5]
  \arrow[""{name=4, anchor=center, inner sep=0}, "{\widehat{\alpha}_b}", "\shortmid"{marking}, from=1-5, to=1-6]
  \arrow[Rightarrow, no head, from=1-6, to=2-6]
  \arrow[""{name=5, anchor=center, inner sep=0}, "Gf~"', "\shortmid"{marking}, from=2-4, to=2-5]
  \arrow[""{name=6, anchor=center, inner sep=0}, "\shortmid"{marking},Rightarrow, no head, from=2-5, to=2-6]
  \arrow[""{name=7, anchor=center, inner sep=0}, "{\widehat{\alpha}_a~}"', "\shortmid"{marking}, from=2-3, to=2-4]
  \arrow["\displaystyle\eta~"{description}, draw=none, from=2, to=7]
  \arrow["~\displaystyle\varepsilon"{description}, draw=none, from=4, to=6]
  \arrow["\displaystyle ~~{=}~"{description}, draw=none, from=0, to=1]
  \arrow["{\displaystyle\alpha_f}"{description}, draw=none, from=3, to=5]
  \end{tikzcd}
  \label{eq:fold2globular}
  \ee
are obtained from `folding' $\alpha_{f}$ into a globular 2-morphism.

In the present subsection we show that, as a consequence of further results from
\cite{wesSh}, the assignment $\ducorc \,{\xmapsto{\phantom{xx}}}\, \widehat{\ducorc}$
of the bicategorical natural transformation $\widehat{\ducorc}$ to the
vertical transformation $\ducorc$ is compatible with the symmetric monoidal
structures. We show further that the resulting monoidal pseudonatural 
transformation is an equivalence of bicategorical modular functors.

There is a somewhat restricted notion of invertibility for 2-morphisms in a double
category $\dd$: a 2-cell $\alpha$ is (strictly) invertible if it is an isomorphism 
in the category $\dd_1$ of horizontal 1-morphisms and 2-cells. The invertibility of 
the 2-cell components of the universal correlator $\ducorc\colon\dSNfrc\RarR~\dSN$  
is instead to be understood in the sense of following definition, which uses the
observation that the vertical 2-category $\cv(\da)$ of a double category $\da$ is
indeed a (strict) 2-category. Up to transposition, this coincides with the notion of
invertibility introduced in \Cite{Sect.\,2}{grPar} and in \Cite{Sect.\,2}{mosV}).

\begin{defn} \label{def:weakly-inf}
Let $\begin{tikzcd}[column sep=0.8em, row sep=0.6em]
  a & b \\ c & d
  \arrow[""{name=0, anchor=center, inner sep=0}, "P\,", shorten <=-3pt, shorten >=-3pt, "\shortmid"{marking}, from=1-1, to=1-2]
  \arrow[""{name=1, anchor=center, inner sep=0}, "~~Q"', shorten <=-3pt, shorten >=-3pt, "\shortmid"{marking}, from=2-1, to=2-2]
  \arrow["f"', shorten <=-2pt, shorten >=-2pt, from=1-1, to=2-1]
  \arrow["g",  shorten <=-2pt, shorten >=-2pt, from=1-2, to=2-2]
  \arrow["\displaystyle\alpha"{description}, draw=none, from=0, to=1]
  \end{tikzcd}$ 
be a 2-morphism in a double category $\dd$.
 \\
$\alpha$ is said to be \emph{weakly invertible} iff the vertical 1-morphisms 
$f$ and $g$ are both adjoint 
equivalences in the vertical 2-category $\cv(\dd)$ and there exists a 2-morphism
  \be
  \begin{tikzcd}
  c & d \\ a & b
  \arrow["{f^{-1}}"', from=1-1, to=2-1]
  \arrow[""{name=0, anchor=center, inner sep=0}, "Q~", "\shortmid"{marking}, from=1-1, to=1-2]
  \arrow["{g^{-1}}", from=1-2, to=2-2]
  \arrow[""{name=1, anchor=center, inner sep=0}, "~~P"', "\shortmid"{marking}, from=2-1, to=2-2]
  \arrow["{\displaystyle\alpha^{-1}}"{description}, draw=none, from=0, to=1] 
  \end{tikzcd} 
  \ee
where $f^{-1}$ and $g^{-1}$ are any choice of quasi-inverses of $f$ and $g$ 
(rather than inverses on the nose) in $\cv(\dd)$, such that the equality
  \be
  \begin{tikzcd}
  a & a & b & b & a & b \\ & c & d \\ a & a & b & b & a & b
  \arrow["f^{-1}\!",swap,from=2-2, to=3-2]
  \arrow[""{name=0, anchor=center, inner sep=0}, "\shortmid"{marking}, from=2-2, to=2-3]
  \arrow["g^{-1}",from=2-3, to=3-3]
  \arrow[""{name=1, anchor=center, inner sep=0}, "~~P"', "\shortmid"{marking}, from=3-2, to=3-3]
  \arrow["f",swap,from=1-2, to=2-2]
  \arrow["g",from=1-3, to=2-3]
  \arrow[""{name=2, anchor=center, inner sep=0}, "P~", "\shortmid"{marking}, from=1-2, to=1-3]
  \arrow[""{name=3, anchor=center, inner sep=0}, "\shortmid"{marking}, Rightarrow, no head, from=1-1, to=1-2]
  \arrow[Rightarrow, no head, from=1-1, to=3-1]
  \arrow[""{name=4, anchor=center, inner sep=0}, "\shortmid"{marking}, Rightarrow, no head, from=3-1, to=3-2]
  \arrow[""{name=5, anchor=center, inner sep=0}, "\shortmid"{marking}, Rightarrow, no head, from=1-3, to=1-4]
  \arrow[""{name=6, anchor=center, inner sep=0}, Rightarrow, no head, from=1-4, to=3-4]
  \arrow[""{name=7, anchor=center, inner sep=0}, "\shortmid"{marking}, Rightarrow, no head, from=3-3, to=3-4]
  \arrow[""{name=8, anchor=center, inner sep=0}, Rightarrow, no head, from=1-5, to=3-5]
  \arrow["P~", "\shortmid"{marking}, from=1-5, to=1-6]
  \arrow[Rightarrow, no head, from=1-6, to=3-6]
  \arrow["~~P"', "\shortmid"{marking}, from=3-5, to=3-6]
  \arrow["{\displaystyle\varepsilon^{-1}_f}"{description}, draw=none, from=3, to=4]
  \arrow["\displaystyle\varepsilon^{}_g"{description}, draw=none, from=5, to=7]
  \arrow["\displaystyle ~=~"{description}, draw=none, from=6, to=8]
  \arrow["{\displaystyle\alpha^{-1}}"{description}, draw=none, from=0, to=1]
  \arrow["\displaystyle\alpha"{description}, draw=none, from=2, to=0]
  \end{tikzcd} 
  \label{eq:weakinv}
  \ee
holds, where $\varepsilon_f^{-1}$ is any choice of unit for the adjunction 
$f \Dashv f^{-1}$ in the 2-category $\cv(\dd)$ and $\varepsilon_g$ is any choice of
counit for the adjunction $g^{-1}\Dashv g$ in $\cv(\dd)$.
\end{defn}

 \bigskip

The string diagram description of the equality \eqref{eq:weakinv} looks as follows:
    \be
    \scalebox{0.9}{%
{\tikzstyle{every picture}=[tikzfig]
  {\begin{tikzpicture}
	\begin{pgfonlayer}{nodelayer}
		\node [style=none] (309) at (0, 0.75) {};
		\node [style=none] (310) at (0, -0.75) {};
		\node [style=none] (311) at (0, 3) {};
		\node [style=none] (312) at (0, -3) {};
		\node [style=none] (313) at (0.5, -0.75) {};
		\node [style=none] (314) at (1.75, 0) {};
		\node [style=none] (315) at (0.5, 0.75) {};
		\node [style=none] (316) at (3, -3) {};
		\node [style=none] (317) at (3, 3) {};
		\node [style=none] (261) at (0, 0.75) {};
		\node [style=none] (262) at (0, -0.75) {};
		\node [style=none] (263) at (0, 3) {};
		\node [style=none] (264) at (0, -3) {};
		\node [style=none] (265) at (-0.5, -0.75) {};
		\node [style=none] (266) at (-1.75, 0) {};
		\node [style=none] (267) at (-0.5, 0.75) {};
		\node [style=none] (268) at (-3, -3) {};
		\node [style=none] (269) at (-3, 3) {};
		\node [style=none] (299) at (0, 0.75) {};
		\node [style=none] (300) at (0, -0.75) {};
		\node [style=none] (301) at (-0.5, -0.75) {};
		\node [style=none] (302) at (-1.75, 0) {};
		\node [style=none] (303) at (-0.5, 0.75) {};
		\node [style=none] (304) at (0, 0.75) {};
		\node [style=none] (305) at (0, -0.75) {};
		\node [style=none] (306) at (0.5, -0.75) {};
		\node [style=none] (307) at (1.75, 0) {};
		\node [style=none] (308) at (0.5, 0.75) {};
		\node [style=none] (101) at (0, 0.75) {};
		\node [style=none] (102) at (0, -0.75) {};
		\node [style=none] (111) at (0, 3) {};
		\node [style=none] (112) at (0, -3) {};
		\node [style=none] (157) at (0, 3.5) {\scriptsize$P$};
		\node [style=none] (158) at (0, -3.5) {\scriptsize$P$};
		\node [style=none] (174) at (-0.5, -0.75) {};
		\node [style=none] (176) at (-1.75, 0) {};
		\node [style=none] (177) at (-0.5, 0.75) {};
		\node [style=none] (247) at (0.5, -0.75) {};
		\node [style=none] (248) at (1.75, 0) {};
		\node [style=none] (249) at (0.5, 0.75) {};
		\node [style=none] (318) at (-1.45, 0.95) {\scriptsize$f$};
		\node [style=none] (319) at (1.7, 0.8) {\scriptsize$g$};
		\node [style=none] (320) at (-1.5, -1.1) {\scriptsize$f^{-1}$};
		\node [style=none] (321) at (1.95, -0.8) {\scriptsize$g^{-1}$};
		\node [style=dot] (322) at (0, 0.75) {};
		\node [style=bdot] (323) at (0, -0.75) {};
		\node [style=none] (324) at (0.35, 1.06) {\scriptsize$\alpha$};
		\node [style=none] (325) at (0.61, -1.16) {\scriptsize$\alpha^{\!-1}$};
	\end{pgfonlayer}
	\begin{pgfonlayer}{edgelayer}
		\draw [style=fi-bl] (314.center)
			 to [in=0, out=-90] (313.center)
			 to (310.center)
			 to (312.center)
			 to (316.center)
			 to (317.center)
			 to (311.center)
			 to (309.center)
			 to (315.center)
			 to [in=90, out=0] cycle;
		\draw [style=fi-pi] (266.center)
			 to [in=180, out=-90] (265.center)
			 to (262.center)
			 to (264.center)
			 to (268.center)
			 to (269.center)
			 to (263.center)
			 to (261.center)
			 to (267.center)
			 to [in=90, out=180] cycle;
		\draw [style=fi-gr] (302.center)
			 to [in=180, out=-90] (301.center)
			 to (300.center)
			 to (299.center)
			 to (303.center)
			 to [in=90, out=180] cycle;
		\draw [style=fi-pu] (307.center)
			 to [in=0, out=-90] (306.center)
			 to (305.center)
			 to (304.center)
			 to (308.center)
			 to [in=90, out=0] cycle;
		\draw [style=vt-bl] (111.center) to (101.center);
		\draw [style=vt-bl] (101.center) to (102.center);
		\draw [style=vt-bl] (102.center) to (112.center);
		\draw [style=vt-bl] (101.center)
			 to (177.center)
			 to [in=90, out=180] (176.center)
			 to [in=180, out=-90] (174.center)
			 to (102.center);
		\draw [style=vt-bl] (102.center)
			 to (247.center)
			 to [in=-90, out=0] (248.center)
			 to [in=0, out=90] (249.center)
			 to (101.center);
	\end{pgfonlayer}
\end{tikzpicture}}
  {%
  }}%
} \quad = \quad
    \scalebox{0.9}{%
{\tikzstyle{every picture}=[tikzfig]
  {\begin{tikzpicture}
	\begin{pgfonlayer}{nodelayer}
		\node [style=none] (310) at (0, -0.75) {};
		\node [style=none] (311) at (0, 3) {};
		\node [style=none] (312) at (0, -3) {};
		\node [style=none] (316) at (3, -3) {};
		\node [style=none] (317) at (3, 3) {};
		\node [style=none] (263) at (0, 3) {};
		\node [style=none] (264) at (0, -3) {};
		\node [style=none] (268) at (-3, -3) {};
		\node [style=none] (269) at (-3, 3) {};
		\node [style=none] (101) at (0, 0.75) {};
		\node [style=none] (102) at (0, -0.75) {};
		\node [style=none] (111) at (0, 3) {};
		\node [style=none] (157) at (0, 3.5) {\scriptsize$P$};
		\node [style=none] (158) at (0, -3.5) {\scriptsize$P$};
	\end{pgfonlayer}
	\begin{pgfonlayer}{edgelayer}
		\draw [style=fi-bl] (310.center)
			 to (312.center)
			 to (316.center)
			 to (317.center)
			 to (311.center)
			 to cycle;
		\draw [style=fi-pi] (264.center)
			 to (268.center)
			 to (269.center)
			 to (263.center)
			 to cycle;
		\draw [style=vt-bl] (111.center) to (101.center);
		\draw [style=vt-bl] (101.center) to (102.center);
		\draw [style=vt-bl] (102.center) to (112.center);
	\end{pgfonlayer}
\end{tikzpicture}}
  {%
  }}%
} 
    \label{eq:pic:verticalequiv}
    \ee

The notion of weak invertibility allows us to formulate the following statement, 
which is a direct consequence of Theorem \ref{thm:Feq} and Theorem \ref{thm:UcorIso},

\begin{cor} \label{cor:UCor-weaklyinv}
Let $\cc$ be a pivotal \ko-linear tensor category. For every open-closed
bordism $\surf\colon\ell\ptO\ell'$, the natural transformation 
  \be
  \ucorc(\surf)\Colon\SNfrc(\surf;\sim,\backsim)
  \Longrightarrow\SNc(\surf;\df_{\!\ell}^{}{\sim},\df_{\!\ell'}^{}{\backsim})
  \ee
is a weakly invertible 2-morphism 
  \be
  \begin{tikzcd}[column sep=4.4em, row sep=2.2em]
  \SNfrc(\ell) & \SNfrc(\ell') \\ \SN_{\cbc}(\ell) & \SN_{\cbc}(\ell')
  \arrow[""{name=0, anchor=center, inner sep=0}, "\SNfrc(\surf)~", shorten <=-3pt, "\shortmid"{marking}, from=1-1, to=1-2]
  \arrow[""{name=1, anchor=center, inner sep=0}, "~\SN_{\cbc}(\surf)~"', shorten <=-3pt, "\shortmid"{marking}, from=2-1, to=2-2]
  \arrow["\df_{\!\ell}^{}"', from=1-1, to=2-1, xshift=-8pt]
  \arrow["\df_{\!\ell'}^{}", from=1-2, to=2-2, xshift=-10pt]
  \arrow["\displaystyle\ucorc(\surf)"{description}, draw=none, from=0, to=1] 
  \end{tikzcd} 
  \label{eq:ucorc=2morph}
  \ee
in the double category $\dprof$.
\end{cor}

Now recall the equivalence \eqref{eq:cv=ccat} of 2-categories 
between $\cv(\dprof)$ and $\ccat_\ko$. 
This isomorphism guarantees that the folding from a double category with companions 
to its horizontal bicategory is compatible with natural tranformations.
 
Given two double categories $\da$ and $\db$, we denote by $\dbl(\da,\db)$ the 
2-category of double {(pseu\-do-)}functors, vertical transformations
and modifications.\,%
 \footnote{~We refer to \S1.6 of \cite{grPar} for the general notion of modifications,
 which are 2-cells in a \emph{double} category of double functors that has
 $\dbl(\da,\db)$ as its vertical 2-category.}
The following is evident:

\begin{lem} \label{lem:equiv-cs-weaklyinv}
Let $F,G\colon\da \,{\rightrightarrows}\, \db$ be double functors. A vertical
transformation $\alpha\colon F \,{\Rightarrow}\, G$ is an equivalence in 
$\dbl(\da,\db)$ if and only if each of its component 2-cells is weakly invertible
in the sense of Definition \ref{def:weakly-inf}.
\end{lem}

We call such a vertical transformation a \emph{vertical equivalence}. By Corollary 
\ref{cor:UCor-weaklyinv} the vertical transformation \eqref{eq:dcftc} given by field
functors $\df_{\!\ell}^{}$ and the universal correlators $\ucorc(\surf)$ is a
vertical equivalence. Note that the proof of Corollary \ref{cor:UCor-weaklyinv}
entails that while $\df_{\!\ell}^{}$ is an equivalence of categories,
it is not an isomorphism.

\begin{defn} \label{def:horstrongcomp}
A vertical transformation $\alpha$ is said to have \emph{horizontally strong
companions} iff every vertical component $\alpha_{a}$ has a companion
and the associated oplax natural transformation $\widehat{\alpha}$ 
given by \eqref{eq:fold2globular} is pseudonatural.
\end{defn}

\begin{prop} \label{prop:aboutstrongcomps}
Let $\alpha\colon F \,{\Rightarrow}\, G$ be a vertical equivalence between
double functors whose vertical components all have companions. Then
$\alpha$ has horizontally strong companions. Moreover, the associated
pseudonatural transformation $\widehat{\alpha}$ is a pseudonatural equivalence.
\end{prop}

\begin{proof}
The last statement follows from the definition of weak invertibility. 
 \\
The inverse of a 2-cell component 
  \be
  \begin{tikzcd}
  Fa & Fb \\ Ga & Gb
  \arrow["Ff", "\shortmid"{marking}, from=1-1, to=1-2]
  \arrow["{\widehat{\alpha}_b}", "\shortmid"{marking}, from=1-2, to=2-2] 
  \arrow["{\widehat{\alpha}_a}"', "\shortmid"{marking}, from=1-1, to=2-1]
  \arrow["Gg"', "\shortmid"{marking}, from=2-1, to=2-2]
  \arrow["{\displaystyle\widehat{\alpha}_f\!\!}"',
        shift left=1, shorten <=4pt, shorten >=4pt, Rightarrow, from=1-2, to=2-1] 
  \end{tikzcd} 
  \ee
(see \eqref{eq:fold2globular}) of $\widehat\alpha$ is given by
  \be
  \begin{tikzcd} 
  Fa & Ga & Gb & Gb \\ Fa & Fa & Fb & Gb
  \arrow[Rightarrow, no head, from=1-4, to=2-4]
  \arrow[""{name=0, anchor=center, inner sep=0}, "Gf", "\shortmid"{marking}, from=1-2, to=1-3]
  \arrow["{\alpha_b^{-1}}", from=1-3, to=2-3]
  \arrow[""{name=1, anchor=center, inner sep=0}, "\shortmid"{marking}, Rightarrow, no head, from=1-3, to=1-4]
  \arrow["{\alpha_a^{-1}}\!"', from=1-2, to=2-2]
  \arrow[""{name=2, anchor=center, inner sep=0}, "Ff"', "\shortmid"{marking}, from=2-2, to=2-3]
  \arrow[""{name=3, anchor=center, inner sep=0}, "\shortmid"{marking}, Rightarrow, no head, from=2-1, to=2-2]
  \arrow[Rightarrow, no head, from=1-1, to=2-1]
  \arrow[""{name=4, anchor=center, inner sep=0}, "{\widehat{\alpha}_a}", "\shortmid"{marking}, from=1-1, to=1-2]
  \arrow[""{name=5, anchor=center, inner sep=0}, "{\widehat{\alpha}_b\;}"', "\shortmid"{marking}, from=2-3, to=2-4]
  \arrow["{\displaystyle\varepsilon'}"{description}, shift right=2, draw=none, from=4, to=3]
  \arrow["{\displaystyle\alpha_f^{-1}}"{description}, draw=none, from=0, to=2]
  \arrow["{\displaystyle\eta'}"{description}, shift left=2, draw=none, from=1, to=5]
  \end{tikzcd} \hspace*{1.2em}
  \ee
As a demonstration we use the graphical calculus to show one of the 
partial-invertibility properties. We have
  \be
{\tikzstyle{every picture}=[tikzfig]
  {\begin{tikzpicture}
	\begin{pgfonlayer}{nodelayer}
		\node [style=none] (113) at (-3, -3) {};
		\node [style=none] (114) at (-3, 3) {};
		\node [style=none] (115) at (-1, 1.5) {};
		\node [style=none] (116) at (-1.5, 1) {};
		\node [style=none] (119) at (-1, -1.5) {};
		\node [style=none] (120) at (-1.5, -1) {};
		\node [style=none] (153) at (0, 1.5) {};
		\node [style=none] (154) at (0, 3) {};
		\node [style=none] (155) at (0, -1.5) {};
		\node [style=none] (156) at (0, -3) {};
		\node [style=none] (127) at (-1, 1.5) {};
		\node [style=none] (128) at (-1.5, 1) {};
		\node [style=none] (129) at (0, 1.5) {};
		\node [style=none] (130) at (0, -1.5) {};
		\node [style=none] (131) at (-1, -1.5) {};
		\node [style=none] (132) at (-1.5, -1) {};
		\node [style=none] (133) at (0, 1.5) {};
		\node [style=none] (134) at (1, 1.5) {};
		\node [style=none] (135) at (1.5, 2) {};
		\node [style=none] (136) at (1.5, 3) {};
		\node [style=none] (137) at (0, 3) {};
		\node [style=none] (138) at (0, -1.5) {};
		\node [style=none] (139) at (1, -1.5) {};
		\node [style=none] (140) at (1.5, -2) {};
		\node [style=none] (141) at (1.5, -3) {};
		\node [style=none] (142) at (0, -3) {};
		\node [style=none] (143) at (3, 3) {};
		\node [style=none] (144) at (3, -3) {};
		\node [style=none] (145) at (0, 1.5) {};
		\node [style=none] (146) at (0, -1.5) {};
		\node [style=none] (147) at (1, -1.5) {};
		\node [style=none] (148) at (1.5, -2) {};
		\node [style=none] (149) at (1, 1.5) {};
		\node [style=none] (150) at (1.5, 2) {};
		\node [style=none] (151) at (1.5, -3) {};
		\node [style=none] (152) at (1.5, 3) {};
		\node [style=none] (99) at (-1, 1.5) {};
		\node [style=none] (100) at (-1.5, 1) {};
		\node [style=dot] (101) at (0, 1.5) {};
		\node [style=bdot] (102) at (0, -1.5) {};
		\node [style=none] (103) at (-1, -1.5) {};
		\node [style=none] (104) at (-1.5, -1) {};
		\node [style=none] (105) at (1, -1.5) {};
		\node [style=none] (106) at (1.5, -2) {};
		\node [style=none] (107) at (1, 1.5) {};
		\node [style=none] (108) at (1.5, 2) {};
		\node [style=none] (109) at (1.5, -3) {};
		\node [style=none] (110) at (1.5, 3) {};
		\node [style=none] (111) at (0, 3) {};
		\node [style=none] (112) at (0, -3) {};
		\node [style=none] (157) at (0, 3.5) {$Ff$};
		\node [style=none] (158) at (0, -3.5) {$Ff$};
		\node [style=none] (159) at (1.5, 3.5) {$\widehat{\alpha}_b$};
		\node [style=none] (160) at (1.5, -3.5) {$\widehat{\alpha}_b$};
	\end{pgfonlayer}
	\begin{pgfonlayer}{edgelayer}
		\draw [style=fi-pi] (119.center)
			 to [bend left=45, looseness=1.25] (120.center)
			 to (116.center)
			 to [bend left=45, looseness=1.25] (115.center)
			 to (153.center)
			 to (154.center)
			 to (114.center)
			 to (113.center)
			 to (156.center)
			 to (155.center)
			 to cycle;
		\draw [style=fi-gr] (130.center)
			 to (131.center)
			 to [bend left=45, looseness=1.25] (132.center)
			 to (128.center)
			 to [bend left=45, looseness=1.25] (127.center)
			 to (129.center)
			 to cycle;
		\draw [style=fi-bl] (136.center)
			 to (137.center)
			 to (133.center)
			 to (134.center)
			 to [bend right=45, looseness=1.25] (135.center)
			 to cycle;
		\draw [style=fi-bl] (141.center)
			 to (142.center)
			 to (138.center)
			 to (139.center)
			 to [bend left=45, looseness=1.25] (140.center)
			 to cycle;
		\draw [style=fi-pu] (145.center)
			 to (149.center)
			 to [bend right=45, looseness=1.25] (150.center)
			 to (152.center)
			 to (143.center)
			 to (144.center)
			 to (151.center)
			 to (148.center)
			 to [bend right=45, looseness=1.25] (147.center)
			 to (146.center)
			 to cycle;
		\draw [style=vt-bl, bend right=45, looseness=1.25] (99.center) to (100.center);
		\draw [style=vt-bl, bend left=45, looseness=1.25] (103.center) to (104.center);
		\draw [style=vt-bl, bend left=45, looseness=1.25] (105.center) to (106.center);
		\draw [style=vt-bl, bend right=45, looseness=1.25] (107.center) to (108.center);
		\draw [style=vt-bl] (101) to (99.center);
		\draw [style=vt-bl] (101) to (107.center);
		\draw [style=vt-bl] (102) to (103.center);
		\draw [style=vt-bl] (102) to (105.center);
		\draw [style=vt-bl] (106.center) to (109.center);
		\draw [style=vt-bl] (100.center) to (104.center);
		\draw [style=vt-bl] (110.center) to (108.center);
		\draw [style=vt-bl] (111.center) to (101);
		\draw [style=vt-bl] (101) to (102);
		\draw [style=vt-bl] (102) to (112.center);
	\end{pgfonlayer}
\end{tikzpicture}}
  {%
  }}%
~~=~~%
{\tikzstyle{every picture}=[tikzfig]
  {\begin{tikzpicture}
	\begin{pgfonlayer}{nodelayer}
		\node [style=none] (235) at (0, -1.5) {};
		\node [style=none] (236) at (1, -1.5) {};
		\node [style=none] (237) at (1.5, -3) {};
		\node [style=none] (238) at (0, -3) {};
		\node [style=none] (239) at (2, -1.5) {};
		\node [style=none] (240) at (2.5, -1.75) {};
		\node [style=none] (241) at (2, -2) {};
		\node [style=none] (242) at (1.5, -2.5) {};
		\node [style=none] (184) at (-1, 1.5) {};
		\node [style=none] (185) at (-1.5, 1) {};
		\node [style=none] (186) at (0, 1.5) {};
		\node [style=none] (187) at (0, -1.5) {};
		\node [style=none] (188) at (-1, -1.5) {};
		\node [style=none] (189) at (0, 3) {};
		\node [style=none] (190) at (0, -3) {};
		\node [style=none] (192) at (-3, 3) {};
		\node [style=none] (193) at (-3, -3) {};
		\node [style=none] (194) at (-2, -1.5) {};
		\node [style=none] (195) at (-2.5, -1.25) {};
		\node [style=none] (196) at (-2, -1) {};
		\node [style=none] (197) at (-1.5, -0.5) {};
		\node [style=none] (198) at (-1, 1.5) {};
		\node [style=none] (199) at (-1.5, 1) {};
		\node [style=none] (200) at (0, 1.5) {};
		\node [style=none] (201) at (0, -1.5) {};
		\node [style=none] (202) at (-1, -1.5) {};
		\node [style=none] (203) at (-2, -1.5) {};
		\node [style=none] (204) at (-2.5, -1.25) {};
		\node [style=none] (205) at (-2, -1) {};
		\node [style=none] (206) at (-1.5, -0.5) {};
		\node [style=none] (207) at (0, 1.5) {};
		\node [style=none] (208) at (1, 1.5) {};
		\node [style=none] (209) at (1.5, 2) {};
		\node [style=none] (210) at (1.5, 3) {};
		\node [style=none] (211) at (0, 3) {};
		\node [style=none] (217) at (0, 1.5) {};
		\node [style=none] (218) at (0, -1.5) {};
		\node [style=none] (219) at (1, -1.5) {};
		\node [style=none] (220) at (1, 1.5) {};
		\node [style=none] (221) at (1.5, 2) {};
		\node [style=none] (222) at (1.5, -3) {};
		\node [style=none] (223) at (1.5, 3) {};
		\node [style=none] (226) at (3, 3) {};
		\node [style=none] (227) at (3, -3) {};
		\node [style=none] (228) at (2, -1.5) {};
		\node [style=none] (229) at (2.5, -1.75) {};
		\node [style=none] (230) at (2, -2) {};
		\node [style=none] (231) at (1.5, -2.5) {};
		\node [style=none] (99) at (-1, 1.5) {};
		\node [style=none] (100) at (-1.5, 1) {};
		\node [style=dot] (101) at (0, 1.5) {};
		\node [style=bdot] (102) at (0, -1.5) {};
		\node [style=none] (103) at (-1, -1.5) {};
		\node [style=none] (105) at (1, -1.5) {};
		\node [style=none] (107) at (1, 1.5) {};
		\node [style=none] (108) at (1.5, 2) {};
		\node [style=none] (109) at (1.5, -3) {};
		\node [style=none] (110) at (1.5, 3) {};
		\node [style=none] (111) at (0, 3) {};
		\node [style=none] (112) at (0, -3) {};
		\node [style=none] (157) at (0, 3.5) {$Ff$};
		\node [style=none] (158) at (0, -3.5) {$Ff$};
		\node [style=none] (159) at (1.5, 3.5) {$\widehat{\alpha}_b$};
		\node [style=none] (160) at (1.5, -3.5) {$\widehat{\alpha}_b$};
		\node [style=none] (174) at (-2, -1.5) {};
		\node [style=none] (176) at (-2.5, -1.25) {};
		\node [style=none] (177) at (-2, -1) {};
		\node [style=none] (178) at (-1.5, -0.5) {};
		\node [style=none] (180) at (2, -1.5) {};
		\node [style=none] (181) at (2.5, -1.75) {};
		\node [style=none] (182) at (2, -2) {};
		\node [style=none] (183) at (1.5, -2.5) {};
	\end{pgfonlayer}
	\begin{pgfonlayer}{edgelayer}
		\draw [style=fi-bl] (237.center)
			 to (238.center)
			 to (235.center)
			 to (236.center)
			 to (239.center)
			 to [in=90, out=0, looseness=1.25] (240.center)
			 to [in=0, out=-90, looseness=1.25] (241.center)
			 to [bend right=45, looseness=1.25] (242.center)
			 to cycle;
		\draw [style=fi-pi] (197.center)
			 to [bend left=45, looseness=1.25] (196.center)
			 to [in=90, out=180, looseness=1.25] (195.center)
			 to [in=180, out=-90, looseness=1.25] (194.center)
			 to (188.center)
			 to (187.center)
			 to (190.center)
			 to (193.center)
			 to (192.center)
			 to (189.center)
			 to (186.center)
			 to (184.center)
			 to [bend right=45, looseness=1.25] (185.center)
			 to cycle;
		\draw [style=fi-gr] (201.center)
			 to (200.center)
			 to (198.center)
			 to [bend right=45, looseness=1.25] (199.center)
			 to (206.center)
			 to [bend left=45, looseness=1.25] (205.center)
			 to [in=90, out=180, looseness=1.25] (204.center)
			 to [in=180, out=-90, looseness=1.25] (203.center)
			 to (202.center)
			 to cycle;
		\draw [style=fi-bl] (211.center)
			 to (207.center)
			 to (208.center)
			 to [bend right=45, looseness=1.25] (209.center)
			 to (210.center)
			 to cycle;
		\draw [style=fi-pu] (227.center)
			 to (222.center)
			 to (231.center)
			 to [bend left=45, looseness=1.25] (230.center)
			 to [in=-90, out=0, looseness=1.25] (229.center)
			 to [in=0, out=90, looseness=1.25] (228.center)
			 to (219.center)
			 to (218.center)
			 to (217.center)
			 to (220.center)
			 to [bend right=45, looseness=1.25] (221.center)
			 to (223.center)
			 to (226.center)
			 to cycle;
		\draw [style=vt-bl, bend right=45, looseness=1.25] (99.center) to (100.center);
		\draw [style=vt-bl, bend right=45, looseness=1.25] (107.center) to (108.center);
		\draw [style=vt-bl] (101) to (99.center);
		\draw [style=vt-bl] (101) to (107.center);
		\draw [style=vt-bl] (102) to (103.center);
		\draw [style=vt-bl] (102) to (105.center);
		\draw [style=vt-bl] (110.center) to (108.center);
		\draw [style=vt-bl] (111.center) to (101);
		\draw [style=vt-bl] (101) to (102);
		\draw [style=vt-bl] (102) to (112.center);
		\draw [style=vt-bl, in=-90, out=180, looseness=1.25] (174.center) to (176.center);
		\draw [style=vt-bl, bend right=45, looseness=1.25] (177.center) to (178.center);
		\draw [style=vt-bl] (174.center) to (103.center);
		\draw [style=vt-bl] (100.center) to (178.center);
		\draw [style=vt-bl, in=90, out=180, looseness=1.25] (177.center) to (176.center);
		\draw [style=vt-bl, in=90, out=0, looseness=1.25] (180.center) to (181.center);
		\draw [style=vt-bl, bend right=45, looseness=1.25] (182.center) to (183.center);
		\draw [style=vt-bl, in=-90, out=0, looseness=1.25] (182.center) to (181.center);
		\draw [style=vt-bl] (105.center) to (180.center);
		\draw [style=vt-bl] (183.center) to (109.center);
	\end{pgfonlayer}
\end{tikzpicture}}
  {%
  }}%
~~=~~%
{\tikzstyle{every picture}=[tikzfig]
  {\begin{tikzpicture}
	\begin{pgfonlayer}{nodelayer}
		\node [style=none] (261) at (-1.25, 0.25) {};
		\node [style=none] (262) at (-1.25, -0.25) {};
		\node [style=none] (263) at (-1.25, 3) {};
		\node [style=none] (264) at (-1.25, -3) {};
		\node [style=none] (265) at (-1.75, -0.25) {};
		\node [style=none] (266) at (-2.25, 0) {};
		\node [style=none] (267) at (-1.75, 0.25) {};
		\node [style=none] (268) at (-3, -3) {};
		\node [style=none] (269) at (-3, 3) {};
		\node [style=none] (270) at (-1.25, 0.25) {};
		\node [style=none] (271) at (-1.25, -0.25) {};
		\node [style=none] (272) at (1.25, 0.75) {};
		\node [style=none] (273) at (1, -3) {};
		\node [style=none] (274) at (1.25, 3) {};
		\node [style=none] (275) at (-1.25, 3) {};
		\node [style=none] (276) at (-1.25, -3) {};
		\node [style=none] (277) at (1.5, -0.25) {};
		\node [style=none] (278) at (2, -0.5) {};
		\node [style=none] (279) at (1.5, -0.75) {};
		\node [style=none] (280) at (1, -1.25) {};
		\node [style=none] (281) at (-0.75, -0.25) {};
		\node [style=none] (282) at (-0.25, 0) {};
		\node [style=none] (283) at (-0.75, 0.25) {};
		\node [style=none] (284) at (0.75, -0.25) {};
		\node [style=none] (285) at (0.25, 0) {};
		\node [style=none] (286) at (0.75, 0.25) {};
		\node [style=none] (287) at (1.25, 0.75) {};
		\node [style=none] (288) at (1, -3) {};
		\node [style=none] (289) at (1.25, 3) {};
		\node [style=none] (290) at (1.5, -0.25) {};
		\node [style=none] (291) at (2, -0.5) {};
		\node [style=none] (292) at (1.5, -0.75) {};
		\node [style=none] (293) at (1, -1.25) {};
		\node [style=none] (294) at (3, 3) {};
		\node [style=none] (295) at (3, -3) {};
		\node [style=none] (296) at (0.75, -0.25) {};
		\node [style=none] (297) at (0.25, 0) {};
		\node [style=none] (298) at (0.75, 0.25) {};
		\node [style=none] (299) at (-1.25, 0.25) {};
		\node [style=none] (300) at (-1.25, -0.25) {};
		\node [style=none] (301) at (-1.75, -0.25) {};
		\node [style=none] (302) at (-2.25, 0) {};
		\node [style=none] (303) at (-1.75, 0.25) {};
		\node [style=none] (304) at (-1.25, 0.25) {};
		\node [style=none] (305) at (-1.25, -0.25) {};
		\node [style=none] (306) at (-0.75, -0.25) {};
		\node [style=none] (307) at (-0.25, 0) {};
		\node [style=none] (308) at (-0.75, 0.25) {};
		\node [style=dot] (101) at (-1.25, 0.25) {};
		\node [style=bdot] (102) at (-1.25, -0.25) {};
		\node [style=none] (108) at (1.25, 0.75) {};
		\node [style=none] (109) at (1, -3) {};
		\node [style=none] (110) at (1.25, 3) {};
		\node [style=none] (111) at (-1.25, 3) {};
		\node [style=none] (112) at (-1.25, -3) {};
		\node [style=none] (157) at (-1.25, 3.5) {$Ff$};
		\node [style=none] (158) at (-1.25, -3.5) {$Ff$};
		\node [style=none] (159) at (1.25, 3.5) {$\widehat{\alpha}_b$};
		\node [style=none] (160) at (1, -3.5) {$\widehat{\alpha}_b$};
		\node [style=none] (174) at (-1.75, -0.25) {};
		\node [style=none] (176) at (-2.25, 0) {};
		\node [style=none] (177) at (-1.75, 0.25) {};
		\node [style=none] (180) at (1.5, -0.25) {};
		\node [style=none] (181) at (2, -0.5) {};
		\node [style=none] (182) at (1.5, -0.75) {};
		\node [style=none] (183) at (1, -1.25) {};
		\node [style=none] (247) at (-0.75, -0.25) {};
		\node [style=none] (248) at (-0.25, 0) {};
		\node [style=none] (249) at (-0.75, 0.25) {};
		\node [style=none] (250) at (0.75, -0.25) {};
		\node [style=none] (251) at (0.25, 0) {};
		\node [style=none] (252) at (0.75, 0.25) {};
	\end{pgfonlayer}
	\begin{pgfonlayer}{edgelayer}
		\draw [style=fi-pi] (266.center)
			 to [in=180, out=-90, looseness=1.25] (265.center)
			 to (262.center)
			 to (264.center)
			 to (268.center)
			 to (269.center)
			 to (263.center)
			 to (261.center)
			 to (267.center)
			 to [in=90, out=180, looseness=1.25] cycle;
		\draw [style=fi-bl] (286.center)
			 to [bend right=45, looseness=1.25] (272.center)
			 to (274.center)
			 to (275.center)
			 to (270.center)
			 to (283.center)
			 to [in=90, out=0, looseness=1.25] (282.center)
			 to [in=0, out=-90, looseness=1.25] (281.center)
			 to (271.center)
			 to (276.center)
			 to (273.center)
			 to (280.center)
			 to [bend left=45, looseness=1.25] (279.center)
			 to [in=-90, out=0, looseness=1.25] (278.center)
			 to [in=0, out=90, looseness=1.25] (277.center)
			 to (284.center)
			 to [in=-90, out=180, looseness=1.25] (285.center)
			 to [in=180, out=90, looseness=1.25] cycle;
		\draw [style=fi-pu] (294.center)
			 to (289.center)
			 to (287.center)
			 to [bend left=45, looseness=1.25] (298.center)
			 to [in=90, out=180, looseness=1.25] (297.center)
			 to [in=180, out=-90, looseness=1.25] (296.center)
			 to (290.center)
			 to [in=90, out=0, looseness=1.25] (291.center)
			 to [in=0, out=-90, looseness=1.25] (292.center)
			 to [bend right=45, looseness=1.25] (293.center)
			 to (288.center)
			 to (295.center)
			 to cycle;
		\draw [style=fi-gr] (302.center)
			 to [in=180, out=-90, looseness=1.25] (301.center)
			 to (300.center)
			 to (299.center)
			 to (303.center)
			 to [in=90, out=180, looseness=1.25] cycle;
		\draw [style=fi-pu] (307.center)
			 to [in=0, out=-90, looseness=1.25] (306.center)
			 to (305.center)
			 to (304.center)
			 to (308.center)
			 to [in=90, out=0, looseness=1.25] cycle;
		\draw [style=vt-bl] (110.center) to (108.center);
		\draw [style=vt-bl] (111.center) to (101);
		\draw [style=vt-bl] (101) to (102);
		\draw [style=vt-bl] (102) to (112.center);
		\draw [style=vt-bl, in=-90, out=180, looseness=1.25] (174.center) to (176.center);
		\draw [style=vt-bl, in=90, out=180, looseness=1.25] (177.center) to (176.center);
		\draw [style=vt-bl, in=90, out=0, looseness=1.25] (180.center) to (181.center);
		\draw [style=vt-bl, bend right=45, looseness=1.25] (182.center) to (183.center);
		\draw [style=vt-bl, in=-90, out=0, looseness=1.25] (182.center) to (181.center);
		\draw [style=vt-bl] (183.center) to (109.center);
		\draw [style=vt-bl] (177.center) to (101);
		\draw [style=vt-bl] (174.center) to (102);
		\draw [style=vt-bl, in=-90, out=0, looseness=1.25] (247.center) to (248.center);
		\draw [style=vt-bl, in=90, out=0, looseness=1.25] (249.center) to (248.center);
		\draw [style=vt-bl, in=-90, out=180, looseness=1.25] (250.center) to (251.center);
		\draw [style=vt-bl, in=90, out=180, looseness=1.25] (252.center) to (251.center);
		\draw [style=vt-bl] (101) to (249.center);
		\draw [style=vt-bl] (102) to (247.center);
		\draw [style=vt-bl, bend right=45, looseness=1.25] (252.center) to (108.center);
		\draw [style=vt-bl] (250.center) to (180.center);
	\end{pgfonlayer}
\end{tikzpicture}}
  {%
  }}%
~~=~~%
{\tikzstyle{every picture}=[tikzfig]
  {\begin{tikzpicture}
	\begin{pgfonlayer}{nodelayer}
		\node [style=none] (305) at (-3, 3) {};
		\node [style=none] (306) at (-3, -3) {};
		\node [style=none] (307) at (-1, 3) {};
		\node [style=none] (308) at (-1, -3) {};
		\node [style=none] (309) at (-1, 3) {};
		\node [style=none] (310) at (-1, -3) {};
		\node [style=none] (311) at (1, 3) {};
		\node [style=none] (312) at (1, -3) {};
		\node [style=none] (313) at (1, 3) {};
		\node [style=none] (314) at (1, -3) {};
		\node [style=none] (315) at (3, 3) {};
		\node [style=none] (316) at (3, -3) {};
		\node [style=none] (317) at (-1, 3.5) {$Ff$};
		\node [style=none] (318) at (1, 3.5) {$\widehat{\alpha}_b$};
		\node [style=none] (319) at (-1, -3.5) {$Ff$};
		\node [style=none] (320) at (1, -3.5) {$\widehat{\alpha}_b$};
	\end{pgfonlayer}
	\begin{pgfonlayer}{edgelayer}
		\draw [style=fi-pi] (305.center)
			 to (307.center)
			 to (308.center)
			 to (306.center)
			 to cycle;
		\draw [style=fi-bl] (309.center)
			 to (311.center)
			 to (312.center)
			 to (310.center)
			 to cycle;
		\draw [style=fi-pu] (313.center)
			 to (315.center)
			 to (316.center)
			 to (314.center)
			 to cycle;
		\draw [style=vt-bl] (309.center) to (310.center);
		\draw [style=vt-bl] (313.center) to (314.center);
	\end{pgfonlayer}
\end{tikzpicture}}
  {%
  }}%
 \text{~}
  \ee
Here the first equality uses \eqref{eq:pic:conjoint(co)unit}, while the second and
third equalities use the snake relations \eqref{eq:pic:companion2+3} for companions
and the premise that $\alpha$ is a vertical equivalence, whose component 2-cells
satisfy \eqref{eq:pic:verticalequiv}.  Alternatively, this result can be obtained by
realizing that $\widehat{\alpha}$ is an internal adjoint equivalence in the bicategory
$\cb\mathrm{icat}_{\mathrm{oplax}}(\ch(\da),\ch(\db))$ of pseudofunctors, oplax
natural transformations and modifications, and is therefore by doctrinal adjunction
\cite{kelly7} automatically pseudonatural.\,%
 \footnote{~Doctrinal adjunctions for a 2-monad $T$ are in fact \emph{conjunctions}
 in the double category $T\text{-}\da\mathrm{lg}$ of algebras, whose
 vertical 1-morphisms are oplax $T$-morphisms and horizontal 1-morphisms
 are lax $T$-morphisms, see \Cite{Example\,5.4}{shul6}.}
\end{proof}

Invoking also Theorem 5.12 of \cite{wesSh}, we then end up with the following result:

\begin{thm} \label{thm:double2bi}
Let $F,G\colon\da \,{\rightrightarrows}\, \db$ be symmetric monoidal double functors
and $\alpha\colon F \,{\Rightarrow}\, G$ be a monoidal vertical equivalence. If both
$\da$ and $\db$ are framed bicategories and every 2-cell component of $\alpha$ 
is weakly invertible, then we have:
 \\[3pt]
(1) $\ch(\da)$ and $\ch(\db)$ are symmetric monoidal bicategories.
 \\[3pt]
(2) $F$ and $G$ lift to symmetric monoidal pseudofunctors 
$\ch(\da) \,{\rightrightarrows}\, \ch(\db)$.
 \\[3pt]
(3) $\widehat{\alpha}$ is a monoidal pseudonatural equivalence.
\end{thm}

We are now in a position to present the main result of this section, 
which follows as a Corollary of Theorem \ref{thm:double2bi}

\begin{thm} \label{thm:eqvmf}
For any \ko-linear pivotal tensor category $\cc$, the monoidal vertical
transformation $\ducorc\colon\dSNfrc\,{\xRightarrow{\phantom{\,\simeq\,}}}\,\dSNc$
is an equivalence and lifts to a monoidal pseudonatural equivalence
  \be
  \widehat{\ducorc}\Colon\SNfrc \xRightarrow{~\simeq\,} \SNc \,,
  \label{eq:eqvmf}
  \ee
i.e.\ to an equivalence between modular functors.
\end{thm}

\begin{proof}
In view of Lemma \ref{lem:equiv-cs-weaklyinv}, Corollary \ref{cor:UCor-weaklyinv}
shows that the universal correlator $\ducorc$ satisfies the condition imposed in 
Proposition \ref{prop:aboutstrongcomps}. $\ducorc$
is therefore a monoidal vertical equivalence. Thus by Theorem \ref{thm:double2bi} it 
lifts to a monoidal pseudonatural equivalence.
\end{proof}

\begin{rem}
An (anomaly free) open-closed modular functor is equivalently a 
\emph{modular algebra} over a modular operad of open-closed 
surfaces in the bicategory $\bprof$ (see \Cite{Def.\,7.16}{muWo2} for the closed case,
for which the target bicategory differs slightly from ours). To ensure that the 
(a priori) oplax natural transformation \eqref{eq:eqvmf} is monoidal, we invoked
Theorem 5.12 of \cite{wesSh}, which crucially relies on the condition that the 
vertical transformation $\ducorc$ has horizontally strong companions in the sense of 
Definition \ref{lem:equiv-cs-weaklyinv}. In order for this argument to work, by
definition the oplax natural transformation $\widehat{\ducorc}$ needs to be
pseudonatural. On the other hand, a monoidal pseudonatural transformation between
modular functors corresponds to a morphism of the corresponding modular algebras,
and according to Proposition 2.18 of \cite{muWo2} such a morphism is necessarily an 
equivalence. This furnishes a sanity check on the validity of Theorem \ref{thm:eqvmf}.
\end{rem}

   \newpage
%%%%%%%%%%%%%%%%%%%%%%%%%%%%%%%%%%%%%%%%%%%%%%%%%%%%%%%%%%%%%%%%%%%%%%%%

\bibliographystyle{amsalpha}
%% \phantomsection\addcontentsline{toc}{section}{\refname}\bibliography{Library}

 \newcommand\wb{\,\linebreak[0]} \def\wB {$\,$\wb}

 \newcommand\Bi[2]    {\bibitem[#2]{#1}}
 \newcommand\BOOK[4]  {{\em #1\/} ({#2}, {#3} {#4})}
 \newcommand\inBO[9]  {{\em #9}, in:\ {\em #1}, {#2}\ ({#3}, {#4} {#5}), p.\ {#6--#7} {\tt [#8]}}
 \newcommand\J[7]     {{\em #7}, {#1} {#2} ({#3}) {#4--#5} {{\tt [#6]}}}
 \newcommand\JO[6]    {{\em #6}, {#1} {#2} ({#3}) {#4--#5} }
 \newcommand\Mast[2]  {{\em #2\/}, Master thesis (#1)}
 \newcommand\PhD[2]   {{\em #2\/}, Ph.D.\ thesis (#1)}
 \newcommand\Prep[2]  {{\em #2}, preprint {\tt #1}}
 \newcommand\uPrep[2] {{\em #2}, unpublished preprint {\tt #1}}
 
   \def\adma  {Adv.\wb Math.}
   \def\apcs  {Applied\wB Cate\-go\-rical\wB Struc\-tures}
   \def\bumi  {Boll.\wB Unione\wB Mat.\wb Ital.}
   \def\coma  {Con\-temp.\wb Math.}
   \def\comp  {Com\-mun.\wb Math.\wb Phys.}
   \def\ctgc  {Cah.\wb Topol.\wb G\'eom.\wb Diff\'er.\wB Ca\-t\'e\=goriques}
   \def\jpaa  {J.\wB Pure\wB Appl.\wb Alg.}
   \def\kyjm  {Ky\-o\-to J.\ Math.}
   \def\nyjm  {New\wB York\wB J.\wb Math.}
   \def\nupb  {Nucl.\wb Phys.\ B}
   \def\phrb  {Phys.\wb Rev.\ B}
   \def\qjmo  {Quart.\wb J.\wb Math.\wB Oxford}
   \def\quto  {Quantum Topology}
   \def\rtac  {Reprints in Theo\-ry\wB and\wB Appl.\wb Cat.}
   \def\slnm  {Sprin\-ger\wB Lecture\wB Notes\wB in\wB Mathematics}
   \def\taac  {Theory\wB and\wB Appl.\wb Cat.}

\end{document}